%% file: nonsmooth-bilevel.tex
\documentclass[a4paper,english]{jnsao} 
\usepackage[utf8]{inputenc}
\usepackage[T1]{fontenc}
\usepackage{csquotes}
\usepackage{babel}
\usepackage{enumitem}
\usepackage{comment}
\usepackage{mathtools}
\usepackage[indLines=false,noEnd=false]{algpseudocodex}
\usepackage[nameinlink,capitalize]{cleveref}
\usepackage{todonotes}
\usepackage{float}
\usepackage{tikz, pgfplots}
\usepgfplotslibrary{fillbetween}
\usepgfplotslibrary{colorbrewer}
\usepackage[createShortEnv,commandRef = cref,conf={
    restate,
    text link section,
    text proof = {Proof},
}]{proof-at-the-end}
\newif\ifAtEnd\AtEndfalse

\input{symbols}
\input{symbols-texonly}

\usepackage{nicematrix}
\usepackage{booktabs}
\usepackage{mdframed}

\mdfdefinestyle{examplebg}{
    backgroundcolor=structure!7,%
    hidealllines=true,%
    innertopmargin=-.3em,%
    innerbottommargin=.7em,%
    innerleftmargin=.7em,%
    innerrightmargin=.7em,%
}

\surroundwithmdframed[style=examplebg]{example}

\manuscriptcopyright{}
\manuscriptlicense{}

\floatstyle{ruled}
\floatname{algorithm}{Algorithm}
\newfloat{algorithm}{tbp!}{loa}
\crefname{algorithm}{Algorithm}{Algorithms}
\numberwithin{algorithm}{section}

\newcommand{\obtodo}[2][]{%
}

\newenvironment{limiting}{}{}

\author{
    Ensio Suonperä\thanks{%
        Department of Mathematics and Statistics, University of Helsinki, Finland.
        \email{ensio.suonpera@helsinki.fi},
        \orcid{0000-0002-2111-0239}
    }
    \and
    Tuomo Valkonen\thanks{%
        MODEMAT Research Center in Mathematical Modeling and Optimization, Quito, Ecuador
        \emph{and}
        Department of Mathematics and Statistics, University of Helsinki, Finland,
        \email{tuomo.valkonen@iki.fi},
        \orcid{0000-0001-6683-3572}
    }
}
\shortauthor{E.~Suonperä, T.~Valkonen}

\title{Single-loop approaches to nonsmooth bilevel optimisation}

\date{2026-06-17}

\acknowledgements{%
    This research has been supported by the Academy of Finland grants 314701, 320022, and 345486, and the Vilho, Yrjö and Kalle Väisälä Foundation grant for PhD studies.
}

\begin{document}

\maketitle

\begin{abstract}
    We study bilevel optimisation problems in which the inner problem is represented as a set-valued, parametric constraint. We develop relevant optimistic and pessimistic calculus rules, derive corresponding optimality conditions, and formulate nonsmooth adjoint inclusions based on both the Fréchet and limiting coderivatives. Founded on these results, we propose a single-loop algorithm that accommodates a wide range of inner and adjoint steps, including those of primal-dual methods. We prove its convergence.
    Numerical experiments on total variation regularised inverse problems demonstrate the practicality of the approach.
\end{abstract}

\section{Introduction}
\label{sec:intro}

We consider bilevel optimisation problems of the form
\begin{equation}
    \label{eq:intro:bilevel-problem}
    \min_{u, x} J(u) + R(x)\quad\text{subject to}\quad 0 \in G(u, x)
\end{equation}
in normed spaces $U$ and $X.$
The set-valued function $G$ typically encodes the optimality conditions for an inner problem; for example, for $g$ convex in the first argument, we can take $G=\subdiff g(u; x)$ corresponding to the parametric inner problem $\min_u g(u; x)$.
Very few works, so far, allow $g$ to be nonsmooth; $G$ to be set-valued.
Most bilevel optimisation works impose strong differentiability, smoothness, and second-order growth assumptions on $G$.

Moreover, numerically, bilevel optimisation can be very expensive.
At the outset, it requires solving the inner problem several times to find a solution to the problem \cref{eq:intro:bilevel-problem}. Therefore, in recent years, there has been a growing interest in “single-loop” methods \cite{suonpera2022bilevel,suonpera2024general,dizonvalkonen2024tracking,online-tracking,chen2021single,hong2020two,li2022fully,yang2021provably,dagreou2022framework,ji2021bilevel,ji2022lower,ghadimi2018approximation}, as well as methods that solve the subproblems to a set accuracy \cite{kwon2023fully,salehi2024maid,ehrhardt2021inexact}.
Single-loop methods take only a single or a fixed number of steps of a conventional optimisation method towards a solution of the inner objective on each step of an algorithm to solve the outer problem.
Most works take $G(u,x) = \grad_u g(u,x)$ for a smooth $g$, and use (stochastic) gradient descent as the inner optimisation method.
In \cite{suonpera2024general}, we took $G$ as primal-dual optimality conditions of $\min_u f(u; x)+ g(Ku; x)$, which let us use the Primal-Dual Proximal Splitting (PDPS) of Chambolle and Pock \cite{chambolle2010first} for the inner steps.
We also allowed using the steps of standard algorithms for the approximate solution of an adjoint problem.
The latter is used to compute $[J \circ S]'(x)$, where the solution mapping $S$ is defined through the satisfaction of the inner problem, $0=G(S(x), x)$.
However, we still required $G$ to be smooth.

\emph{Our present results, based on coderivative calculus and nonsmooth adjoint inclusions, allow $G$ and $S$ to be set-valued; for $g$ to be nonsmooth while also weakening second-order growth assumptions to metric subregularity.}
The theory has connections to the recent \cite{villacis2025variational}, where, for inner problems of the form $\min_u f(u) + g(u;x)$ with strongly convex and smooth $f$ and convex $g$, the authors take $G(u, x) = \inv\gamma(u - \prox_{\gamma g(\freevar,x)}(u - \tau \grad f(u)))$ as a proximal reformulation of the first-order necessary and sufficient optimality conditions.
Then they employ limiting coderivatives of $G$.
Such a proximal reformulation is commonly used with semismooth Newton methods, see, e.g., \cite{clason2020introduction}, as well as in the SparseHO of \cite{bertrand2022implicit} for bilevel problems.
Unlike \cite{villacis2025variational}, the latter, however, requires significant differentiability from the proximal reformulation.
Our theory also applies to $G$ as above while also allowing it to be more general and set-valued.

Algorithmically, both  \cite{villacis2025variational,bertrand2022implicit} are based on exact or near-exact differentials or coderivatives of the inner solution mapping in an outer forward-backward method or a nonsmooth trust region method.
We propose a highly-efficient single-loop method instead, based on the previous, still smooth, works \cite{suonpera2022bilevel,suonpera2024general,dizonvalkonen2024tracking}.
Our approach can work both with the Fréchet and limiting coderivatives of $G$.
Moreover, our numerical examples would not be practical to solve with forward-backward splitting as the inner algorithm, as the involved proximal map would not have closed-form solution.
Our approach allows the use of (the steps of) more appropriate primal-dual methods.

There are a few relevant recent more analytical works -- that do not represent algorithms -- for nonsmooth bilevel optimisation based on the solution mapping.
In particular, \cite{bolte2021nonsmooth,bolte2024differentiating} treat nonsmooth inner problems through proximal reformulations.
They rely on Rademacher's theorem to obtain differentiability almost everywhere in $\R^n$.
However, those tools are not available in our infinite dimensional functional analytic setting.
On the other hand, \cite{hang2024chain,hang2024role} show that metric regularity of an inner problem with linear dependence on the outer parameters (i.e., tilt stability), is, often, equivalent to the Lipschitz continuity of the solution mapping.

Methods of more second-order flavour have also been proposed: \cite{christof2017non} study bundle-based trust region methods in special cases, while \cite{jolaoso2024fresh} propose a Levenberg--Marquardt method based on Newton-derivatives. The latter is based on finding roots of the optimality conditions of an alternative constrained single-level reformulation, known as the value function reformulation \cite{outrata1990numerical}.
Although such a simple reformulation may sound attractive, not requiring the solution mapping, it still depends on knowing the value function $\phi(x) \defeq \min_u g(u, x)$, which is no easier to compute than the solution mapping $S(u)= \argmin_u g(u, x)$.
Here we recall that $g$ is the inner objective when $G$ arises as $\subdiff g(u; x)$.
For smooth problems, similarly to our approach \cite{suonpera2022bilevel,suonpera2024general} based on the solution mapping, \cite{kwon2023fully} avoids knowing $\phi$ through solution quality control.
Moreover, for nonsmooth problems, the value function reformulation poses several challenges with the satisfaction of constraint qualifications.
This is also true of the conventional approach of reformulating \eqref{eq:intro:bilevel-problem} as an MPCC (Mathematical Program with Complementarity Constraints) \cite{luo1996mpec,mangasarian2001complementarity}, and finding roots of the full-system KKT conditions.
That reformulation is also no longer completely equivalent to the original problem, while for convex $g$, \eqref{eq:intro:bilevel-problem} remains equivalent.
We do not, however, entirely avoid difficulties with coderivative calculus closely related to constraint qualifications.
Alternative duality-based reformulations for nonsmooth problems have recently been studied in, e.g., \cite{dempe2025duality}.

We start our journey by developing optimistic and pessimistic calculus rules for the Fréchet subdifferential and coderivative in \cref{sec:calculus}: it turns out that the directions of inclusions the calculus correspond to optimistic and pessimistic solutions of bilevel problems.
We then use this calculus in \cref{sec:oc} to derive optimality conditions for nonsmooth bilevel optimisation problems of form \cref{eq:intro:bilevel-problem}. Based on these optimality conditions, we devote \cref{sec:algorithm} to discussing a general single-loop algorithms for \eqref{eq:intro:bilevel-problem}.
We base this on the general “tracking theory” of differential estimation from \cite{dizonvalkonen2024tracking}. \cref{sec:tv} focuses providing an actual algorithm when the inner problem is total variation regularised inverse problem. We describe our numerical experiments and their results in \cref{sec:numerical}.

\subsection*{Notation and basic concepts}

We write $\dualprod{x^*}{x}_{X^*,X} \defeq x^*(x)$, or, for short, $\dualprod{x^*}{x}$ for the dual pairing in a normed space $X$ and its dual $X^*$.
The notation $\iprod{x}{y}$ is reserved for inner products in Hilbert spaces, where we also write $\mathscr{R}x^* \in X$ for the Riesz representation of $x^* \in X^*$.
The notation $\linear(X; Y)$ stands for the space of bounded linear operators between $X$ and $Y$.
The identity operator is written $\Id$.

We now introduce several concepts of set-valued analysis; for further details we refer to \cite{clason2020introduction}.
Let $F: X \to \extR \defeq (-\infty, \infty]$ in a normed space $X$.
We write $\Dom F \defeq \{ x \in X \mid F(x) < \infty\}$ for the effective domain and say that $F$ is proper if $\Dom F \ne \emptyset$.
The Fenchel conjugate of $F$ is written $F^*$, and the proximal map by $\prox_F$.

The convex subdifferential of $F$ at $x$ reads
\begin{equation}
    \label{eq:intro:convex-subdiff}
    \subdiff F(x) \defeq \{ x^* \in X^* \mid F(\tilde x) - F(x) \ge \dualprod{x^*}{\tilde x - x} \text{ for all } \tilde x\}.
\end{equation}
The $\epsilon$-subdifferential, for a $\epsilon \ge 0$, has the definition
\begin{equation}
    \label{eq:intro:frechet-subdiff}
    \subdiff_\epsilon F(x) \defeq \left\{
    x^* \in X^*
    \,\middle|\,
    \liminf_{\tilde x \to x}
    \frac{F(\tilde x) - F(x) - \dualprod{x^*}{\tilde x - x}}{\norm{\tilde x - x}} \ge -\epsilon
    \right\}.
\end{equation}
The Fréchet subdifferential is $\frechetSubdiff F(x)=\subdiff_0 F(x)$.
The limiting subdifferential is then
\[
    \limitingSubdiff F(x) \defeq \{ x^* \in X^* \mid \text{there exist } \tilde x \to x,\, \epsilon \downto 0,\text{ and } \tilde x^* \in \subdiff_\epsilon F(\tilde x) \text{ with }
    \tilde x^* \weaktostar x^* \text{ and } F(\tilde x) \to F(x) \}.
\]
If $X$ is reflexive, we can simplify
\[
    \limitingSubdiff F(x) = \{ x^* \in X^* \mid \text{there exist} \tilde x \weaktostar x \text{ and } \tilde x^* \in \subdiff F(\tilde x) \text{ with } \tilde x^* \weaktostar x^*\}.
\]
For convex $F$, $\frechetSubdiff F=\limitingSubdiff F = \subdiff F$.
For continuously differentiable $F$, $\frechetSubdiff F(x)=\limitingSubdiff F(x) = \{F'(x)\}$.

For a set-valued $S: X \setto Y$, where also $Y$ is a normed space, the $\epsilon$-Fréchet coderivative \emph{at $x$ for $y$}, denoted by $\frechetCod_\epsilon F(x|y): Y^* \setto X^*$, is defined through
\begin{equation}
    \label{eq:intro:frechet-coderivative}
    \frechetCod_\epsilon F(x|y)(y^*) \defeq \left\{
    x^* \in X^*
    \,\middle|\,
    \limsup_{\tilde x \to x,\, F(\tilde x) \ni \tilde y \to y}
    \frac{\dualprod{x^*}{\tilde x - x} - \dualprod{y^*}{\tilde y - y}}{\norm{\tilde x - x} + \norm{\tilde y - y}} \le \epsilon
    \right\}.
\end{equation}
That is, $(x^*, -y^*) \in \frechetNormal_{\graph F}^\epsilon(x, y)$, where $\frechetNormal_{\graph F}^\epsilon$ is the ($\epsilon$-)Fréchet normal cone to $\graph F$.
We call $\frechetCod F(x|y) \defeq \frechetCod_0 F(x|y)$ the Fréchet coderivative at $x$ for $y$.
We do not impose $y \in F(x)$: it is possible that $\graph \frechetCod F(x|y) \ne \emptyset$ for $y \in \closure F(x)$.

The (normal) limiting/Mordukhovich coderivative is then defined as
\[
    D^* F(x|y)(y^*) \defeq \weakstarlimsup_{\substack{(\tilde x, \tilde y) \to (x, y),
            \\
            \tilde y^* \weaktostar y^*,\,\epsilon \downto 0}}
    \frechetCod_\epsilon F(\tilde x|\tilde y)(\tilde y^*).
\]
It also has the smaller “mixed” variant
\[
    D^*_M F(x|y)(y^*) \defeq \weakstarlimsup_{\substack{(\tilde x, \tilde y) \to (x, y),
            \\
            \tilde y^* \to y^*,\,\epsilon \downto 0}}
    \frechetCod_\epsilon F(\tilde x|\tilde y)(\tilde y^*).
\]
We call $F$ $N$-regular at $(x, y)$ if $D_M F(x|y)=\frechetCod F(x|y)$.

We recall that if $F$ is Gateaux-differentiable, and the differential $DF: X \to X^*$ is Lipschitz continuous with factor $L$, then the \term{descent inequality} holds (see, e.g., \cite[Lemma 7.1]{clason2020introduction})
\[
    F(y) \le F(x) + \dualprod{DF(x)}{y-x} + \frac{L}{2}\norm{x-y}^2.
\]
If $F$ is convex and $X$ is reflexive, the two properties are equivalent.
If $F$ is Fréchet differentiable, we write $F'$ instead of $DF$ for the differential.
We write partial derivatives (generally Fréchet) of $G: X \times Y \to Z$, $G: (x, y) \to z$ as $G^{(x)}(x, y) \in \linear(X; Z)$ and $G^{(y)}(x, y) \in \linear(Y; Z)$.

For a set $B$ and a point $a$ in a metric space $(X, d)$, we define
\[
    \dist(B, a) \defeq \dist(a,B) \defeq \inf_{b \in B} d(a, b).
\]

When $M \in \linear(X; X^*)$ is self-adjoint and positive (semi-)definite, we write $\norm{x}_M \defeq \sqrt{\dualprod{Mx}{x}}$ as well as $\dist_M(A, u)=\inf_{\tilde u \in A} \norm{\tilde u -u}_M$. We also write $\B_M(u, \delta) = \{ \tilde u \in U \mid \norm{u}_M \le \delta \}$.

Finally, we write $\proj_A(x) \defeq \{z \in A \mid \dist(A, x)=d(z, x)\}$.

\section{Calculus for the Fréchet subdifferential and coderivative}
\label{sec:calculus}

In this section, we provide chain and sum rules for the Fréchet subdifferential and coderivative, for which results are generally much less developed than for the limiting coderivative.
We refer to \cite{rockafellar-wets-va,clason2020introduction,mordukhovich2006variational} for such classical results.
In case of the sum rule, we exploit convexity of the second component.
In case of the chain rule, we obtain otherwise relaxed assumptions by considering optimistic and pessimistic variants of scalar functions.
Towards this end, for a function $F: X \setto \extR$, we define the optimistic and pessimistic selections
\[
    \infmap{F}(x) \defeq \inf\{ v \mid v \in F(x) \}
    \quad\text{and}\quad
    \supmap{F}(x) \defeq \sup\{ v \mid v \in F(x) \}.
\]
For the composition $J \circ S$ with $J: U \to \extR$ and $S: X \setto U$, we also define the optimistic and pessimistic inner maps
\[
    \Sopt(\tilde x) \defeq \{ u \in S(\tilde x) \mid J(u) = \infmap{J \circ S}(\tilde x) \}
\]
and
\[
    \Spess(\tilde x) \defeq \{ u \in S(\tilde x) \mid J(u) = \supmap{J \circ S}(\tilde x) \}.
\]

We start by introducing the regularity properties that we will work with.

\subsection{Regularity properties}
\label{sec:regularity}

We call $F: X \setto Y$ between normed spaces $X$ and $Y$ \term{inner Lipschitz} \emph{at} $x$ relative to a set $A \subset Y$ if for \emph{every} $y \in A$ there exist $\lambda_y \ge 0$ and $\rho_y>0$ (which may depend on the point $y$) such that
\[
    \dist(F(\tilde x), y) \le \lambda_y \norm{x-\tilde x}
    \quad\text{for all}\quad \tilde x \in B(x, \rho_y) \isect \Dom F.
\]
We call $F$ \term{inner semi-Lipschitz} \emph{at} $x$ relative to a set $A \subset Y$ if, given a sequence $\Dom F \ni \this x \to x$, there exists a subsequence, unrelabelled, and \emph{some} $y \in A$ and $\lambda_y \ge 0$, such that
\[
    \dist(F(\this x), y) \le \lambda_y \norm{x-\this x}
    \quad\text{for all}\quad k \in \N.
\]
If $A=F(x)$, we drop the specification “relative to $A$”.

\begin{remark}
    Aside from \cref{lemma:calculus:section-chain} in the next subsection, where it is possible that $A \isect \closure F(x) \setminus F(x) \ne \emptyset$, we will take $A \subset F(x)$.
\end{remark}

Clearly an inner Lipschitz mapping is also inner semi-Lipschitz relative to any $A \subset F(x)$, and a single-valued mapping that is Lipschitz \emph{at} the point $x$, is inner Lipschitz at $x$.

\begin{remark}[Outer Lipschitz, linear recession and inner calmness properties]
    The inner Lipschitz properties bear superficial resemblance to outer Lipschitz and linear recession properties, however, are distinct; see \cite[§3.4]{dontchev2014implicit} and \cite{ioffe2017variational}. If $\inv F$ is inner Lipschitz at $y$ (relative to a set $A\subset X$) then $F$ satisfies inner calmness (w.r.t. $A$) property introduced in \cite{benko2019calculus,benko2021inner}. The latter is weaker property as it doesn't require Lipschitz inequality to hold in entire neighbourhood but only for some sequence in it.
\end{remark}

\begin{remark}[The Aubin property]
    \label{rem:regularity:hausdorff}
    If $F$ has the Aubin property at $\bar x$ for $\bar y$, i.e.,
    \[
        \dist(F(\tilde x), \tilde y)  \le \kappa \dist(\inv F(\tilde y), \tilde x)
        \quad\text{for all}\quad
        \tilde x \in B(\bar x, \delta_x) \text{ and } \tilde y \in B(\bar y, \delta_y),
    \]
    for some factor (called the \term{graphical modulus}) $\kappa>0$ and radii $\delta_x, \delta_y>0$,
    then it also has the inner Lipschitz property at $\bar x$ relative to $A=B(\bar y, \delta_y)$.
    However, the inner Lipschitz property does not imply the Aubin property, especially as it only considers $\tilde x \in \Dom F$, and allows the factors $\lambda_y$ to depend on $y$.
\end{remark}



The next lemma considers inverses of translated radial functions, in particular, inverses of $J(x)=\frac{1}{2}\norm{x-b}^2$.

\begin{lemma}
    \label{lemma:calculus-inner-lipschitz-norm}
    In a normed space $Y$, let $J: Y \to \R$, $J(y)=\phi(\norm{y-b})$ for some $\phi: \R \to \R$ with a $\kappa$-Lipschitz inverse.
    Then $\inv J$ is inner Lipschitz.
\end{lemma}

\begin{proof}
    Let $v = J(y)$ for some $y$.
    Then
    \[
        \inv J(\tilde v) = \{ \tilde y \mid \norm{\tilde y-b}=\inv \phi (\tilde v) \} \ni \frac{\inv\phi(\tilde v)}{\inv\phi(v)}(y  - b) + b.
    \]
    It follows for any $\tilde v \in \Dom \inv J=\range J$ that
    \begin{equation}
        \label{eq:calculus-inner-lipschitz-norm:1}
        \dist(\inv J(\tilde v), y)
        \le \adaptabs{\frac{\inv\phi(\tilde v)}{\inv\phi(v)} - 1} \norm{y  - b}
        = \abs{\inv\phi(\tilde v) - \inv\phi(v)}
        \le \kappa\abs{\tilde v - v}.
    \end{equation}
    That is, $\inv J$ is inner Lipschitz.
\end{proof}

\subsection{Chain rules}

We now start our calculus results with chain rules.
Our main result follows from the next lemma when we take $g=\infmap{J \circ S}$ or  $g=\supmap{J \circ S}$, but it can also be applied to derive calculus rules for arbitrary selections of $J \circ S$.
In \cref{item:calculus:selection-chain:ii,item:calculus:selection-chain:iii}, note that it is possible that $u \in A \subset \Spess_g(x)$ with $u \in \closure \Sopt_g(x) \setminus \Sopt_g(x)$, and $\Sopt_g(x|u)(u^*) \ne \emptyset$.

\begin{lemma}
    \label{lemma:calculus:section-chain}
    Let $S: X \setto U$, $J: U \to \R$, and $g: X \to \extR$ on normed spaces $U$ and $X$.
    Define
    \[
        \Sopt_g(x) \defeq \{  u \in S(x) \mid g(x) \ge J(u) \}
        \quad\text{and}\quad
        \Spess_g(x) \defeq \{  u \in S(x) \mid g(x) \le J(u) \}.
    \]
    Then the following hold at any $x \in X$:
    \begin{enumerate}[label=(\roman*)]
        \item\label{item:calculus:selection-chain:i}
              If $J$ is continuously differentiable near $x$ and $\Sopt_g(x)$ is non-empty, then
              \[
                  \frechetSubdiff g(x)
                  \subset
                  \Isect_{u \in \Sopt_g(x)} \frechetCod \Spess_g(x|u)(J'(u)).
              \]
        \item\label{item:calculus:selection-chain:ii}
            If $\Sopt_g$ is inner Lipschitz at $x$ relative to an $A \subset \Spess_g(x)$, then
            \[
                \Union_{u \in A,\, u^* \in \frechetSubdiff J(u)} \frechetCod \Sopt_g(x|u)(u^*)
                \subset
                \frechetSubdiff g(x).
            \]
            \item\label{item:calculus:selection-chain:iii}
                If $\Sopt_g$ is inner semi-Lipschitz at $x$ relative to an $A \subset \Spess_g(x)$, then
                \[
                    \Isect_{u \in A} \Union_{u^* \in \frechetSubdiff J(u)} \frechetCod \Sopt_g(x|u)(u^*) \
                    \subset
                    \frechetSubdiff g(x).
                \]
    \end{enumerate}
\end{lemma}

\begin{proof}
    \cref{item:calculus:selection-chain:i}:
    Let  $x^*\in \frechetSubdiff g(x)$.
    Then, by definition
    \begin{equation}
        \label{eq:calculus:selection-chain:1}
        \limsup_{\this x \to x}
        \frac{\dualprod{x^*}{\thisx - x} - (g(\thisx) - g(x))}
        {\norm{\thisx - x}} \le 0.
    \end{equation}
    By assumption there exists some $u \in \Sopt_g(x)$. Take any.
    Then $J(u) \le g(x)$.
    Since $J(\thisu) \ge g(\thisx)$ for any $\thisu \in \Spess_g(\thisx)$, the characterisation \eqref{eq:calculus:selection-chain:1} implies
    \begin{equation*}
        \limsup_{\this x \to x,\, \Spess_g(\thisx) \ni \this u \to u}
        \frac{\dualprod{x^*}{\thisx - x} - (J(\this u) - J(u))}
        {\norm{\thisx - x} + \norm{\this u - u}} \le 0.
    \end{equation*}
    By the mean value theorem, for some $\this\zeta \in [u, \this u]$, this implies
    \[
        \limsup_{\this x \to x,\, \Spess_g(\thisx) \ni \this u \to u}
        \frac{\dualprod{x^*}{\thisx - x} - \dualprod{J'(u)}{\this u - u} + \dualprod{J'(u)-J'(\this\zeta)}{\this u -u}}
        {\norm{\thisx - x} + \norm{\this u - u}} \le 0.
    \]
    By the continuity of $J'$, we have
    \[
        \lim_{\this u \to u} \frac{\dualprod{J'(u)-J'(\this\zeta)}{\this u -u}}{\norm{\this u - u}}  = 0.
    \]
    Thus,
    \begin{equation}
        \label{eq:calculus:selection-chain:2}
        \limsup_{\this x \to x,\, \Spess(\thisx) \ni \this u \to u}
        \frac{\dualprod{x^*}{\thisx - x} - \dualprod{J'(u)}{\this u - u}}
        {\norm{\thisx - x} + \norm{\this u - u}} \le 0,
    \end{equation}
    which is to say $x^*\in \frechetCod \Spess_g(x|u)(J'(u))$.
    Since $u \in \Sopt_g(x)$, we obtain  \cref{item:calculus:selection-chain:i}.

    \cref{item:calculus:selection-chain:ii}:
    Suppose $x^* \in \frechetCod \Sopt_g(x|u)(u^*)$ for a $u \in A \subset \Spess_g(x)$ and $u^* \in \frechetSubdiff J(u)$.
    The former is characterised by the inequality analogous to \eqref{eq:calculus:selection-chain:2},
    \begin{equation}
        \label{eq:calculus:selection-chain:ii:0}
        \limsup_{\this x \to x,\, \Sopt_g(\thisx) \ni \this u \to u}
        \frac{\dualprod{x^*}{\thisx - x} - \dualprod{u^*}{\this u - u}}
        {\norm{\thisx - x} + \norm{\this u - u}} \le 0.
    \end{equation}
    We need to show that $x^* \in \frechetSubdiff g(x)$.
    By definition, this holds if and only if \eqref{eq:calculus:selection-chain:1} holds.
    Therefore, take any $\this x \to x$.
    Since $\Sopt_g$ is inner Lipschitz at $x$ relative to $A \ni u$ there exists $\kappa_u>0$ such that, for any $\kappa'>\kappa_u$, for large enough $k$, we can find $\this u \in \Sopt_g(\thisx)$ with $\norm{\thisu - u} \le \kappa' \norm{\thisx-x}$. Thus, in particular, $\Sopt_g(\thisx) \ni \this u \to u$.
    But then \eqref{eq:calculus:selection-chain:ii:0} implies
    \begin{equation}
        \label{eq:calculus:selection-chain:ii:1}
        \limsup_{k \to \infty}
        \frac{\dualprod{x^*}{\thisx - x} - \dualprod{u^*}{\this u - u}}
        {\norm{\thisx - x} + \norm{\this u - u}} \le 0.
    \end{equation}
    By the definition \eqref{eq:intro:convex-subdiff} of the Fréchet subdifferential, we have
    \[
        \limsup_{k \to \infty}
        \frac{\dualprod{u^*}{\thisu-u}- (J(\thisu) - J(u))}{\norm{\this x - x}} \le 0.
    \]
    Using this and  $\norm{\thisu - u} \le \kappa' \norm{\thisx-x}$ in \eqref{eq:calculus:selection-chain:ii:1}, we obtain
    \begin{equation}
        \label{eq:calculus:selection-chain:ii:2}
        \limsup_{k \to \infty}
        \frac{\dualprod{x^*}{\thisx - x} - (J(\this u) - J(u))}
        {\norm{\thisx - x}} \le 0.
    \end{equation}
    Keeping in mind that $g(\thisx) \ge J(\this u)$ by the definition of $\Sopt(\thisx) \ni \this u$, and $g(x) \le J(u)$ by the definition of $\Spess(x) \ni u$, since the sequence $\this x \to x$ was arbitrary, this proves \eqref{eq:calculus:selection-chain:1}.

    \cref{item:calculus:selection-chain:iii}: We proceed analogously to \cref{item:calculus:selection-chain:ii}, except in \cref{eq:calculus:selection-chain:ii:1}, after the use of the inner semi-Lipschitz property relative to $\Sopt_g(x)$, we have passed to a subsequence, unrelabelled, and $u \in A \subset \Spess_g(x)$ is no longer arbitrary, but determined by the property.
    To justify passing from \eqref{eq:calculus:selection-chain:ii:2} to \eqref{eq:calculus:selection-chain:1}, we argue that we can repeat the argument for an arbitrary subsequence of the original.
    Note that the choice of $u$ does not depend on $u^* \in \frechetSubdiff J(u)$, therefore, the union over $u^*$ can be inside the intersection over $u$ in the claim.
\end{proof}

\begin{theorem}
    \label{thm:calculus:single-valued-chain}
    Let $S: X \setto U$, and $J: U \to \R$ on normed spaces $U$ and $X$.
    Then the following hold:
    \begin{enumerate}[label=(\roman*)]
        \item\label{item:calculus:single-valued-chain:i}
              If $J$ is continuously differentiable near $x$ and $\Sopt(x)$ is non-empty, then
              \[
                  \frechetSubdiff \infmap{J \circ S}(x)
                  \subset
                  \Isect_{u \in \Sopt(x)} \frechetCod S(x|u)(J'(u)).
              \]
        \item\label{item:calculus:single-valued-chain:ii}
            If $S$ is inner Lipschitz at $x$ relative to an $A \subset \Spess(x) \subset S(x)$, then
            \[
                \Union_{u \in A,\, u^* \in \frechetSubdiff J(u)} \frechetCod S(x|u)(u^*)
                \subset
                \frechetSubdiff \supmap{J \circ S}(x).
            \]
            \item\label{item:calculus:single-valued-chain:iii}
                If $S$ is inner semi-Lipschitz at $x$ relative to an $A \subset \Spess(x) \subset S(x)$, then
                \[
                    \Isect_{u \in A} \Union_{u^* \in \frechetSubdiff J(u)}\frechetCod S(x|u)(u^*)
                    \subset
                    \frechetSubdiff \supmap{J \circ S}(x).
                \]
    \end{enumerate}
\end{theorem}

\begin{proof}
    \cref{item:calculus:single-valued-chain:i}:
    Apply \cref{lemma:calculus:section-chain}\,\cref{item:calculus:selection-chain:i} with $g=\infmap{J \circ S}$. Then $\Sopt_g=\Sopt$, and $\Spess_g=S$.

    \cref{item:calculus:single-valued-chain:ii,item:calculus:single-valued-chain:iii}:
    Apply \cref{lemma:calculus:section-chain}\,\cref{item:calculus:selection-chain:ii,item:calculus:selection-chain:iii} with $g=\supmap{J \circ S}$.
    Then $\Sopt_g=S$, and $\Spess_g=\Spess$.
\end{proof}

\subsection{Sum rules}

We now provide a sum rule that relaxes assumptions in known sum rules; see, e.g, \cite{clason2020introduction,rockafellar-wets-va}.
For $A \subset X$, we set,
\begin{equation}
    \label{eq:calculus:f-a}
    F_A(x) \defeq
    \begin{cases}
        \emptyset,    & x \not \in A,
        \\
        F(x), & x \in A.
    \end{cases}
\end{equation}
Then we define the \term{$A$-lower subdifferential}
\[
    \lsubdiff{A} R(x) \defeq \left\{
    x^*\in \frechetSubdiff R(x)
    \,\middle|\,
    \limsup_{A \ni \thisx \to x}
    \frac{R(\thisx)-R(x)-\dualprod{x^*}{\thisx-x}}{\norm{\thisx-x}} \le 0
    \right\}.
\]

Despite relaxing standard results, we still have the relatively strict assumption  $\lsubdiff{A} R(x) \ne \emptyset$, which, in particular, excludes the one-norm, but does allow several forms of constraints.

\begin{example}
    \label{example-lower-subdiff-1}
    $\lsubdiff{\Dom R} R(x) =\frechetSubdiff R(x)$ if $R$ is Lipschitz differentiable.
\end{example}

\begin{example}
    \label{example-lower-subdiff-2}
    Since $0 \in \subdiff \delta_C(x)$ for all $x \in C$ for a non-empty closed convex set $C$, we always have $0 \in \lsubdiff{\Dom \delta_C} \delta_C(x)$.
\end{example}

\begin{example}
    \label{example-lower-subdiff-3}
    For $R(x)=\abs{x}$ on $\R$, we have $\lsubdiff{[0, \infty)}R(0)=\{1\}$ and $\lsubdiff{(-\infty,0]}R(0)=\{-1\}$.
\end{example}

\begin{theorem}
    \label{thm:calculus:single-valued-sum}
    Suppose $R: X \to \extR$ is convex, proper, and locally Lipschitz on its effective domain (not only the interior of the domain), and $F: X \to \R$ on a normed space $X$.
    Let$x \in \Dom R$. Then the following hold:
    \begin{enumerate}[label=(\roman*)]
        \item
              \label{item:calculus:single-valued-sum:sum-d-subset-d-sum}
              The inclusion\footnote{This inclusion holds in general.}
              \[
                  \frechetSubdiff F(x) + \frechetSubdiff R(x)
                  \subset \frechetSubdiff[F + R](x)
              \]
        \item
              \label{item:calculus:single-valued-sum:d-sum-subset-sum-d}
              If, $\lsubdiff{A} R(x) \ne \emptyset$ for a set $A \subset X$, the inclusion
              \[
                  \frechetSubdiff[F + R](x) \subset \frechetSubdiff F_A(x) + \tilde x^*
                  \quad\text{for any}\quad \tilde x^* \in \lsubdiff{A} R(x).
              \]
    \end{enumerate}
\end{theorem}

\begin{proof}
    By definition, $x^* \in \frechetSubdiff[F + R](x)$ if and only if
    \begin{equation}
        \label{eq:calculus:single-valued-sum:0}
        \limsup_{\this x \to x}
        \frac{\dualprod{x^*}{\thisx - x} - \left([F+R](\thisx) - [F+R](x)\right)}
        {\norm{\thisx - x}} \le 0.
    \end{equation}

    \emph{\cref{item:calculus:single-valued-sum:sum-d-subset-d-sum}:}
    Let $\tilde x^* \in \frechetSubdiff R(x)$.
    Using the definition \eqref{eq:intro:frechet-subdiff} of the latter, we see \eqref{eq:calculus:single-valued-sum:0} to hold if
    \begin{equation}
        \label{eq:calculus:single-valued-sum:2}
        \limsup_{\this x \to x}
        \frac{\dualprod{x^*-\tilde x^*}{\thisx - x} - (F(\thisx) - F(x))}
        {\norm{\thisx - x}} \le 0.
    \end{equation}
    This says exactly that $x^* - \tilde x^*\in \frechetSubdiff F(x)$, proving \cref{item:calculus:single-valued-sum:sum-d-subset-d-sum}.

    \emph{\cref{item:calculus:single-valued-sum:d-sum-subset-sum-d}:}
    In the other direction, restricting $\thisx \in A$, \eqref{eq:calculus:single-valued-sum:0} implies for any $\tilde x^* \in X^*$ that
    \begin{equation}
        \label{eq:calculus:single-valued-sum:3}
        \limsup_{A \ni \this x \to x}
        \frac{\dualprod{x^*-\tilde x^*}{\thisx - x}  - (F(\thisx) - F(x))}
        {\norm{(\thisx - x, F(\thisx) - F(x))}}
        \le
        \limsup_{A \ni \this x \to x}
        \frac{R(x^k)-R(x)-\dualprod{\tilde x^*}{\thisx - x}}
        {\norm{\thisx - x}}.
    \end{equation}
    The right-hand-side is non-positive for any $\tilde x^* \in \lsubdiff{A} R(x) \subset \frechetSubdiff R(x)$.
    We have assumed $\lsubdiff{A} R(x) $ non-empty.
    Therefore, such a choice of $\tilde x^*$ exists.
    From \eqref{eq:calculus:single-valued-sum:3} we now get that $x^* \in \frechetSubdiff F_A(x) +\tilde x^*$.
    This proves \cref{item:calculus:single-valued-sum:d-sum-subset-sum-d}.
\end{proof}

\begin{remark}
    Similar sum rule to \cref{item:calculus:single-valued-sum:d-sum-subset-sum-d} is obtained for $\tilde x^*\in -\frechetSubdiff (-R(x))$ in \cite{mordukhovich2006frechet}. The function $x \mapsto -\frechetSubdiff (-R(x))$ is called  Fréchet upper subdifferential therein, and it behaves like $A$-lower subdifferential in \cref{example-lower-subdiff-1,example-lower-subdiff-2,example-lower-subdiff-3}.
\end{remark}

\section{Optimality conditions for nonsmooth bilevel problems}
\label{sec:oc}

Throughout this section, on normed spaces $U$, $W_*, W = (W_*)^*$, and $X$, where $W_*$ is a predual space of $W$, we consider the bilevel optimisation problem \eqref{eq:intro:bilevel-problem}, i.e.,
\begin{equation}
    \label{eq:oc:bilevel-problem}
    \min_{u, x} J(u) + R(x)\quad\text{subject to}\quad 0 \in G(u, x).
\end{equation}
Here $G: U \times X \setto W_*$, $J: U \to \R$, and $R: X \to \extR$.

We start in \cref{sec:optimistic} by deriving optimality conditions for \eqref{eq:oc:bilevel-problem} or, more precisely, the optimistic and pessimistic variants
\[
    \min_{x} [\infmap{J \circ S} + R](x)
    \quad\text{and}\quad
    \min_{x} [\supmap{J \circ S} + R](x),
\]
where the optimistic and pessimistic selections $\infmap{J \circ S}$ and $\supmap{J \circ S}$,
are defined in \cref{sec:calculus} along with the corresponding slices $\Sopt$ and $\Spess$ of the inner solution mapping
\begin{equation}
    \label{eq:oc:solution-mapping}
    S: X \setto U,
    \quad
    S(x) \defeq \{u \in U \mid 0 \in G(u, x)\}.
\end{equation}
In \cref{sec:adjoint} we further characterise these optimality conditions with the help of nonsmooth adjoint inclusions.
We finish the section in \cref{sec:fullsystem} by writing out the full inner-adjoint-outer optimality system with the help of additional variables.

\subsection{Criticality conditions}
\label{sec:optimistic}

By the Fermat principles for the Fréchet and limiting subdifferentials \cite[Theorem 16.2 and Corollary 16.6]{clason2020introduction}, for $\bar x$ to be an optimistic minimiser of \eqref{eq:oc:bilevel-problem}, the condition
\begin{equation}
    \label{eq:optimistic:weak-oc-nec}
    0 \in \anySubdiff[\infmap{J \circ S} + R](\bar x)
\end{equation}
is necessary for both $\anySubdiff=\limitingSubdiff, \frechetSubdiff$.
Likewise for $\bar x$ to be a pessimistic minimiser of \eqref{eq:oc:bilevel-problem}, it is necessary that
\[
    0 \in \anySubdiff[\supmap{J \circ S} + R](\bar x).
\]

The next results provide further expanded conditions under additional assumptions.
For the first result, recall the definition of $S_A$ from \eqref{eq:calculus:f-a}.

\begin{theorem}
    \label{thm:optimistic:weak-oc}
    Let $J$, $S$, and $R$ be as in \cref{eq:oc:bilevel-problem,eq:oc:solution-mapping} with $J$ continuously differentiable.
    Suppose $\infmap{J \circ S} + R$ is minimised at $\bar x$.
    If $\lsubdiff{A} R(\bar x) \ne \emptyset$ for a set $A \subset X$, then
    \begin{equation}
        \label{eq:optimistic:weak-oc}
        0 \in \frechetCod S_A(\bar x|\bar u)(J'(\bar u)) + \lsubdiff{A} R(\bar x)
        \quad\text{for all}\quad \bar u \in \Sopt(\bar x).
    \end{equation}
    If, in particular, $\Dom R=X$ with $\lsubdiff{X} R(\bar x) \ne \emptyset$ (e.g., if $R$ is Lipschitz continuously differentiable), then
    \begin{equation}
        \label{eq:optimistic:weak-oc:lc}
        0 \in \frechetCod S(\bar x|\bar u)(J'(\bar u)) + \frechetSubdiff R(\bar x)
        \quad\text{for all}\quad \bar u \in \Sopt(\bar x).
    \end{equation}
\end{theorem}

\begin{proof}
    Let $\bar v \defeq \infmap{J \circ S}(\bar x)$.
    By \eqref{eq:optimistic:weak-oc-nec}, \cref{thm:calculus:single-valued-sum}\,\cref{item:calculus:single-valued-sum:d-sum-subset-sum-d} and
    \cref{thm:calculus:single-valued-chain}\,\cref{item:calculus:single-valued-chain:i}
    \[
        \begin{split}
            0
            \in
            \frechetSubdiff[\infmap{J \circ S} + R](\bar x)
             &
            \subset
            \frechetCod [J \circ S]_A(\bar x|\bar v)(1)
            + \lsubdiff{A} R(\bar x)
            \\
             &
            =
            \frechetCod [J \circ S_A](\bar x|\bar v)(1)
            + \lsubdiff{A} R(\bar x)
            \\
             &
            \subset
            \Isect_{u \in \Sopt(\bar x)}
            \frechetCod S_A(\bar x|u)(J'(u)) + \lsubdiff{A} R(\bar x).
        \end{split}
    \]
    This proves \eqref{eq:optimistic:weak-oc}. For \eqref{eq:optimistic:weak-oc:lc} we take $A=X$.
\end{proof}

\begin{limiting}
Since the limiting coderivative and subdifferential contain the Fréchet coderivative and subdifferential, the following is immediate.
\begin{corollary}
    \label{cor:optimistic:weak-oc:limiting}
    \Cref{thm:optimistic:weak-oc} holds with the limiting coderivative in place of the Fréchet coderivative.
\end{corollary}
\end{limiting}

We will, however, wish to avoid the $A=\Dom R$ restriction on $S$, and the restrictions on $R$, so base our work with the Fréchet coderivative mainly on the condition of the next lemma, which is neither necessary nor sufficient, but is a \emph{sufficient condition for the pessimistic necessary condition \eqref{eq:optimistic:weak-oc-nec}}.\obtodo{If we can solve one, sufficient-for-necessary is better than necessary-for-necessary.}
It will be necessary in each case to verify that \eqref{eq:optimistic:weak-critical} or \eqref{eq:optimistic:strong-critical} can be satisfied (e.g., through becoming necessary): it is no way necessary that they can be satisfied.

\begin{theorem}
    \label{thm:optimistic:weak-oc-necsuf}
    Consider $F \defeq  J \circ S + R$, where $J$, $S$, and $R$ are as in \cref{eq:oc:bilevel-problem,eq:oc:solution-mapping}.
    Then the pessimistic necessary optimality condition $0 \in \frechetSubdiff[\supmap{J \circ S} + R](\bar x)$ holds in the following cases:
    \begin{enumerate}[label=(\roman*)]
        \item\label{item:optimistic:weak-oc-necsuf:strong}
              $S$ is inner Lipschitz at $\bar x$ relative to $\{\bar u\}$ for a $\bar u \in \Spess(\bar x)$, and
              \begin{equation}
                  \label{eq:optimistic:weak-critical}
                  0 \in \Union_{u^* \in \frechetSubdiff J(\bar u)} \frechetCod S(\bar x|\bar u)(u^*) + \frechetSubdiff R(\bar x)
              \end{equation}
        \item\label{item:optimistic:weak-oc-necsuf:weak}
              $S$ is inner semi-Lipschitz at $\bar x$ relative to an $A \subset \Spess(\bar x)$, and
              {\belowdisplayskip=0pt
              \begin{equation}
                  \label{eq:optimistic:strong-critical}
                  0 \in \Isect_{\bar u \in A} \Union_{u^* \in \frechetSubdiff J(\bar u)} \frechetCod S(\bar x|\bar u)(u^*) + \frechetSubdiff R(\bar x).
              \end{equation}}
    \end{enumerate}
\end{theorem}

\begin{proof}
    \cref{item:optimistic:weak-oc-necsuf:strong}:
    Let $\bar v \defeq J(\bar u)$.
    By \cref{thm:calculus:single-valued-sum}\,\cref{item:calculus:single-valued-sum:sum-d-subset-d-sum} and  \cref{thm:calculus:single-valued-chain}\,\cref{item:calculus:single-valued-chain:ii}, we have
    \begin{equation}
        \label{eq:optimistic:weak-oc-necsuf:0}
        \frechetSubdiff[\supmap{J \circ S} + R](\bar x)
        \supset
        \frechetCod [J \circ S](\bar x|\bar v)(1)
        + \frechetSubdiff R(\bar x)
        \supset
        \Union_{u^* \in \frechetSubdiff J(\bar u)} \frechetCod S(\bar x|\bar u)(u^*) + \frechetSubdiff R(\bar x).
    \end{equation}

    \cref{item:optimistic:weak-oc-necsuf:weak}:
    We use  \cref{thm:calculus:single-valued-chain}\,\cref{item:calculus:single-valued-chain:iii} instead of \cref{item:calculus:single-valued-chain:ii} in the final equality of \eqref{eq:optimistic:weak-oc-necsuf:0}.
\end{proof}

\begin{remark}[Limiting coderivative]
    If $\limitingSubdiff \supmap{J \circ S}(\bar x) \subset D^* [J \circ S](\bar x|\bar v)(1)$ for $\bar v = \supmap{J \circ S}(\bar x)$, such as when $S$ is single-valued, and various qualification conditions hold, the sum and chain rules of \cite[Theorems 3.10 and 3.13]{mordukhovich2006variational} can be used to used to derive expanded necessary optimality conditions.
    We do not pursue this route here, concentrating on optimalistic optimality conditions for the limiting coderivative.
\end{remark}

\begin{remark}[Optimistic/pessimistic, necessary/sufficient, and limiting/Fréchet]
    We note that with the limiting coderivative, we easily get necessary optimality conditions for the optimistic bilevel problem.
    The smaller Fréchet coderivative, on the other hand, seems to better correspond to the pessimistic bilevel problem, but -- according to what we have been able to derive here -- does gives neither a necessary nor sufficient optimality condition, but a sufficient-for-necessary optimality condition. Under additional assumptions, it can be made either sufficient or necessary.
\end{remark}

\subsection{Nonsmooth adjoint inclusions}
\label{sec:adjoint}

The next lemma establishes for both $\anyCod = \frechetCod$ and $\anyCod = D^*$ the \term{nonsmooth adjoint inclusion}
\[
    (-u^*, x^*) \in \range \anyCod G(u, x | 0).
\]
For the limiting coderivative this is necessary for $x^* \in D^* S(x|u)(u^*)$, subject to a qualification condition, while for the Fréchet coderivative it is only sufficient.
In the latter case it provides a way to show that $\frechetCod S(x|u)(u^*)$ is non-empty, and thus to a suggestion that the sufficient-for-necessary condition \eqref{eq:optimistic:weak-critical} could be satisfied.

\begin{lemma}
    \label{lemma:solution-map:nonempty-coderivative}
    Suppose $u \in S(x)$. Then, for all $u^* \in U^*$,
    \[
        \{ x^* \in X^* \mid  (-u^*, x^*) \in \range \frechetCod G(u, x | 0, 0) \}
        \subset
        \frechetCod S(x|u)(u^*)
        \subset
        D^* S(x|u)(u^*).
    \]
    If
    \begin{equation}
        \label{eq:solution-map:nonempty-coderivative:qc}
        D_M^*\inv G(0|u, x)(0) = \{ 0 \},
    \end{equation}
    then
    \[
        D^* S(x|u)(u^*)
        \subset
        \{ x^* \in X^* \mid  (-u^*, x^*) \in \range D^* G(u, x | 0) \}
    \]
\end{lemma}

\begin{proof}
    By definition, we have $x^* \in \frechetCod S(x|u)(u^*)$ if and only if
    \begin{equation}
        \label{eq:solution-map:S-coderivative-def}
        \limsup_{\this x \to x,\, \this u \to u,\, G(\this u, \thisx) \ni 0}
        \frac{\dualprod{x^*}{\thisx-x}-\dualprod{u^*}{\this u - u}}
        {\norm{(\thisx-x,\this u - u)}}
        \le 0.
    \end{equation}
    Likewise $(-u^*, x^*) \in \frechetCod G(u, x | 0)(w)$ for some $w \in W$ if and only if
    \begin{equation}
        \label{eq:G-coderivative-def}
        \limsup_{\this x \to x,\, \this u \to u,\, G(\this u, \thisx) \ni \this w_* \to 0}
        \frac{-\dualprod{w}{\this w_* - 0} + \dualprod{x^*}{\thisx-x}-\dualprod{u^*}{\this u - u}}
        {\norm{(\thisx-x,\this u - u,\this w_* - 0)}}
        \le 0.
    \end{equation}
    Restricting attention to $\this w_* = 0$ in \cref{eq:G-coderivative-def}, it implies \eqref{eq:solution-map:S-coderivative-def}.
    This proves the first inclusion.

    For the second inclusion, we use \cite[Theorem 3.8]{mordukhovich2006variational} with $\Theta=\{0\}$ and $F=G$.
    Since $N_{\Theta}(0)=X^*$, the qualification condition in the result requires that $\ker \tilde D_M^*G(u, x|0) = \{ 0 \}$ for the “reverse mixed coderivative”, defined for $F: X \setto Y$ as \cite[(1.40)]{mordukhovich2006variational}
    $
        \tilde D_M^*F(x|y)(y^*) \defeq \{ x^* \in X^* \mid y^* \in -D_M^*\inv F(y|x)(-x^*)\}.
    $
    That is, the qualification condition reads
    $w_* \in  -D_M^*\inv G(0|u,x)(0) \implies w_*=0$ with $0 \in -D_M^*\inv G(0|u,x)(0)$.
    This can be equivalently be written \eqref{eq:solution-map:nonempty-coderivative:qc}.
    Noting that $\graph \inv S = \inv G(\Theta)$, that $\Theta$ is partially sequentially normally compact, and $(u, x) \mapsto G(u, x) \isect \Theta \subset \{0\}$ is inner semicompact,\footnote{%
        There are inconsistencies in the definition of inner semicompactness in \cite{mordukhovich2006variational}. According to Definition 1.63 therein, $S: X \setto Y$ is inner semicompact at $\bar x$ if for every $x^k \to \bar x$, there exist $y^k \in S(x^k)$ and a convergent subsequence $\{y^{k_\ell}\}_{\ell \in \N}$.
        Our mapping does not satisfy this definition.
        However, in practise in the proofs, $x^k \in \Dom S$.
        Our mapping satisfies this domain-restricted definition.
        Indeed, in the proof of \cite[Theorem 3.8]{mordukhovich2006variational}, one takes $x^k \in \inv F(\Theta)$, which guarantees the existence of $y^k \in \Theta \isect F(x^k)$.
        In our case, $y^k \equiv 0$.
    }
    it follows that 
    \[
        \begin{split}
            D^*S(x|u)(u^*)
             &
            =
            \{ x^* \mid (x^*, -u^*) \in N_{\graph S}(x, u) \}
            =
            \{ x^* \mid (-u^*, x^*) \in N_{\graph \inv S}(u, x) \}
            \\
             &
            =
            \{ x^* \mid (-u^*, x^*) \in N_{\inv G(\Theta)}(u, x) \}
            \\
             &
            \subset
            \{x ^* \mid (-u^*, x^*) \in D^*G(u, x|w)(w^*),\, w^* \in N_\Theta(w), \, w \in G(u, x) \isect \Theta \}
            \\
             &
            =
            \{ x^* \in X^* \mid  (-u^*, x^*) \in \range D^* G(u, x | 0) \}.
            \qedhere
        \end{split}
    \]
\end{proof}

\begin{remark}
    The “mixed” limiting coderivative is used in \eqref{eq:solution-map:nonempty-coderivative:qc}.
    When $Y^*$ is finite-dimensional, and the “mixed” and “normal” limiting coderivatives agree, the condition can be written as
    \[
        0 \in D^* G(u, x|0)(w) \implies w = 0
        \quad\text{i.e.}\quad
        \ker D^* G(u, x|0) = \{0\}.
    \]
\end{remark}

\begin{corollary}
    \label{cor:solution-map:nonempty-coderivative:regular}
    Suppose $u \in S(x)$.
    If the qualification condition \eqref{eq:solution-map:nonempty-coderivative:qc} holds and $G$ is $N$-regular at $(u, x)$ for $0$, then $S$ is $N$-regular at $x$ for $u$, and we have
    \[
        \frechetCod S(x|u)(u^*)
        =
        \{ x^* \in X^* \mid  (-u^*, x^*) \in \range \frechetCod G(u, x | 0) \}.
    \]
\end{corollary}

The following example demonstrates that the non-smooth (optimality) condition $0\in \anyCod S(x|u)(J'(u))$ generalises the smooth one, e.g. \cite{suonpera2024general}.

\begin{example}
    Suppose $G$ is single-valued and Fréchet differentiable. Then by, e.g., \cite[Theorem 20.12]{clason2020introduction}, when $0 = G(u, x)$, i.e. $u \in S(x)$, we have, for $\anyCod=\frechetCod,D^*,$
    \[
        \anyCod G(u, x | 0)(w)
        = \{G'(u, x)^*w\}
        =: \{(\diffwrt{G}{u}(u, x)^*w, \diffwrt{G}{x}(u, x)^*w)\}.
    \]
    Consequently, by the lemma, for any $w \in W$ and
    \[
        u^* = -\diffwrt{G}{u}(u, x)^*w
        \quad\text{and}\quad
        x^* = \diffwrt{G}{x}(u, x)^*w,
    \]
    we have $x^* \in \anyCod S(x|u)(u^*)$.
    In particular, if $\diffwrt{G}{u}(u, x)^*$ is invertible, for arbitrary $u^*$, solving $w=\diffwrt{G}{u}(u, x)^{*,-1}u^*$, we have $-\diffwrt{G}{ x}(u, x)^*\diffwrt{G}{u}(u, x)^{*,-1}u^* \in \anyCod S(x|u)(u^*)$.
\end{example}

\subsection{Full system optimality conditions}
\label{sec:fullsystem}

We can now write our full-system optimality conditions.
We start with a sufficient-for-pessimistic necessary condition based on the Fréchet coderivative.
For simplicity, we now always take $J$ continuously differentiable.

\begin{theorem}
    \label{thm:fullsystem:frechet}
    Consider $F \defeq  J \circ S + R$, where $J$, $S$, and $R$ are as in \cref{eq:oc:bilevel-problem,eq:oc:solution-mapping} with $J$ continuously differentiable.
    Then the pessimistic necessary optimality condition $0 \in \frechetSubdiff[\supmap{J \circ S} + R](\bar x)$ for the problem \eqref{eq:oc:bilevel-problem} holds if $S$ is inner Lipschitz at $\bar x$ relative to $\{\bar u\}$, and
    \begin{equation}
        \label{eq:optimistic:weak-oc-necsuf:further}
        (-J'(\bar u), \bar x^*) \in \range \frechetCod G(\bar u, \bar x | 0)
        \quad\text{with}\quad
        0 \in \bar x^* + \frechetSubdiff R(\bar x)
        \quad\text{and}\quad
        \bar u \in \Spess(\bar x).
    \end{equation}
\end{theorem}

\begin{proof}
    Combine \cref{thm:optimistic:weak-oc-necsuf}\,\cref{item:optimistic:weak-oc-necsuf:strong} with \cref{lemma:solution-map:nonempty-coderivative}.
\end{proof}

We next provide a necessary optimistic optimality condition for the problem \eqref{eq:oc:bilevel-problem}.

\begin{theorem}
    \label{thm:fullsystem:limiting}
    Consider $F \defeq  J \circ S + R$, where $J$, $S$, and $R$ are as in \cref{eq:oc:bilevel-problem,eq:oc:solution-mapping} with $J$ continuously differentiable.
    Suppose $\infmap{J \circ S} + R$ is minimised  at $\bar x \in X$, and $\bar u \in \Sopt(\bar x)$.
    If $R$ is Lipschitz differentiable (or, more generally $\lsubdiff{\Dom R} R \ne \emptyset$) with $\Dom R=X$, and the qualification condition \eqref{eq:solution-map:nonempty-coderivative:qc} holds, then
    \begin{equation}
        \label{eq:optimistic:weak-oc-necsuf:further:limiting}
        (-J'(\bar u), \bar x^*) \in \range D^* G(\bar u, \bar x | 0)
        \quad\text{with}\quad
        0 \in \bar x^* + \frechetSubdiff R(\bar x)
        \quad\text{and}\quad
        \bar u \in \Sopt(\bar x)
    \end{equation}

\end{theorem}

\begin{proof}
    Combine \cref{cor:optimistic:weak-oc:limiting} with \cref{lemma:solution-map:nonempty-coderivative}.
\end{proof}

To work with either the Fréchet coderivative or the limiting coderivative, we now recall the notation $\anyCod$ to refer to either $\frechetCod$ or $D^*$.

\begin{remark}
    \label{rem:composition-subdiff}
    Consider for $\anyCod=D^*, \frechetCod$ and $\square=\sharp,\flat$, the sets
    \begin{align}
        \nonumber
        \anyCOsubdiffS_1^\square F(x) & \defeq \{x^*\in X^* \mid x^*\in \anyCod S(x|u)(J'(u)),\, u\in S^\square(x)\}
        \quad\text{and}
        \\
        \nonumber
        \anyCOsubdiffS_2^\square F(x) & \defeq \{x^*\in X^* \mid (-J'(u),x^*) \in \range \anyCod G(u,x|0),\, u\in S^\square(x)\}.
    \end{align}
    The following table indicates our results on when $0 \in [\COsubdiffS F +\subdiff R](x)$ for $\COsubdiffS=\anyCOsubdiffS_i^\square$ is a meaningful optimality condition for the problem \cref{eq:oc:bilevel-problem,eq:oc:solution-mapping}.
    \begin{center}
    \begin{tabular}{llllll}
        \toprule
        & & \multicolumn{2}{c}{Fréchet} & \multicolumn{2}{c}{Limiting}
        \\
        \cmidrule(lr){3-4} \cmidrule(lr){5-6}
        & $\square$ & $\widehat\COsubdiffS_1^\square F$ & $\widehat\COsubdiffS_2^\square F$ & $\COsubdiffS_1^\square F$ & $\COsubdiffS_2^\square F$
        \\
        \midrule
        suff.-for-nec. pessimistic & $\flat$ & \cref{thm:optimistic:weak-oc-necsuf}  & \cref{thm:fullsystem:frechet} &  &
        \\
        necessary optimistic & $\sharp$ & \cref{thm:optimistic:weak-oc} & & \cref{cor:optimistic:weak-oc:limiting} & \cref{thm:fullsystem:limiting}
        \\
        \bottomrule
    \end{tabular}
    \end{center}
\end{remark}

Practically, we will relax $u \in \Sopt(x)$ and $u \in \Spess(x)$ to $u \in S(x)$, i.e., $0 \in G(u, x)$, and work with
\begin{equation}
    \label{def:set-valued-target}
    \anyCOsubdiffS_2 F(x) \defeq \{x^*\in X^* \mid (-J'(u),x^*) \in \range \anyCod G(u,x|0),\, u\in S(x)\}
\end{equation}
that contains both  $\anyCOsubdiffS_2^\sharp F(x)$ and $\anyCOsubdiffS_2^\flat F(x)$.
Then $0 \in [\anyCOsubdiffS_2 F + \subdiff R](x)$ can also be written as
\begin{equation}
    \label{eq:algorithm:full-oc}
    0 \in \breve H(u, x, x^*, w)
    \defeq
    \begin{pmatrix}
        G(u, x)
        \\
        \anyCod G(u,x|0)(w) +  (J'(u), -x^*)
        \\
        x^* + \subdiff R(x)
    \end{pmatrix}.
\end{equation}
For the limiting coderivative, subject to the conditions of  \cref{thm:fullsystem:limiting}, this still remains a necessary optimalistic optimality condition.
If $S(x)$ is a singleton, then the version with the Fréchet coderivative provides a sufficient-for-necessary optimality condition.
In the next section, we will build algorithms to solve this system based on combining standard algorithms for each component inclusion.

\begin{remark}
    For smooth functions, \eqref{eq:algorithm:full-oc} differs from \cite{suonpera2022bilevel,suonpera2024general} through a different lower-dimensional adjoint variable.
    Indeed, \eqref{eq:algorithm:full-oc} directly gives the dimension-reduced adjoint inclusion as used in \cite[Example 2.1]{dizonvalkonen2024tracking} for bilevel optimisation, or classically derived for PDE-constrained optimisation in, e.g., \cite[§1.6.2]{hinze2009pde} or \cite[§1.2]{delosreyes2015numerical}.
    It directly calculates $x^*$ that corresponds to $J'(S(x))S'(x)$ in the smooth case, instead of first, separately, calculating $p=S'(x)$ as in \cite{suonpera2022bilevel,suonpera2024general}.
\end{remark}

\begin{remark}[Proximal optimality conditions in Hilbert spaces]
    \label{rem:optimistic:prox-oc}
    Suppose $R$ is convex, proper, and lower semicontinuous, and $X$ is a Hilbert space.
    Then, following standard results, \cite[Lemma 6.22]{clason2020introduction}, we can for any $\tau>0$ reinterpret \cref{eq:optimistic:weak-oc-necsuf:further} as
    \[
        \bar x = \prox_{\tau R}(\bar x - \tau \mathscr{R}\bar x^*)
        \quad\text{for some}\quad
        \bar x^* \in \widehat\COsubdiffS_2^\flat F(\bar x),
    \]
    and \eqref{eq:optimistic:weak-oc-necsuf:further:limiting} as
    \[
        \bar x = \prox_{\tau R}(\bar x - \tau \mathscr{R}\bar x^*)
        \quad\text{for some}\quad
        \bar x^* \in \COsubdiffS_2^\sharp F(\bar x).
    \]
    Likewise, \eqref{eq:optimistic:weak-oc:lc} and  \eqref{eq:optimistic:weak-oc:lc} can be given the proximal form in terms of $\anyCOsubdiffS_1^\square F(x)$.
\end{remark}

\section{Algorithm}
\label{sec:algorithm}

We now build a general approach to single-loop algorithms for \eqref{eq:intro:bilevel-problem}, based on the optimality conditions \eqref{eq:algorithm:full-oc}.
First in \cref{sec:implicit-abstract-algorithm}, we lay out the approach.
Then in \cref{sec:tracking}, we recall the \term{tracking theory} of \cite{dizonvalkonen2024tracking} for the construction of single-loop algorithms based on differential estimation.
This involves so-called \term{inner tracking}, \term{adjoint tracking}, and \term{differential transformation} conditions, which we state in \cref{ass:tracking:main}, and further interpret in \cref{sec:inner-tracking,sec:adjoint-tracking,sec:differential-transformation}.
In \cref{sec:inexact-convergence} we then show how these components give convergence to an outer forward-backward method.

\subsection{An general single-loop algorithm}
\label{sec:implicit-abstract-algorithm}

\begin{subequations}
    \label{eq:algorithm:abstract}
    Suppose  $U, W_*$ and $X$ are normed spaces, and denote $W=(W_*)^*$.
    Let $\anyCod = \frechetCod$ or $\anyCod = D^*$.
    We assume that $J: U \to \R$ is continuously differentiable and $R: X \to \extR$ convex.
    We divide the solution of \eqref{eq:algorithm:full-oc} into the following rough components:
    \begin{enumerate}
        \item an \term{inner algorithm}, which, on each iteration $k \in \N$, given the history of (outer, inner, adjoint, approximate subderivative) iterates $\{(x^j, u^j, w^j, \tilde x^*_j)\}_{j=0}^k$, produces $\nextu$ that approximately solves the first line, i.e.,
              \begin{equation}
                  \label{eq:algorithm:abstract:inner}
                  0 \approx \Delta\nexxt w_* \in G(\nextu, \thisx);
              \end{equation}

        \item an \term{adjoint algorithm}, which, given the previous iterates and $\nextu$, produces $w^{k+1}$ and $\tilde x^*_{k+1}$ that approximately solve the second line, i.e., for $\anyCod=\frechetCod$ or $\anyCod=D^*$,
            \begin{equation}
                \label{eq:algorithm:abstract:adjoint}
                0 = \anyCod G(\nextu,\thisx|\Delta\nexxt w_*)(w^{k+1}) +  (J'(\nextu) - \Delta u_{k+1}^*, -\tilde x_{k+1}^*)
                \quad\text{for some}\quad\Delta u_{k+1}^* \approx 0;
            \end{equation}
        \item an \term{outer algorithm}, which, given the previous iterates, produces $\nextx$ that approximately solves the third line, i.e.,
              \begin{equation}
                  \label{eq:algorithm:abstract:outer}
                  0 \approx \Delta x_{k+1}^* \in \tilde x_{k+1}^* + \subdiff R(\nextx).
              \end{equation}
    \end{enumerate}
\end{subequations}

The variable  $\tilde x^*_{k+1}$ estimates some element of $\anySubdiff\infmap{J \circ S}(\thisx)$ or $\anySubdiff\supmap{J \circ S}(\thisx)$.
Adopting the terminology of \cite{dizonvalkonen2024tracking}, we call it a \term{differential estimate}, despite possible lack of differentiability.
The perturbations $\Delta\nexxt w_* \approx 0$, $\Delta u^*_{k+1} \approx 0$, and $\Delta x_{k+1}^* \approx 0$ arise from different algorithms for the realisation of these steps, as demonstrated in the following examples.
We will make the allowed level of perturbation more precise in \cref{ass:tracking:main}.
In \eqref{eq:algorithm:abstract:adjoint}, for simplicity, and for reasons that will become apparent in \cref{sec:tv}, we only include inexactness in the first component. We could also include a perturbation to $\tilde x_{k+1}^*$, but, in fact, $\tilde x_{k+1}^*$ is already an estimate of $x_{k+1}^* \in \anyCOsubdiffS_2 F(\thisx)$ that solves the steps with $\Delta\nexxt w_* = 0$ and $\Delta u^*_{k+1} = 0$, i.e., the first two lines of \eqref{eq:algorithm:full-oc}.

\begin{example}[Outer forward-backward splitting]
    \label{ex:outer-fb}
    Assuming $X$ to be Hilbert, for a forward-backward algorithm, as an instance of \eqref{eq:algorithm:abstract:outer} we would solve
    $
        -(\nextx-\thisx)/\tau \in \mathscr{R}[\tilde x_{k+1}^* + \subdiff R(\nextx)]
    $
    for a step length parameter $\tau>0$.
    Here we recall that $\mathscr{R} \in \linear(X^*; X)$ is the Riesz map.
    This corresponds to
    \[
        \nexxt x = \prox_{\tau R}(\thisx - \tau\mathscr{R}\tilde x_{k+1}^*).
    \]
\end{example}

\begin{example}[Inner primal-dual proximal splitting, PDPS]
    \label{ex:inner-pdps}
    Let $U_p, U_d$ and $X$ be Hilbert spaces, $f \defeq f_0 + e$ for $f_0, e: U_p \times X \to \extR$ and $g: U_d \times X \to \extR$ that are convex, proper, and lower semicontinuous in their first parameter. Moreover, let $e$ have $L$-Lipschitz gradient with respect to its first parameter and $K\in \linear(U_p;U_d)$.
    We consider the problem
    \[
        \min_u f(u, x) + g(Ku; x),
    \]
    which can be equivalently written as the saddle point problem
    \[
        \min_{u_p\in U_p} \max_{u_d\in U_d^*} f(u_p; x) + \dualprod{u_d}{Ku_p} - g^*(u_d; x).
    \]
    Following \cite[Theorem 5.11]{he2012convergence,clason2020introduction} and writing $u=(u_p,u_d)\in U_p\times U_d^*$, this problem has primal-dual optimality conditions of the form $0\in G(u,x)$ with
    \begin{equation}
        \label{PDPS-A-and-B}
        G(u; x)
        :=
        \begin{pmatrix}
            \partial_{u_p} f_0(u_p;x) + \grad_{u_p} e(u_p;x) + K^*u_d
            \\
            \partial_{u_d} g^*(u_d;x) - Ku_p
        \end{pmatrix}
        .
    \end{equation}
    Defining for step length parameters $\tau_x,\tau_y>0$ the preconditioning operator
    \begin{equation}
        \label{PDPS-M}
        M :=
        \begin{pmatrix}
            \inv\tau_p\Id & - K^*
            \\
            - K           & \inv\tau_d\Id
        \end{pmatrix}
        ,
    \end{equation}
    and taking $\Delta w_*^{k+1}=-M(\nextu - \thisu)$, the inner algorithm \cref{eq:algorithm:abstract:inner} then becomes the primal-dual proximal splitting (PDPS) method of \cite{chambolle2010first}, here with an additional forward step with respect to $e$:
    \begin{equation}
        \label{alg:PDPS}
        \begin{cases}
            u_p^{k+1} & = \prox_{\tau_p f_0(\freevar; x)}(u_p^k - \tau_p(K^*u_d^k + \grad_{u_p} e(u_p^k;x)),
            \\
            u_d^{k+1} & = \prox_{\tau_d g^*(\freevar; x)} (u_d^k + \tau_d K(2u_p^{k+1} -  u_p^k)).
        \end{cases}
    \end{equation}
    The method converges (for fixed $x$) if $\tau_p L/2 + \tau_p\tau_d\norm{K}^2 \le 1$; see, e.g., \cite{clason2020introduction}.
\end{example}

\subsection{Tracking inequalities}
\label{sec:tracking}

We next discuss conditions that the inner and adjoint algorithms \cref{eq:algorithm:abstract:inner,eq:algorithm:abstract:adjoint} have to satisfy, for us to be able to prove the convergence of standard outer methods.
We assume that the inner and adjoint algorithms satisfy the following “tracking inequalities”, adapted from \cite{dizonvalkonen2024tracking,suonpera2024general}.
The idea in \cite{dizonvalkonen2024tracking} is to construct single-loop estimates $\nextEstF$ of differentials $F'(\thisx)$ of smooth functions $F=J \circ S$.
These tracking inequalities ensure that the accumulated errors remain controlled and that the approximate differential remains meaningful for descent.

To adapt the theory to our nonsmooth and multi-valued case, we replace $F'(\thisx)$ by a “target” $\targetF$ in one of the sets $\anyCOsubdiffS_i$ of \cref{rem:composition-subdiff}, modelling Fréchet or limiting subdifferentials or their relaxations, depending on the exactness of the calculus rules used in their derivation.
Thus, $0 \in \targetF + \subdiff R(\thisx)$ models a notion of (relaxed) criticality of $\thisx$ for the problem $\min(F+R)$.
We estimate this target by $\nextEstF$
as generated by the general algorithm \cref{eq:algorithm:abstract}.
We also work with single-valued selections $S_u(\thisx)$ of the inner solution mapping $S(\thisx)$, and a single-valued selection $S_w(\thisx)$ of the adjoint solution mapping $\breve Z(S_u(\thisx), \thisx, 0)$, where we set
\begin{equation}
    \label{eq:tracking:adjoint-solution-map}
    \breve Z(u, x, w_*)
    \defeq
    \{
    w\in W
    \mid
    0 \in \anyCod G(u,x|w_*)(w) +  (J'(u), -x^*),\,
    x^*\in X^*
    \}
    \quad\text{when}\quad
    w_* \in G(u,x).
\end{equation}

\begin{assumption}
    \label{ass:tracking:main}
    Let $X$, $X^*$, $U$, and $W_*$ be normed spaces, $W=(W_*)^*$, and both $\localset \subset X$ and $\Omega_U \subset U$.
    For some norms $\norm{\freevar}_*$ in $W_*$ and $\norm{\freevar}_{**}$ in $U^*$, define
    \begin{subequations}
        \label{eq:tracking:defs}
        \begin{align}
            \label{eq:tracking:du-def}
            \nextDistU
             &
            \defeq
            \norm{\nextu - S_u(\thisx)}_U + \norm{\Delta w_*^{k+1}}_*,
            \\
            \label{eq:tracking:dw-def}
            \nextDistW
             &
            \defeq
            \norm{w^{k+1} - S_w(\thisx)}_W + \norm{\Delta u_{k+1}^*}_{**},
            \quad
            \text{ and }
            \\
            \label{eq:tracking:dx-def}
            \thisDistXprev
             &
            \defeq
            \norm{\thisx - \prev x}_X.
        \end{align}
    \end{subequations}
    Then on each iteration $k \ge 0$, given $\{(x^n, u^n, w^n)\}_{n=0}^k \subset \localset \times \Omega_U \times W$ from the previous iterations, we are given:
    \begin{enumerate}[label=(\roman*)]

        \item\label{item:tracking:main:inner-tracking}
              \textbf{(Inner tracking property)}
              An \term{inner algorithm} that produces $\nextu \in U$ satisfying, when $k \ge 1$, for some $\pi_u>0$ and  $\kappa_u>1$ that
              \[
                  \kappa_u \nextDistU
                  \le
                  \thisDistU
                  + \pi_u\thisDistXprev.
              \]

        \item\label{item:tracking:main:adjoint-tracking}
              \textbf{(Adjoint tracking property)}
              An \term{adjoint algorithm} that produces $\nexxt w \in W$ satisfying, when $k \ge 1$, for some ${\primaldifffact}, \pi_w>0$ and $\kappa_w>1$ that
              \[
                  \kappa_w \nextDistW
                  \le
                  \thisDistW
                  + {\primaldifffact} \nextDistU
                  + \pi_w\thisDistXprev.
              \]

        \item\label{item:tracking:main:differential-transformation}
             \textbf{(Differential transformation)}
              A \term{differential transformation} that produces $\nextEstF \in X^*$ that, for a \term{target} $\targetF \in X^*$, satisfies for some $\alpha_u,\alpha_w \ge 0$ the bound
              \[
                  \norm{\nextEstF-\targetF}_{X^*}
                  \le
                  \alpha_u \nextDistU
                  + \alpha_w \nextDistW.
              \]
    \end{enumerate}
\end{assumption}

The condition \cref{item:tracking:main:inner-tracking} formalises what the rough inner algorithm \cref{eq:algorithm:abstract:inner} has to satisfy, and \cref{item:tracking:main:differential-transformation,item:tracking:main:adjoint-tracking} formalises what the rough adjoint algorithm \cref{eq:algorithm:abstract:adjoint} has to satisfy.
The inner and adjoint tracking conditions \cref{item:tracking:main:inner-tracking,item:tracking:main:adjoint-tracking} are parameter change aware contractivity conditions for the inner and adjoint algorithms:
if $\thisx=\prev x$, the former reduces to a standard contractivity condition. Both are trivial for exact inner and adjoint solutions $\nextu=S_u(\thisx)$ and $\nexxt w=S_w(\thisx)$.
The condition \cref{item:tracking:main:differential-transformation} allows converting the construction error of $\nextEstF$ to the tracking errors of the inner and adjoint algorithms.

We next discuss how \cref{ass:tracking:main} can be satisfied, then in \cref{thm:main-convergence} we show how it can be used to prove the convergence of an outer optimisation method \eqref{eq:algorithm:abstract:outer}, especially the forward-backward splitting of \cref{ex:outer-fb}.
The proof depends on scalar tracking results from \cite{dizonvalkonen2024tracking}, quoted in \cref{sec:scalar-tracking}.

\begin{remark}[Exact solutions and iterative solutions]
    \label{rem:tracking:exact}
    Both the adjoint and inner tracking properties \cref{ass:tracking:main}\,\cref{item:tracking:main:inner-tracking,item:tracking:main:adjoint-tracking} are always satisfied if we solve the adjoint and inner inclusion exactly, as then $\nexxt u \in S(\thisx)$ and  $\nexxt w \in Z(\thisx)$ so we can take $S_u(\thisx) = \nexxt u$ and $S_w(\thisx) = \nexxt w$.
    Given contractive inner and adjoint algorithms for the fixed parameter $\thisx$, we can also for any $\pi_u,\pi_w>0$ take sufficiently many iterations of those algorithms, until \cref{ass:tracking:main}\,\cref{item:tracking:main:inner-tracking,item:tracking:main:adjoint-tracking} are satisfied.
\end{remark}

\begin{remark}[Relaxed distances]
    It is possible to relax the distance in $X$ to any semi-norm, in which case the norm $\norm{\freevar}_{X^*}$ in \cref{item:tracking:main:differential-transformation} is replaced by its support function.
    We refer to \cite{dizonvalkonen2024tracking} for details.
\end{remark}

\subsection{Inner tracking property}
\label{sec:inner-tracking}

We now provide simpler component conditions to satisfy the inner tracking \cref{ass:tracking:main}\,\cref{item:tracking:main:inner-tracking}.
We split it into a regularity property of the solution mapping together with the simpler non-parametric \textbf{inner algorithm contractivity}
\begin{equation}
    \label{ineq:linear-inner-convergence}
    \kappa_u \dist(u^{k+1}, S(x^k)) + \kappa_u\norm{\Delta w_*^{k+1}}_* \le \dist(u^k, S(x^k)) \quad \text{for a}\quad \kappa_u > 1
\end{equation}
for some (semi-)norm $\norm{\freevar}_*$ on $W_*$.
In practical algorithms it is common to have
\[
    \kappa_u\norm{\Delta w_*^{k+1}}_*=\lambda\norm{\nextu-\thisu}_M
\]
for some $\lambda > 0$ and a self-adjoint positive (semi-)definite operator $M$.
This follows from the following theorems for forward-backward splitting and primal-dual proximal splitting.
In both results, the first factor $\kappa_u = 2\sqrt{(1+\rho)/(2+\alpha)} > 1$ if $0<\alpha<2$.
The second factor may be small, but can be absorbed into the factor $\lambda$ in the definition of the auxiliary norm $\norm{\freevar}_*$ employed in \eqref{ineq:linear-inner-convergence}.

We recall from \cite{aragon2008characterization} (a theme also treated in \cite[§28.2]{clason2020introduction}) that convex subdifferentials are metrically subregular when they satisfy a version of strong local subdifferentiability where the growth term is relaxed to be with respect to a set, instead of a single point.

\begin{theoremEnd}[category=inner-fb]{theorem}[Contractivity of forward-backward splitting under metric subregularity]
    \label{thm:subreg:convergence-result-sub-peb:fb}
    For convex functions $g, f: U \to \extR$ on a Hilbert space $U$, where $f$ has an $L$-Lipschitz gradient, suppose $T \defeq \subdiff g + \grad f$ is metrically subregular \emph{at a $\optu\in\inv T(0)$ for $\hat w=0$} with the factor $\kappa>0$ and radius $\delta>0$.
    Given $\thisu \in \B(\optu, \delta)$, let $\nextu \defeq \prox_{\tau g}(\thisu-\tau\grad f(\thisu))$ be generated through the forward-backward method with $\tau>0$ satisfying  $L\tau \le 2(1-\epsilon)$ for an $\epsilon>0$.
    Then  $\nextu \in \B(\optu, \delta)$, and, for $\rho=(1-L\tau/2 - \epsilon)\tau^2/(2\kappa^2(1+L^2\tau^2))>0$ and any $\alpha>0$, we have
    \begin{equation}
        \label{eq:ubreg:convergence-result-sub-peb:fb}
        2\sqrt{\frac{1+\rho}{2+\alpha}}\dist(\nextu, \inv T(0))
        +
        2\sqrt{\frac{\epsilon}{2+\inv\alpha}}\norm{\nextu-\thisu}
        \le
        \dist(\thisu, \inv T(0)).
    \end{equation}
\end{theoremEnd}

\begin{proofE}
    We take $\Xi=0$ and $M=\inv\tau\Id$ in \cref{thm:subreg:convergence-result-sub-peb:general}, observing that metric subregularity (with respect to the standard norms) with the modulus $\kappa$ translates to metric subregularity with respect to the $M$ and $\inv M$ norms with modulus $\kappa_M=\kappa/\tau$ and radius $\delta/\tau$.
    Likewise, the Lipschitz property of $\grad f$ translates to the one with respect to these norms with factor $L_M=L\tau$.
\end{proofE}

\begin{theoremEnd}[category=inner-pdps]{theorem}[Contractivity of primal-dual proximal splitting under metric subregularity]
    \label{thm:subreg:convergence-result-sub-peb:pdps}
    Let $G$ and $M$ be as in \cref{ex:inner-pdps}.
    For a fixed $x$, suppose $T=G(\freevar; x)$ is metrically subregular with respect to the $M$ and $\inv M$-norms at $\opt u$ for $\opt w=0$ with the modulus $\kappa_M>0$ and radius $\delta>0$\ifAtEnd. \else, i.e.,
    \[
        \dist_M(u, \inv T(0)) \le \kappa_M \dist_{\inv M}(0, T(u))
        \quad (u \in \B_M(\optu, \delta)).
    \]
    \fi
    Given $\thisu = (\this u_p, \this u_d) \in \B_M(\optu, \delta)$, let $\nextu$ be generated through the PDPS \eqref{alg:PDPS} with $\tau_p,\tau_d>0$ satisfying $\tfrac{1}{2(1-\epsilon)}L\tau_p + \tau_p\tau_d\norm{K}^2 \le 1$ for an $\epsilon \in (0, 1)$. Then $\nextu \in \B_M(\optu, \delta)$ and \ifAtEnd \eqref{eq:subreg:contractivity} holds for \else
    \[
        2\sqrt{\frac{1+\rho}{2+\alpha}}\dist_M(\nextu, \inv T(0))
        +
        2\sqrt{\frac{\epsilon}{2+\inv\alpha}}\norm{\nextu-\thisu}_M
        \le
        \dist_M(\thisu, \inv T(0))
    \]
    for $\rho=(1-L_M/2 - \epsilon)/(2(1+L_M^2)\kappa_M^2)>0$,
    \fi
    $L_M = L\tau_p/(1-\tau_p\tau_d\norm{K}^2)$ and any $\alpha>0$.
\end{theoremEnd}

\begin{proofE}
    We can write the PDPS \eqref{alg:PDPS} as \eqref{eq:subreg:fb} with $\Xi =
        \begin{psmallmatrix}
            K & 0
            \\
            -K & 0
        \end{psmallmatrix}
        ,
    $
    and the liftings $\bar f(u) = (e(u_p; x), 0)$ and  $\bar g(u) = (f_0(u_p; x), g^*(u_d; x))$.
    We have for $N=\inv{(\Id - \tau_d\tau_p KK^*)}$ that
    \[
        \inv M =
        \begin{pmatrix}
            A_{11}            & \tau_p \tau_d K^* N
            \\
            \tau_p \tau_d N K & \tau_d N
        \end{pmatrix}
        \quad\text{for}\quad
        A_{11} \defeq \tau_p \Id + \tau_p^2 \tau_d K^* N K.
    \]
    We have
    \[
        A_{11}
        \le
        \tau_p[1 + \tau_p\tau_d\norm{K}^2/(1-\tau_p\tau_d\norm{K}^2)]\Id
        =
        [\tau_p/(1-\tau_p\tau_d\norm{K}^2)]\Id.
    \]
    We also have
    $
        \norm{u}_M^2 \ge \inv\tau_p(1-\tau_p\tau_d\norm{K})\norm{u_p}^2.
    $
    Since $e$ has an $L$-Lipschitz gradient, it follows
    \[
        \begin{split}
            \norm{\grad \bar f(u) - \grad \bar f(\tilde u)}_{\inv M}^2
             &
            =\norm{\grad e(u_p; x) - \grad e(\tilde u_p; x)}_{A_{11}}^2
            \\
             &
            \le
            [\tau_p/(1-\tau_p\tau_d\norm{K}^2)]L^2\norm{u_p-\tilde u_p}^2
            \le
            [\tau_pL/(1-\tau_p\tau_d\norm{K}^2)]^2\norm{u-\tilde u}_M^2.
        \end{split}
    \]
    That is, the lifted function $\bar f$ is $L_M$-Lipschitz in the $\inv M$ and $M$ norms.

    The step length condition guarantees $\tau_p\tau_d\norm{K}^2 <1 $, hence by \cite[Lemma 9.12] {clason2020introduction} the positive definiteness (in addition to self-adjointness by construction) of $M$.
    The result now follows from \cref{thm:subreg:convergence-result-sub-peb:general}.
\end{proofE}

\begin{remark}[Compatibility of neighbourhoods and moduli]
    \label{rem:subreg:chaining}
    To apply \cref{thm:subreg:convergence-result-sub-peb:pdps,thm:subreg:convergence-result-sub-peb:fb} (with $T=G(\freevar, \thisx)$, hence $\inv T(0)=S(\thisx)$) to proving \eqref{ineq:linear-inner-convergence}, it should be verified, in a case-by-case basis for each problem of interest that:
    \begin{enumerate}[label=(\roman*),nosep]
        \item The moduli of metric subregularity, hence the factors $\kappa_u$, can be chosen uniformly over $k \in \N$.
        \item The neighbourhoods of metric subregularity of $G(\freevar, \thisx)$ and $G(\freevar, \nextx)$ are compatible: $\nextu \in B(\this{\opt u}, \delta_k)$ guarantees $\nextu \in B(\opt u^{k+1}, \delta_{k+1})$, where $0 \in G(\this{\opt u}, \thisx)$ and $\delta_k>0$ is the corresponding radius of metric subregularity.
    \end{enumerate}
    The theorems guarantee that $\nextu \in B(\this{\opt u}, \delta_k)$ when $\thisu \in B(\this{\opt u}, \delta_k)$, but since the parametric problem changes on each iteration, the neighbourhood of metric subregularity may change.
    Obviously, it is enough that the neighbourhoods are global, $\delta_k \equiv \infty$, such as when $G(\freevar, \thisx)$ is strongly monotone with a uniform factor; see \cref{lemma:subreg:monotonicity-subregularity,lemma:subreg:monotonicity-subregularity:basic}.
\end{remark}


We can now split the inner tracking property into several simpler conditions:
the inner algorithm contractivity that we have just treated, a Lipschitz-like condition on the inner solution mapping, and the optimality of the inner selection.

\begin{theorem}
    \label{thm:inner-tracking}
    Assume the setup of \,\cref{sec:implicit-abstract-algorithm}.
    Let $u^{k+1}\in U$ generated by \cref{eq:algorithm:abstract:inner} 
    satisfy the \textbf{inner algorithm contractivity} \eqref{ineq:linear-inner-convergence}.
    Suppose the \textbf{inner solution mapping one-sided Lipschitz-like property}
    \begin{equation}
        \label{eq:inner-tracking:lip-cond}
        \dist(S(x^k), \thisu)
        \le
        \dist(S(x^{k-1}), \thisu)
        +
        \pi_u
        \norm{x^{k-1} - x^k}_X
    \end{equation}
    for some constants $\pi_u > 0, \kappa_u>1$,
    and that $S_u(\thisx) \in S(\thisx)$ is such that this \textbf{inner selection is optimal},
    \begin{equation*}
        \dist(S(\thisx), \nextu) = \norm{\nextu - S_u(\thisx)}_U
        \quad
        \text{ for all }\,
        k\in\N.
    \end{equation*}
    Then \cref{ass:tracking:main}\,\cref{item:tracking:main:inner-tracking} holds.
\end{theorem}

\begin{proof}
    For the proof, recall the abbreviations \eqref{eq:tracking:defs}.
    We combine  \cref{eq:inner-tracking:lip-cond,ineq:linear-inner-convergence} to obtain
    \[
        \begin{split}
            \kappa_u \nextDistU
             &
            =
            \kappa_u\norm{\nextu - S_u(\thisx)}_U + \kappa_u\norm{\Delta w_*^{k+1}}_*
            \\
             &
            \le
            \dist(u^k, S(x^k))
            \\
             &
            \le
            \dist(S(x^{k-1}), \thisu)
            +
            \pi_u
            \norm{x^{k-1} - x^k}_X
            \\
             &
            =
            \norm{\thisu - S_u(\prev x)}_U
            +
            \pi_u
            \norm{x^{k-1} - x^k}_X
            \\
             &
            \le
            \norm{\thisu - S_u(\prev x)}_U
            +
            \norm{\Delta w_*^{k}}_*
            +
            \pi_u
            \norm{x^{k-1} - x^k}_X
            \\
             &
            =
            \thisDistU
            +
            \pi_u
            \thisDistXprev.
            \qedhere
        \end{split}
    \]
\end{proof}

It can be shown that the Pompeui–Hausdorff distance between sets $A$ and $B$,
\[
    \dist(A, B)
    =
    \sup_{u \in U} \abs{\dist(A, u) - \dist(B, u)}.
\]
Therefore, in the worst case, if the inner iterates $\thisu$ are unbounded, \eqref{eq:inner-tracking:lip-cond} becomes a Pompeui--Hausdorff--Lipschitz property. However, usually, we can assume that $\thisu$ belongs to a bounded set $\Omega_U$, relaxing the requirement.
The next lemma shows that  \eqref{eq:inner-tracking:lip-cond} arises from the Aubin property of $S$ in a suitable set. However -- in the bounded setting -- it is conceivable that there are cases where \eqref{eq:inner-tracking:lip-cond} holds but the Aubin property of $S$ does not hold.

\begin{lemma}
    \label{lemma:inner-tracking:aubin}
    Suppose $S: X \setto U$ has the Aubin property in $\Omega_X$ and $\Omega_U$ with the factor $\pi_u>0$ in the form\footnote{For spherical neighbourhoods $\Omega_X=B(\bar x, \delta)$, the equivalence of this form to the the standard form $\dist(S(x), \tilde u) \le \kappa_u\dist(\inv S(\tilde u), x)$ is proved in \cite[Theorem 27.3]{clason2020introduction}. This latter superficially stricter form always implies the one we assume.}
    \[
         \dist(S(x), \tilde u) \le \pi_u \dist(\inv S(\tilde u) \isect \Omega_X, x)
         \quad\text{for all}\quad
         x \in \Omega_X \text{ and } \tilde u \in \Omega_U,
    \]
    then, for any $u \in U$, $x, \tilde x \in \Omega_X$,
    \[
        \dist(S(x), u)
        \le
        \dist(S(\tilde x) \isect \Omega_U, u)
        +
        \pi_u
        \norm{x - \tilde x}.
    \]
\end{lemma}
\begin{proof}
    Expanding the definitions of $\dist$, we have
    \[
        \begin{split}
        d
        &
        \defeq
        \dist(S(x), u)
        -
        \dist(S(\tilde x) \isect \Omega_U, u)
        =
        \sup_{\tilde u \in S(\tilde x) \isect \Omega_U} \inf_{u' \in S(x)} \norm{u'-u}_U - \norm{\tilde u-u}_U
        \\
        &
        \le
        \sup_{\tilde u \in S(\tilde x)  \isect \Omega_U} \inf_{u' \in S(x)} \norm{u' - \tilde u}_U
        =
        \sup_{\tilde u \in S(\tilde x)  \isect \Omega_U} \dist(S(x), \tilde u).
        \end{split}
    \]
    Now, since $x \in \Omega_X$, and $S(\tilde x)  \isect \Omega_U \subset \Omega_U$, the Aubin property gives
    \[
        d
        \le
        \sup_{\tilde u \in S(\tilde x)  \isect \Omega_U} \pi_u \dist(\inv S(\tilde u) \isect \Omega_X, x)
        \le
        \pi_u \norm{\tilde x - x}.
        \qedhere
    \]
\end{proof}

\begin{remark}[Single-valued Lipschitz solution mappings]
    In particular, if $S$ is single-valued and Lipschitz, then \cref{eq:inner-tracking:lip-cond} holds.
    It is not uncommon that a primal-only inner solution mapping is single-valued and Lipschitz, compare \cite[Theorem 2]{villacis2025variational} and \cite[Theorem 4.51]{bonnans2000perturbation}, or combine \cite[Theorem 28.6]{clason2020introduction} with a strict convexity assumption. 
\end{remark}


\subsection{Adjoint tracking property}
\label{sec:adjoint-tracking}

Similarly to the inner tracking property, we split the adjoint tracking property into regularity properties of the solution mappings together with the simpler non-parametric contractivity condition on the algorithm.
Recall the definition of $\breve Z$ in \eqref{eq:tracking:adjoint-solution-map}.

\newcommand{\Zunext}{Z(\nextu, \thisx, \Delta w_*^{k+1})}
\newcommand{\anyZunext}{\breve{Z}(\nextu, \thisx, \Delta w_*^{k+1})}
\newcommand{\Zutilde}{Z(\tilde u, \thisx, \Delta \tilde w_*)}
\newcommand{\ZunextInd}{Z(\nextu,\, \thisx,\, \Delta w_*^{k+1})}
\newcommand{\ZutildeInd}{Z(\tilde u,\, \thisx,\, \Delta \tilde w_*)}

\begin{theorem}
    \label{thm:adjoint-ZZ}
    Assume the setup of \cref{sec:implicit-abstract-algorithm}.
    Let $\breve Z: U \times X \times W_* \setto W,$
    abbreviating $Z(x) \defeq \breve Z(S_u(x),x, 0)$, satisfy for some $\pi_w, \mu_u>0 $ and $\kappa_w > 1$ 
    the \textbf{adjoint solution mapping one-sided bi-Lipschitz-like property}
    \begin{equation}
        \label{ineq:adjoint-ZZ}
        \begin{split}
            \kappa_w \dist( Z(x^k), \nexxt w) & - \dist(Z(x^{k-1}), \this w)
            \\
                                                      &
            \le
            \kappa_w \dist(\anyZunext, \nexxt w) - \dist(\anyZunext, \this w)
            \\
            \MoveEqLeft[-1]
            +
            \pi_w\norm{x^k - x^{k-1}}
            + \mu_u\left(
            \norm{\nextu - S_u(\thisx)}_U + \norm{\Delta \nexxt w}_*
            \right).
        \end{split}
    \end{equation}
    Moreover, suppose $w^{k+1}\in W$ generated by \cref{eq:algorithm:abstract:adjoint} satisfy \textbf{adjoint algorithm contractivity}
    \begin{equation}
        \label{ineq:linear-adjoint-convergence-lem}
        \kappa_w \dist(\anyZunext, w^{k+1}) + \kappa_w \norm{\Delta u_{k+1}^*}_{**} \le \dist(\anyZunext, w^k)
    \end{equation}
    and that the \textbf{adjoint selection is optimal}, i.e., $S_w(\thisx) \in Z(\thisx)$ is such that
    \begin{equation}
        \label{eq:linear-adjoint-convergence-lem:selection}
        \dist(Z(\thisx), w^{k+1}) = \norm{w^{k+1} - S_w(\thisx)}_W
        \quad
        \text{ for all }\,
        k\in\N.
    \end{equation}
    Then \cref{ass:tracking:main}\,\cref{item:tracking:main:adjoint-tracking} holds.
\end{theorem}

\begin{proof}
    For the proof, recall the abbreviations \eqref{eq:tracking:defs}.
    Combining \cref{ineq:adjoint-ZZ} with \eqref{ineq:linear-adjoint-convergence-lem}, we obtain
    \[
        \kappa_w \norm{\Delta u_{k+1}^*}_{**}
        +
        \kappa_w \dist( Z(x^k), \nexxt w) - \dist(Z(x^{k-1}), \this w)
        \le
        \pi_w\norm{x^k - x^{k-1}}
        + \mu_u\nextDistU.
    \]
    Using \cref{eq:tracking:dw-def,eq:tracking:dx-def,eq:linear-adjoint-convergence-lem:selection} we find this equivalent to
    $
        \kappa_w \nextDistW
        \le
        \thisDistW
        + \mu_u\nextDistU
        +
        \pi_w\thisDistXprev,
    $
    i.e., \cref{ass:tracking:main}\,\cref{item:tracking:main:adjoint-tracking}.
\end{proof}

As in the case of  \eqref{eq:inner-tracking:lip-cond}, if $\this w$ and $\nexxt w$ are unbounded, the respective distance differences in \eqref{ineq:adjoint-ZZ} can be reformulated as Pompeui--Hausdorff distances. Then the property holds if $\breve Z$ is Pompeui--Hausdorff--Lipschitz, and the inner selection $S_u$ is Lipschitz.
Next lemma shows, that, if the latter holds, then similarly to \cref{lemma:inner-tracking:aubin}, the condition \eqref{ineq:adjoint-ZZ} can also be verified through two applications of the Aubin property. However, there are cases where condition \eqref{ineq:adjoint-ZZ} holds but the assumptions of next lemma does not hold. We encounter such an example in \cref{sec:simple-example}.
In particular, the Lipschitz assumption on $S_u$, where $S_u(x) \subset S(x)$, in practise forces $S$ to be single-valued. This is not the case with \eqref{ineq:adjoint-ZZ}.

\begin{lemma}
    \label{lemma:adjoint-tracking:aubin}
    Suppose  $U, W_*$ and $X$ are normed spaces, 
    $W=(W_*)^*.$ Let $J: U \to \R$ and $G: U\times X \setto W_*.$
    Moreover, let $S_u: X \to U$ be $L$-Lipschitz and satisfy $0\in G(S_u(x), x)$ for all $x\in X,$
    $\breve Z: U \times X \times W_* \setto W$ be given by \cref{eq:tracking:adjoint-solution-map} and
    abbreviate $Z(x) \defeq \breve Z(S_u(x),x, 0).$ Assume $\breve Z$  has the Aubin property in $\Omega
    \defeq \{(u,x,w_*)\in  \Omega_U \times \Omega_X \times W_*\mid w_*\in G(u,x)\}$
    and $\Omega_W$ with the factor $C_w>0$ in the form 
    \[
    \dist(\breve Z(u,x,w_*), \tilde w) \le C_w \dist(\inv{\breve Z}(\tilde w) \isect \Omega, (u,x,w_*))
    \quad\text{for all}\quad
    (u,x,w_*) \in \Omega \text{ and } \tilde w \in \Omega_W.
    \]
    Then, for any $w, \bar w \in W, u\in \Omega_U, \tilde x \in \Omega_X, x\in \Omega_X \isect \inv S_u(\Omega_U), \kappa_w>0$ and $\pi_w\defeq C_w(1+L), \mu_u\defeq C_w(1+\kappa_w)$
    \begin{equation}
        \label{ineq:adjoint-ZZ-lemma}
        \begin{split}
            \kappa_w \dist( Z(x), w) & - \dist(Z(\tilde x)\isect \Omega_W, \bar w)
            \\
            &
            \le
            \kappa_w \dist(\breve Z(u, x, w_*)\isect \Omega_W, w) - \dist(\breve Z(u, x, w_*), \bar w)
            \\
            \MoveEqLeft[-1]
            +
            \pi_w\norm{x - \tilde x}
            + \mu_u\left(
            \norm{u - S_u(x)}_U + \norm{w_*}_*
            \right).
        \end{split}
    \end{equation}
\end{lemma}
\begin{proof}
    Again expanding the definitions of $\dist$ like in the proof of \cref{lemma:inner-tracking:aubin}, we have
    \begin{align*}
    d_1
    &
    \defeq
    \dist(Z(x), w)
    -
    \dist(\breve Z(u, x, w_*) \isect \Omega_W, w)
    =
    \sup_{\tilde w \in \breve Z(u, x, w_*) \isect \Omega_W} \inf_{w' \in Z(x)} \norm{w'-w}_W - \norm{\tilde w-w}_W
    \\
    &
    \le
    \sup_{\tilde w \in \breve Z(u, x, w_*)\isect \Omega_W} \inf_{w' \in Z(x)} \norm{w' - \tilde w}_W
    =
    \sup_{\tilde w \in \breve Z(u, x, w_*) \isect \Omega_W} \dist(\breve Z(S_u(x), x, 0), \tilde w)
    \\
    \shortintertext{and}
        d_2
        &
        \defeq
        \dist(\breve Z(u, x, w_*), \bar w)
        -
        \dist(Z(\tilde x)\isect \Omega_W, \bar w)
         =
        \sup_{\tilde w \in  Z(\tilde x) \isect \Omega_W} \inf_{w' \in \breve Z(u, x, w_*)} \norm{w'-\bar w}_W - \norm{\tilde w-\bar w}_W
        \\
        &
        \le
        \sup_{\tilde w \in  Z(\tilde x) \isect \Omega_W} \inf_{w' \in \breve Z(u, x, w_*)} \norm{w' - \tilde w}_W
        =
        \sup_{\tilde w \in \breve Z(S_u(\tilde x), \tilde x, 0) \isect \Omega_W} \dist(\breve Z(u, x, w_*), \tilde w).
        \end{align*}
    Now, since $x \in \Omega_X \isect \inv S_u(\Omega_U), u\in \Omega_U$, and $\tilde w \in \Omega_W$ in both cases, the Aubin property gives
    \begin{align}
        \label{aling-row-1}
        \kappa_w d_1
        &
        \le
        \kappa_w  \sup_{\tilde w \in \breve Z(u, x, w_*) \isect \Omega_W} C_w \dist(\inv{\breve Z}(\tilde w) \isect \Omega, (S_u(x),x,0))
        \le
        \kappa_w C_w \left(
            \norm{u - S_u(x)}_U + \norm{w_*}_*
        \right)
        \\
        \shortintertext{and}
        \label{aling-row-2}
        d_2
        &
        \le
        \sup_{\tilde w \in \breve Z(S_u(\tilde x), \tilde x, 0) \isect \Omega_W} C_w \dist(\inv{\breve Z}(\tilde w) \isect \Omega, (u,x,w_*))
        \\
        \nonumber
        &
        \le
        C_w
        \left(
             \norm{u - S_u(\tilde x)}_U +
             \norm{\tilde x - x}
             +
             \norm{w_*}_*
        \right).
    \end{align}
    The claim follows from summing up and rearranging \cref{aling-row-1,aling-row-2}, and applying Lipschitz-continuity of $S_u$
    \[
      \norm{u - S_u(\tilde x)}_U
      \le
      \norm{u - S_u(x)}_U
      +
      \norm{S_u(x) - S_u(\tilde x)}_U
      \le
      \norm{u - S_u(x)}_U
      +
      L\norm{x - \tilde x}
      .
       \qedhere
    \]
\end{proof}

\subsection{Differential transformation}
\label{sec:differential-transformation}

To motivate the next construction, recall that in the algorithm of \cref{sec:implicit-abstract-algorithm}, specifically \eqref{eq:algorithm:abstract:adjoint},
the differential estimate $\tilde x_{k+1}^*$ arises jointly with the adjoint variable $\nexxt w$ from the adjoint inclusion
\begin{equation}
    \label{eq:exact-adjoint}
    0 \in \anyCod G(\nexxt{u},\thisx|\Delta w^{k+1}_*)(w^{k+1}) +  (J'(\nexxt{u}) - \Delta u_{k+1}^*, -\tilde x_{k+1}^*)
\end{equation}
for either $\anyCod=\frechetCod,D^*$.
We can equivalently write this as
\begin{equation}
    \label{eq:adjoint-tk}
    \tilde x_{k+1}^* \in \breve T(\nexxt{u}, \nexxt w, \Delta u_{k+1}^*, \Delta w^{k+1}_*; \thisx),
\end{equation}
where $\breve T: U \times W \times U^* \times W_* \to X^*$,
\begin{equation*}
    \breve T(u,w,\Delta u^* \Delta w_*; x) \defeq \{ x^* \in X^* \mid 0 \in \anyCod G(u,x|\Delta w_*)(w) +  (J'(u)- \Delta u^*, -x^*) \}.
\end{equation*}
Likewise, we pick the target $x_{k+1}^* \in \anyCOsubdiffS_2 F(\thisx)$ such that
\begin{equation}
    \label{eq:adjoint-t-exact}
    x_{k+1}^* \in \breve T(S_u(\thisx), S_w(\thisx), 0, 0; \thisx).
\end{equation}
Such a choice exists if there exists an adjoint solution $S_w(\thisx) \in \breve Z(\thisx) \defeq \breve Z(S_u(\thisx), \thisx, 0)$.
(We do not have this implicit expression for $\tilde x_{k+1}^*$, as $\breve Z$ lacks the perturbation parameter $\Delta u_{k+1}^*$.)
We can, thus, use Lipschitz-like properties of $\breve T$ to construct the target $\targetF$, and to ensure \cref{ass:tracking:main}\,\cref{item:tracking:main:differential-transformation}:

\begin{theorem}
    \label{thm:tracking:diff-transformation}
    Let $U, X$ and $W_*$  be normed spaces, $W=(W_*)^*, \Omega_X\subset X, \Omega_U\subset U$, and   $G:U\times X \setto W_*$.
    Let $S_u$ and $S_w$ be the single-valued selections of the inner and adjoint solution mapping as in \cref{sec:tracking}.
    Assume that there exist exact and algorithmic adjoint solutions, i.e., \eqref{eq:adjoint-t-exact} and \eqref{eq:adjoint-tk}.
    Moreover, assume whenever $\thisx \in \Omega_X$ and $\nextu \in \Omega_U$ that $T(\freevar; \thisx)$ is inner Lipschitz at $v \defeq (\nextu, \nexxt w, \Delta u_{k+1}^*, \Delta w^{k+1}_*)$ with the factor $L_T$ and radius $\rho>0$ such that $(S_u(\thisx), S_w(\thisx), 0, 0) \in \B(v, \delta)$, more precisely,
    \begin{multline}
        \label{eq:tracking:diff-transformation:lipschitz}
        \dist(\breve T(S_u(\thisx), S_w(\thisx), 0, 0; \thisx), \tilde x_{k+1}^*)
        \\
        \le
        L_T( \norm{S_u(\thisx)-\nextu}_U
        + \norm{\Delta w^{k+1}_*}_*
        + \norm{S_w(\thisx)-\nexxt w}_W
        + \norm{\Delta u_{k+1}^*}_{**}
        ).
    \end{multline}
    Then \cref{ass:tracking:main}\,\cref{item:tracking:main:differential-transformation} holds for any $\alpha_u=\alpha_w>L_T$ and some $x_{k+1}^* \in \breve T(S_u(\thisx), S_w(\thisx), 0; \thisx)$.
\end{theorem}

\begin{proof}
    The claim is immediate from the assumed Lipschitz-like property and the existence of the exact and algorithmic adjoint solutions.
    (We take any $\alpha_u=\alpha_w>L_T$ to avoid having to assume any compactness of $\breve T(S_u(\thisx), S_w(\thisx), 0, 0; \thisx)$).
\end{proof}

\begin{remark}[Existence of solutions]
    \label{rem:tracking:adjoint-existence}
    Subject to the conditions of  \cref{thm:fullsystem:limiting}, for the limiting coderivative, \cref{eq:adjoint-t-exact} holds for a choice of $S_u(\thisx)$ when $\thisx$ (locally) minimises the tilted functional $\infmap{J \circ S} + x^*$ for some $R=x^* \in X^*$.
    Likewise, \cref{eq:adjoint-tk} holds when $\thisx$ (locally) minimises $\infmap{[J + \dir u_{k+1}^*]\circ S} + x^*$ for the inner problem $0 \in G(u, x) -\dualprod{\Delta w_*^{k+1}}{u}$.
\end{remark}

\subsection{A forward-backward type outer method}
\label{sec:inexact-convergence}

We now prove the convergence of a forward-backward type outer method, assuming inner and adjoint methods that satisfy \cref{ass:tracking:main}.
The next convergence result does not depend on our specific problem structure; it applies to general problems of the form
\begin{equation}
    \label{eq:fb:problem}
    \min_{x \in X} E(x) + R(x),
\end{equation}
for which \cref{ass:tracking:main} holds with $R: X \to \extR$ is convex and possibly nonsmooth, and $E: X \to \R.$
Example choices for $E$ are the optimistic and pessimistic selections $\infmap{J \circ S}$ and $\supmap{J \circ S}$, which
are defined in \cref{sec:calculus}, with $J:U\to\R$ and a solution mapping $S$ as in \cref{eq:oc:solution-mapping}.

Given a step length parameter $\tau>0$, we take as the (inexact) outer algorithm \cref{eq:algorithm:abstract:outer} the forward-backward splitting
\begin{equation}
    \label{inexact-step}
    \nextx
    \defeq
    \argmin_{z\in X}\left\{
    \tau\dualprod{\nextEstF}{z-\thisx} + \tau R(z) + \frac{1}{2}\norm{z-\thisx}^2
    \right\},
\end{equation}
In Hilbert spaces 
this reduces to the familiar
\begin{equation*}
    \label{inexact-step-Hilbert}
    \nextx \defeq \prox_{\tau R}(x^k - \tau \mathscr{R}\nextEstE).  
\end{equation*}

In our analysis, we do not use any specific coderivative or subdifferential of $E$, instead introducing a “target set” $\COsubdiffS E(\thisx) \ni \targetE$, where we recall that $\targetE$ is the tracking target from \cref{ass:tracking:main}.
Informally, $0 \in [\COsubdiffS E +\subdiff R](\bar x)$ should be a meaningful optimality condition for \eqref{eq:fb:problem}.
We have discussed meaningful choices of $\COsubdiffS E$ of $E=J \circ S$ in \cref{rem:composition-subdiff}.
A trivial choice is $\COsubdiffS E(x) = \{ \targetE\}$, although we will impose further technical restrictions that need to be satisfied.
In particular, we require the set-valued descent inequality
\begin{equation}
    \label{ineq:set-valued-descent-inexact}
    \inf_{x^* \in \COsubdiffS E(x^k)} \dualprod{x^*}{x-x^k} \ge E(x) - E(x^k) - \frac{L}{2}\norm{x - x^k}^2
    \quad \text{for all }  k\in\N \text{ and } x\in \Dom R.
\end{equation}
This property -- that obviously benefits from a smaller choice of $\COsubdiffS E(x^k)$ -- reduces to the standard descent inequality for functions with $L$-Lipschitz continuous Fréchet differential \cite[Lemma 7.1]{clason2020introduction} by taking $\COsubdiffS E(x^k)=\{E'(\thisx)\}$. In \cref{sec:simple-example} we study a positive nonsmooth example.

With all of the above constructions and assumptions at hand, we can finally show the convergence of subdifferentials to zero.
The result is global -- even for nonconvex $E$ -- if we can take $\Omega_X=X$ in \cref{ass:tracking:main}.
It does not say anything about the existence of limiting points of $\{\thisx\}_{k \in N}$, or their optimality.
However, the result of the theorem proves that under appropriate compactness and outer semicontinuity assumptions, any limiting point of the sequence of iterates will satisfy the relevant condition discussed in \cref{rem:composition-subdiff}. We prove first a lemma, and then the convergence theorem.

\begin{lemma}
    \label{lemma:main-convergence}
    Let $X$ and $U$ be normed spaces, $E:X \to \R$, and $R:X\to\extR$ convex, proper and lower semicontinuous.
    Given an initial iterate $x^0 \in \Omega_X, N\ge 1$ and $\tau>0$, generate $\{x^k\}_{k=1}^{N+1}$ through \cref{inexact-step}, choosing for each iteration corresponding target sets $\COsubdiffS E(x^k)$ for which \cref{ineq:set-valued-descent-inexact} holds for an $L > 0$, as well as differential estimates $\nextEstE$ 
    for which \cref{ass:tracking:main} holds until $N$.
    Assume, moreover, that $\Omega_X$ contains the $([E+R](x^0) + \Psi_1)$-sublevel set of $E+R$, i.e.,
    \begin{equation}
        \label{eq:main-convergence:locality-lemma}
        \{ x \in X \mid [E+R](x) \le  [E+R](x^0) + \Psi_1 \} \subset \Omega_X.
    \end{equation}
    Then, for each $k\in\{0,...,N+1\},$ we have $\thisx \in \Omega_X$ and
    \begin{equation}
        \label{ineq:finite-series-inexact}
        [E+R](x^{k})
        +
        \tau\left(1 - \frac{\tau (L+1+\trackingressum[1]^2)}{2}\right)
        \sum_{j=0}^{k-1}
        \norm{
            \widetilde x^*_{j+1}
             + q^{j+1}}^2
        \le
        [E+R](x^0)
        + \sum_{j=0}^{k-1}e_j
    \end{equation}
    for some $q^{j+1}\in \subdiff R(x^{j+1})$ and $e_j$, ($j=0,\ldots,N$), satisfying
    \begin{equation}
        \label{eq:main-convergence:e}
        \max_{k \in \{0,\ldots,N\}}\sum_{j=0}^{k}e_j
        \le
        \Psi_1
        \defeq
        \frac{\initDistUsq}{\pi_u} \bigg(\frac{\trackingressum[1]\alpha_u\kappa}{\kappa-1} + \frac{\trackingressum[1]\alpha_w{\primaldifffact}}{(\kappa-1)^2}\bigg)
        +
        \frac{\initDistWsq}{\pi_w} \bigg(\frac{\trackingressum[1]\alpha_w\kappa}{\kappa-1}\bigg),
    \end{equation}
    where $\kappa \defeq\min\{\kappa_u, \kappa_w\}>1$, and $\MAX\kappa \defeq\max\{\kappa_u, \kappa_w\}>1$, the first-step distances $\initDistU$ and $\initDistW$ are given in \eqref{eq:tracking:defs}, and the constant
    \begin{equation}
        \label{eq:tracking:ressum}
        \trackingressum[1]
        \le
        \frac{(\alpha_u\pi_u+\alpha_w\pi_w)\kappa\overline\kappa}{\kappa-1}
        + \frac{\alpha_w {\primaldifffact}\pi_u\overline\kappa}{(\kappa-1)^2}.
    \end{equation}

    Finally, suppose, additionally, that the step length $\tau$ satisfies $\tau (L+1+\trackingressum[1]^2) < 2$ and $\inf [E+R]>-\infty.$ Then $\sup_{k\le N}\sum_{j=0}^{k-1}\norm{x^{j+1} - x^j}^2<\infty.$
    If additionally $\sup_{x\in\Omega_X}\norm{S_u(x)}_U < \infty,$  then also $\sup_{k \le N}\norm{\nextu}_U < r_U <\infty,$
    with
    \begin{equation}
        \label{eq:rU-def}
        r_U
        \defeq
        \sup_{x\in\Omega_X}\norm{S_u(x)}_U
        +
        \norm{u^1- S(x^0)}_U
        +
        \norm{\Delta w_*^1}_*
        +
        \pi_u
        \sqrt{
            \frac{2\tau ([E+R](x^0)-\inf [E+R]+\Psi_1)}{(1 - \kappa_u^{-2})\left(2 - \tau (L+1+\trackingressum[1]^2)\right)}
        }.
    \end{equation}

\end{lemma}

\begin{proof}
    We use the results of \cref{sec:scalar-tracking} from \cite{dizonvalkonen2024tracking} with the choices \eqref{eq:tracking:defs},
    and $\scalarTrackingErrorThis = 
    \norm{\nextEstE - \targetE}$.
    To verify \cref{ass:scalar-tracking:main} through \cref{ass:tracking:main}, we need to show that $\{x^j\}_{j=0}^k\subset \localset$ for each $k\in\{0,...,N+1\}.$
    We do this inductively along with the inductive proof of \cref{ineq:finite-series-inexact}.
    Clearly both hold for $k=0$, since by assumption, $x^0 \in \Omega_X$.
    This takes care of the inductive basis.

    For the inductive step, let $k\ge 0$ and $k\le N$, and suppose we have proved  $\{x^j\}_{j=0}^k\subset \localset$ and \eqref{ineq:finite-series-inexact}.
    By \cite[Lemma 3.8]{wachsmuth2025proximalgradientmethodsbanach}, we have
    \begin{equation}
        \label{eq:exact-implicit-step-inexact}
        \norm{\nextx - x^k}^2 = \tau^2 \norm{\nextEstE + q^{k+1}}^2
        = - \tau\dualprod{\nextEstE + q^{k+1}}{\nextx - x^k}
        \quad \text{ for some } q^{k+1}\in\subdiff R(\nextx).
    \end{equation}
    Young's inequality gives
    \[
    \dualprod{\nextEstE- 
        \targetE
        }{\nextx - x^k}
    \ge
    -
    \frac{1}{2}\norm{\nextEstE- 
        \targetE
        }^2
    -\frac{1}{2}\norm{\nextx - \thisx}^2
    \]
    Continuing with \cref{lemma:scalar-tracking:error-sum} gives
    \[
    \dualprod{\nextEstE- 
        \targetE
        }{\nextx - x^k}
    \ge
    -\frac{1 +\trackingressum[1]^2}{2}\norm{\nextx - \thisx}^2 
    - e_k.
    \]
    Now \cref{ineq:set-valued-descent-inexact} implies
    \begin{equation}
        \label{ineq:inexact-nonsmooth-smoothess}
        \begin{split}
            \dualprod{\nextEstE}{\nextx - x^k}
            &
            =
            \dualprod{\nextEstE- 
                \targetE
                }{\nextx - x^k}
            +
            \dualprod{
                \targetE
            }{\nextx - x^k}
            \\
            &
            \ge
            -\frac{1 +\trackingressum[1]^2}{2}\norm{\nextx - \thisx}^2 - e_k
            +
            \inf_{x^* \in \COsubdiffS E(x^k)} \dualprod{x^*}{\nextx-x^k}
            \\
            &
            \ge
            E(\nextx) - E(x^k) - \frac{L+1+\trackingressum[1]^2}{2}\norm{\nextx - x^k}^2
            - e_k.
        \end{split}
    \end{equation}
    Using the definition of the convex subdifferential as well as \cref{ineq:inexact-nonsmooth-smoothess,eq:exact-implicit-step-inexact}, we get
    \begin{align*}[E+R](\nextx) - & [E+R](x^k) + \tau\norm{\nextEstE + q^{k+1}}^2
        \\
        &
        = [E+R](\nextx) - [E+R](x^k) - \dualprod{\nextEstE + q^{k+1}}{\nextx - x^k}
        \\
        &
        \le E(\nextx) - E(x^k) -
        \dualprod{\nextEstE}{\nextx - x^k}
        \\
        &
        \le
        \frac{L+1+\trackingressum[1]^2}{2}\norm{\nextx - x^k}^2
        + e_k
        \\
        &
        =
        \frac{\tau^2 (L+1+\trackingressum[1]^2)}{2}\norm{\nextEstE + q^{k+1}}^2
        + e_k.
    \end{align*}
    Therefore,
    \begin{equation*}
        [E+R](\nextx) -  [E+R](x^k)
        +
        \tau\left(1 - \frac{\tau (L+1+\trackingressum[1]^2)}{2}\right)
        \norm{\nextEstE + q^{k+1}}^2 \le e_k.
    \end{equation*}
    Adding this 
    to \eqref{ineq:finite-series-inexact}, we get the same for $k+1.$ 
    Then \cref{eq:main-convergence:locality-lemma} imply $x^{k+1} \in \Omega_X,$ which
    concludes the induction.
    Thus \cref{ineq:finite-series-inexact} holds for each $k\in\{0,...,N+1\},$ 
    and $\{x^k\}_{k=0}^{N+1} \subset \localset$. Consequently, \cref{ass:tracking:main} implies \cref{ass:scalar-tracking:main} for each $k\in\{0,...,N+1\}.$  
    Therefore \cref{lemma:scalar-tracking:error-sum} gives
    \eqref{eq:main-convergence:e}.
    The exact expression for $\trackingressum[1]$ is in \eqref{eq:scalar-tracking:ressum}, and satisfies \eqref{eq:tracking:ressum}.

    Using \cref{eq:exact-implicit-step-inexact,ineq:finite-series-inexact,eq:main-convergence:e} we obtain
        \begin{equation}
        \label{ineq:finite-series-inexact-iterate}
        [E+R](x^{k})
        +
        \inv\tau\left(1 - \frac{\tau (L+1+\trackingressum[1]^2)}{2}\right)
        \sum_{j=0}^{k-1} \norm{x^{j+1} - x^j}^2
        \le
        [E+R](x^0)
        + \sum_{j=0}^{k-1}e_j
        <
        [E+R](x^0) + \Psi_1.
    \end{equation}
    We have $ \inv\tau\left(1 - \tau (L+1+\trackingressum[1]^2)/2\right)>0$ 
    by assumption $\tau (L+1+\trackingressum[1]^2) < 2.$ 
    Therefore \cref{ineq:finite-series-inexact-iterate} implies $\sup_{k\le N}\sum_{j=0}^{k-1}\norm{x^{j+1} - x^j}^2<\infty.$

    Within the aforementioned inductive step we have $\{x^j\}_{j=0}^k\subset \localset$ for some $k\in\{1,...,N\}.$ Therefore we can iterate \cref{ass:tracking:main}\,\cref{item:tracking:main:inner-tracking}, the inner tracking property, to obtain
    \begin{align}
        \label{ineq:inner-tracking-iterating}
        &\sup_{k\le N}\norm{u^{k+1}}_U
        \le
        \sup_{k\le N}\left\{
        \norm{S_u(x^k)}_U + \norm{u^{k+1} - S_u(x^k)}_U +
        \norm{\Delta \nexxt w_*}_*
        \right\}
        \\
        \nonumber
        &\le
        \sup_{k\le N}\norm{S_u(x^k)}_U
        +
        \sup_{k\le N}\kappa_u^{-k}\left(
        \norm{u^1- S(x^0)}_U
        +
        \norm{\Delta w_*^1}_*
        \right)
        +
        \sup_{k\le N}
        \sum_{j=1}^k \kappa_u^{-(k+1-j)}\pi_u\norm{x^{j} - x^{j-1}}
        \\
        \nonumber
        &\le
        \sup_{x\in\Omega_X}\norm{S_u(x)}_U
        +
        \norm{u^1- S(x^0)}_U
        +
        \norm{\Delta w_*^1}_*
        +
        \sup_{k\le N}
        \sum_{j=1}^k \kappa_u^{-(k+1-j)}\pi_u\norm{x^{j} - x^{j-1}}
        .
    \end{align}
    Hölder's inequality and the fact $\{\kappa_u^{-j}\}_{j=1}^\infty$ is a converging geometric sequence due to $\kappa_u>1$  imply
    \begin{align*}
         \sum_{j=1}^k \kappa_u^{-(k+1-j)}\pi_u\norm{x^{j} - x^{j-1}}
         \le
         \pi_u
         \sqrt{\sum_{j=1}^k \kappa_u^{-2j}}
         \sqrt{\sum_{j=1}^k \norm{x^{j} - x^{j-1}}^2}
         <
         \frac{\pi_u}{\sqrt{1 - \kappa_u^{-2}}}\sqrt{\sum_{j=0}^{k-1} \norm{x^{j+1} - x^{j}}^2}.
    \end{align*}
    Combining this with \cref{ineq:inner-tracking-iterating}, the assumption $\sup_{x\in\Omega_X}\norm{S_u(x)}_U < \infty$ and \cref{ineq:finite-series-inexact-iterate} imply $\sup_{k\le N}\norm{\nextu}_U < r_U$ with $r_U$ defined in \eqref{eq:rU-def}.
\end{proof}

\begin{remark}[Bounds on the inner variable]
    \label{remark:main-convergence}
    \Cref{lemma:main-convergence} requires that the differential transformation \cref{ass:tracking:main}\,\cref{item:tracking:main:differential-transformation} holds. We can obtain it for a special case using \cref{thm:tracking:diff-transformation}.
    Its assumption \eqref{eq:tracking:diff-transformation:lipschitz} may depend on $\Omega_U$ to be bounded, as we will see in the applications of \cref{sec:tv:summary}.
    By an inductive argument based on \cref{lemma:main-convergence}, we can, indeed, choose $\Omega_U=B(0, r_U)$, where $r_U$ is the bound \eqref{eq:rU-def} provided the lemma, independent of $N$, for which $\sup_{k \le N}\norm{\nextu}_U < r_U$.

    To avoid circular reasoning, we have to, however, be careful when using $r_U$ to construct $\Omega_U$ and prove \cref{ass:tracking:main}: the radius depends on factors from  \cref{ass:tracking:main}, that may depend on $\Omega_U$.
    However, studying \eqref{eq:rU-def} in more detail, independent of these factors, we see that we can make $r_U$ arbitrarily close to $\sup_{x\in\Omega_X}\norm{S_u(x)}_U$ by (a) either or both taking $\tau>0$ small enough, or initialising $x^0$ close to a minimiser, and (b) making high-quality first steps.
    In (a), the closeness and smallness still depend on these factors, or rough estimates for them.
\end{remark}

\begin{theorem}
    \label{thm:main-convergence}
    Let $X$ and $U$ be normed spaces, $E:X \to \R$, and $R:X\to\extR$ convex, proper and lower semicontinuous.
    Assume additionally that $\inf [E+R] > -\infty$ and \cref{eq:main-convergence:locality-lemma} hold.
    Given an initial iterate $x^0 \in \Omega_X$ and $\tau>0$, generate $\{x^k\}_{k\in\N}$ through \cref{inexact-step}, choosing for each iteration corresponding target sets $\COsubdiffS E(x^k)$ for which \cref{ineq:set-valued-descent-inexact} holds for an $L > 0$, as well as differential estimates $\nextEstE$ and targets $\targetE\in\COsubdiffS E(x^k)$ for which \cref{ass:tracking:main} holds.
    Suppose the step length $\tau$ satisfies $\tau (L+1+\trackingressum[1]^2) < 2$.
    Then $\norm{\nextu - S_u(\thisx)}_U \to 0$.
    Moreover:
    \begin{enumerate}[label=(\roman*)]
        \item\label{item:convergence:main:prox}
        If $X$ is a Hilbert space, then (recall \cref{rem:optimistic:prox-oc})
        \begin{equation}
            \label{eq:main-convergence:prox}
            \inf_{x^*\in\COsubdiffS E(x^k)} \norm{\prox_{\tau R}(\thisx - \tau \mathscr{R}x^*)-\thisx} \to 0.
        \end{equation}

        \item\label{item:convergence:main:subdiff}
        Suppose that either
        \begin{enumerate}[label=(\alph*)]
            \item\label{item:main-convergence:F}
            For all $\epsilon>0$ there exists $\delta>0$ such that, for some $k_N\in \N$ for any $k \ge k_N,$
            \begin{equation}
                \label{ineq:set-valued-descent-cont}
                \norm{\nextx-\thisx} \le \delta
                \implies
                \dist(\COsubdiffS E(\nextx), \targetE) \le \epsilon
                .
            \end{equation}
            \item\label{item:main-convergence:R}
            $E$ is bounded from below and $\subdiff R$ is single-valued
            Lipschitz on the restricted sublevel sets $\mathop{\mathrm{sub}}_R(c|\Omega_X) \defeq \{x \in \Omega_X \mid R(x) \le c\}$ for any $c \in \R$.
        \end{enumerate}
        Then also $\inf_{x^* \in [\COsubdiffS E+\subdiff R](\nextx)}\norm{x^* }\to 0.$
    \end{enumerate}
\end{theorem}

\begin{proof}
    We first prove that
    \begin{equation}
        \label{eq:main-convergence:initial}
        \norm{x^{k+1} - x^k} \to 0,
        \quad
        \norm{\nextEstE + q^{k+1}} \to 0,
        \quad\text{and}\quad
        \norm{\nextEstE - \targetE} \to 0.
    \end{equation}
    To do so, we start by applying \cref{lemma:main-convergence}.
    We have $\sup_{N \in \N}\sum_j^N e_j \le \Psi_1 <\infty,$ from \eqref{eq:main-convergence:e}, and \cref{ineq:finite-series-inexact}.
    Since, by assumption, $\tau (L+1+\trackingressum^2) < 2$ and $\inf [E+R] > -\infty,$ it follows from \eqref{ineq:finite-series-inexact} that $\norm{\nextEstE + q^{k+1}}\to 0.$
    We have again $\tau\norm{\nextEstE + q^{k+1}}=\norm{\nextx-\thisx}$ by \cite[Lemma 3.8]{wachsmuth2025proximalgradientmethodsbanach}, and thus also $\norm{\nextx - x^k}\to 0.$
    By \cref{lemma:scalar-tracking:inner-product-error-estimate} we have $\norm{\nextEstE - \targetE }^2 \le \breve e_{p,k}$ for some $\breve e_{p,k} \ge 0$.
    Since, by \cref{lemma:main-convergence}, $\sup_{N\in\N}\sum_{k=0}^{N-1}\norm{\nextx - \thisx}^2<\infty$, by \cref{lemma:scalar-tracking:error-one}, these quantities satisfy $\sum_{k=0}^\infty \breve e_{p,k} < \infty$. The final convergence of \eqref{eq:main-convergence:initial} follows.

    Next we prove $\norm{\nextu - S_u(\thisx)}_U \to 0$.
    Iterating the inner tracking \cref{ass:tracking:main}\,\cref{item:tracking:main:inner-tracking}, yields
    \begin{equation}
        \label{ineq:u-convergence}
        \norm{u^{k+1} - S_u(x^k)}_U + \norm{\Delta w_*^{k+1}}_*
        \le
        \kappa_u^{-k}\left(
        \norm{u^1- S(x^0)}_U
        +
        \norm{\Delta w_*^1}_*
        \right)
        +
        \sum_{j=1}^k \kappa_u^{-(k+1-j)}\pi_u\norm{x^{j} - x^{j-1}}
    \end{equation}
    where $\kappa_u^{-k}\left(
    \norm{u^1- S(x^0)}_U
    +
    \norm{\Delta w_*^1}_*
    \right)\to 0$ since $\kappa_u>1.$ Let $\epsilon>0.$ We have $\norm{x^{j} - x^{j-1}}<\epsilon$ for all $j\ge n_{\epsilon}$ for some $n_{\epsilon}\in\N$ because $\norm{x^{j} - x^{j-1}} \to 0.$ The norms $\norm{x^{j} - x^{j-1}}$ are bounded for any $j\in\N$ since $\sup_{k\le N}\sum_{j=0}^{k-1}\norm{x^{j+1} - x^j}^2<\infty$ by \cref{lemma:main-convergence}. Using additionally the fact that $\{\kappa_u^{-j}\}_{j=1}^\infty$ is a converging geometric sequence we obtain for any $k>n_{\epsilon}$ that
    \begin{align*}
        \sum_{j=1}^k \kappa_u^{-(k+1-j)}\norm{x^{j} - x^{j-1}}
        &
        =
        \sum_{j=1}^{n_{\epsilon}} \kappa_u^{-(k+1-j)}\norm{x^{j} - x^{j-1}}
        +
        \sum_{j=n_{\epsilon}+1}^k \kappa_u^{-(k+1-j)}\norm{x^{j} - x^{j-1}}
        \\
        &
        \le
        n_{\epsilon}\kappa_u^{-(k+1-n_{\epsilon})}\sup_{1\le j\le n_\epsilon}\norm{x^{j} - x^{j-1}}
        +
        \frac{\epsilon}{1 - \inv\kappa_u}.
    \end{align*}
    Therefore, $\sum_{j=1}^k \kappa_u^{-(k+1-j)}\norm{x^{j} - x^{j-1}}\to 0$ and $\norm{\nextu - S_u(\thisx)}_U \to 0$ follow from \cref{ineq:u-convergence}.

    We now divide the argument into cases for the alternative assumptions.

    \textbf{Case \cref{item:convergence:main:prox}:}
    Since proximal maps are $1$-Lipschitz continuous \cite[Corollary 6.18]{clason2020introduction} with factor 1, we have
    \[
        \begin{split}
        \norm{\prox_{\tau R}(\thisx-\mathscr{R}\targetE) - \thisx}
        &
        \le
        \norm{\prox_{\tau R}(\thisx-\mathscr{R}\nextEstE) - \thisx}
        + \norm{\mathscr{R}\nextEstE-\mathscr{R}\targetE}
        \\
        &
        =
        \norm{\nextx - \thisx}
        + \norm{\nextEstE-\targetE}.
        \end{split}
    \]
    Using \eqref{eq:main-convergence:initial}, the claim follows.

    \textbf{Case \cref{item:convergence:main:subdiff}\ref{item:main-convergence:F}:}
    We have
    \begin{align}
        \label{ineq:option-a-main-convergence-1}
        \begin{split}
        \inf_{x^* \in [\COsubdiffS E+\subdiff R](\nextx)}\norm{x^* }
        &
        \leq
        \norm{\nextEstE + q^{k+1}}
        + \inf_{x^* \in [\COsubdiffS E + \subdiff R](\nextx)}\norm{x^*  - (\nextEstE + q^{k+1})}
        \\
        &
        \le
        \norm{\nextEstE + q^{k+1}}
        + \inf_{x^*\in \COsubdiffS E(\nextx)}\norm{x^* - \nextEstE}
        \\
        &
        \le
        \norm{\nextEstE + q^{k+1}}
         +
        \dist(\COsubdiffS E(\nextx), \targetE)
        +
        \norm{\targetE - \nextEstE}.
        \end{split}
    \end{align}
    We have $\dist(\COsubdiffS E(\nextx), \targetE) \to 0$ from \eqref{ineq:set-valued-descent-cont} and $\norm{\nextx-\thisx} \to 0$.
    The latter holds by  \eqref{eq:main-convergence:initial}, which also establishes that the remaining terms converge to zero.

    \textbf{Case \cref{item:convergence:main:subdiff}\ref{item:main-convergence:R}:}
    We have
    \[
        \begin{split}
        \inf_{x^* \in [\COsubdiffS E+\subdiff R](\thisx)}\norm{x^*}
        &
        \le
        \inf_{q \in \subdiff R(\thisx)}
        \norm{\targetE + q}
        \\
        &
        \le
        \norm{\nextEstE - \targetE} + \norm{\nextEstE + q^{k+1}} +  \inf_{q \in \subdiff R(\thisx)} \norm{q^{k+1} - q}.
        \end{split}
    \]
    Since $E$ is bounded from below, \cref{ineq:finite-series-inexact,eq:main-convergence:e} establish a bound $c \ge R(\thisx)$ for all $k \in \N$.
    By \cref{lemma:main-convergence}, we have $\thisx \in \Omega_X$ for all $k \in \N$.
    Since $\subdiff R$ is single-valued Lipschitz\footnote{We could virtually relax this assumption to a “uniform semi-Lipschitz” property, but such a property does not appear to hold for any nonsmooth functions of interest.}
    on $\mathop{\mathrm{sub}}_R(c|\Omega_X)$, $\norm{q^{k+1} - q^k} \to 0$ follows from $\norm{x^{k+1} - x^k} \to 0$.
    The latter holds by \eqref{eq:main-convergence:initial}, which also establishes that the remaining terms converge to zero.
\end{proof}

\begin{remark}
    \label{rem:convergence:barrier}
    The Lipschitz property on sublevel sets in \cref{item:convergence:main:subdiff}\cref{item:main-convergence:R} holds for, e.g., barrier functions, which are Lipschitz on their sublevel sets, but whose differentials' Lipschitz factors blow up towards the boundary.
\end{remark}

\begin{remark}
    If we are content with subsequential convergence, $\liminf_{k \to \infty} \inf_{x^* \in [\COsubdiffS E+\subdiff R](\nextx)}\norm{x^* } = 0$, then the argument of \cref{item:convergence:main:subdiff}\ref{item:main-convergence:R} can be generalised to functions $R$ that have a Lipschitz differential on their domain, provided the domain has a Lipschitz boundary, and we ensure that $\{\nextEstE\}$ is bounded. This includes, in particular, indicator functions of Lipschitz domains.

    Indeed, we only need to show that $\liminf_{k \to \infty} \inf_{q \in \subdiff R(\thisx)} \norm{q^{k+1} - q} \to 0$.
    Now, the boundedness of $\{\nextEstE\}$ and \eqref{eq:main-convergence:initial} ensure that also $\{\this q\}$ is bounded.
    Using this boundedness when on the boundary of $\Dom R$ with the assumed Lipschitz properties in a case by case analysis of transitions between $\bd\Dom R$ or $\interior\Dom R$, we see that only when $\nextx \in \bd\Dom R$ and $\thisx \in \interior \Dom R$, we cannot bound $ \inf_{q \in \subdiff R(\thisx)} \norm{q^{k+1} - q} \le L \norm{x^{k+1} - \thisx}$ for some $L>0$. This can happen at most on every second step.
\end{remark}

\subsection{Example}
\label{sec:simple-example}

We now study the assumptions of \cref{thm:main-convergence} in a simple one-dimensional example, reminiscent of optimally choosing the regularisation parameter for the Lasso problem.
That is, we consider the problem
\begin{equation}
    \label{eq:simple-example}
    \min_{x,u}  \underbrace{\frac{1}{2}(u-2)^2}_{J(u)} + \underbrace{\delta_{\ge \epsilon}(x)}_{R(x)}
    \quad\text{subject to}\quad
    u = \argmin_v \underbrace{\frac{1}{2}(v-5)^2}_{f(v; x)} + \underbrace{x|v|}_{g(v; x)}
\end{equation}
with a small $\epsilon>0.$
The inner objective is (strongly) convex for any $x \ge 0$.
Therefore, the inner problem can equivalently be written as $0=G(u, x)$ for
\begin{equation}
    \label{eq:G-example}
    G(u,x) =
    (u-5) +
    x
    \begin{cases}
        \{-1\}  & \text{if } u<0,
        \\
        [-1, 1] & \text{if } u=0,
        \\
        \{1\}   & \text{if } u>0.
    \end{cases}
\end{equation}

The next lemma provides the inner solution mapping $S_u.$

\begin{lemma}
    \label{ex:S-lemma}
    Let $G:\R\times \R_+ \setto \R$ given by \cref{eq:G-example} and define the solution mapping $S_u:\R_+ \setto \R$ through the satisfaction of $0\in G(S_u(x), x).$
    Then $S_u(x) = \max(0, 5-x).$
\end{lemma}

\begin{proof}
    Studying the optimality condition $0\in G(u,x)$ yields the following conclusions:
    \begin{itemize}[nosep]
        \item If $u<0$, we must have $u=5+x,$. Given that we restrict $x\ge0$, it follows that $u<0$ cannot be a solution for any admissible $x$.
        \item If $u>0$, we must have $u=5-x$. Such an $u$ is indeed a solution when $0 \le x<5$.
        \item If $u=0$, we must have $5\in x[-1,1]$. Thus $u=0$ is a solution for any $x\ge 5$.
    \end{itemize}
    Therefore we obtain the (single-valued inner) solution mapping $S_u(x) = \max(0, 5-x).$  
\end{proof}
\begin{remark}
    Similarly to \cref{ex:S-lemma}, $\tilde S_u:\R_+ \setto \R$ solving $\Delta w_*^{k+1}\in G(\tilde S_u(\thisx), \thisx)$ can be shown to be $\tilde S_u(x) = \max(0, 5-x+\Delta w_*^{k+1}).$ When $x-\Delta w_*^{k+1}>0$ it holds  $\max(0, 5-x+\Delta w_*^{k+1}) = S_u(x-\Delta w_*^{k+1}),$ which demonstrates that $\tilde S_u(x)$ and $S_u$ as translation relationship, and thus we can deduce properties related to $\Delta w_*^{k+1}\in G(\nextu, \thisx)$ from the ones related to $0\in G(u,x).$
\end{remark}

\begin{figure}[t!]
    \input{Tikz_image1.tex}
    \caption{%
        Illustrations for \cref{sec:simple-example}.
        The left subfigures visualise single-valued Lipschitz continuous solution mapping $S_u$ and corresponding outer objective $E=J\circ S_u$ with important normal vectors for $x_1=1$ and $x_2=6.$ The $y$-directions of the normal vectors are such that for $i=1,2$ it holds $x^*_i\in D^*E(x_i|E(x_i))(1)$ and $x^*_i\in D^* S_u(x_i|S_u(x_i))(J'(S_u(x_i))).$ In this case different versions of $\COsubdiffS E(x)$ from \cref{rem:composition-subdiff} are equal. It is also visually clear that at $x=5$ there does not exist vector orthogonal to $ \graph S_u$ with positive $y$-coordinate or orthogonal to $\graph E$ with negative $y$-coordinate. Therefore $\frechetCod E(5|E(5))(1)=\emptyset$ and $\frechetCod S_u(5|S(5))(J'(S_u(5)))=\emptyset.$
    }
    \label{fig:simple-example}
\end{figure}
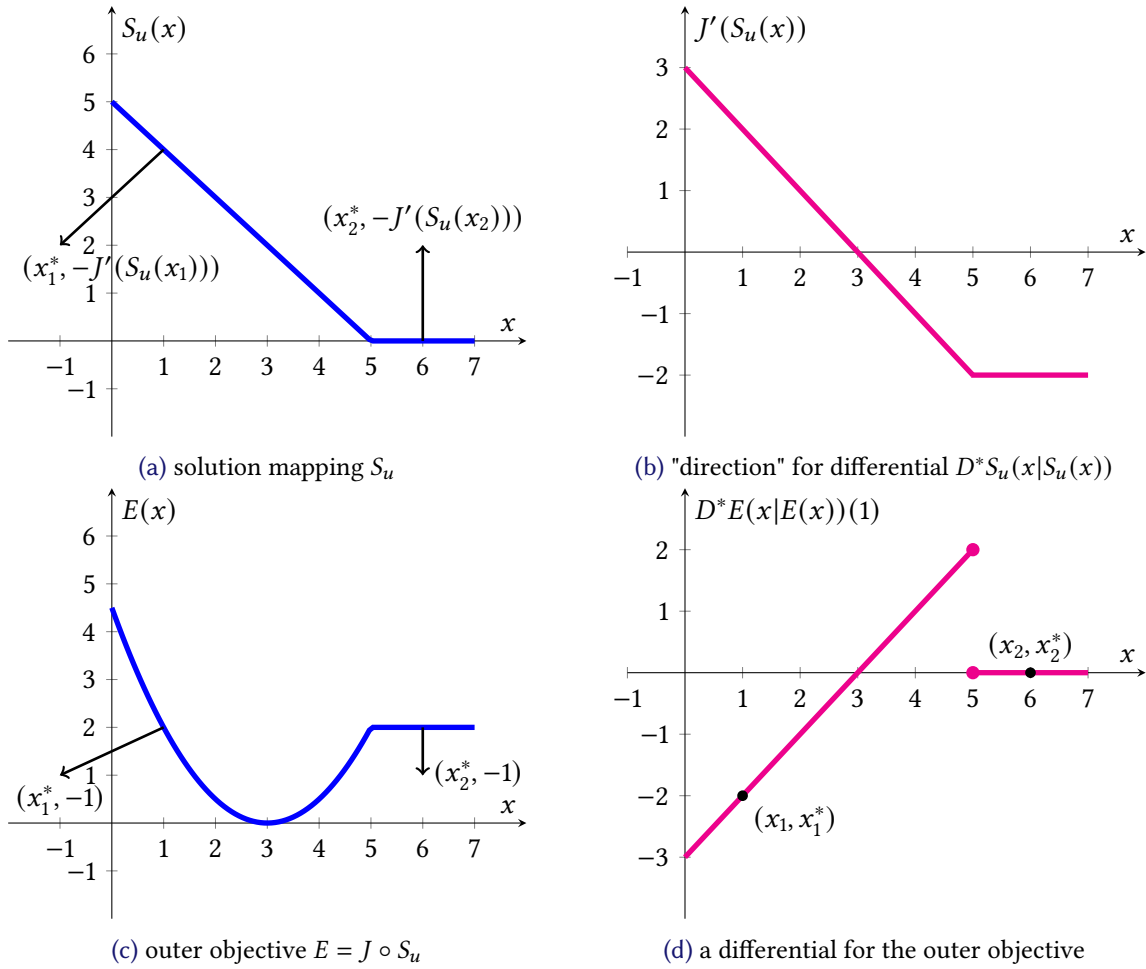

Thus taking $J(u) = \frac{1}{2}(u-2)^2$ simplifies the first term in the outer objective into the form
\begin{equation}
    \label{eq:example:F}
    E(x) = [J\circ S_u](x) = \frac{1}{2}(\max(0, 5-x)-2)^2 =
    \begin{cases}
        \frac{1}{2}(x-3)^2 & \text{if } 0 \le x \le 5
        \\
        2                  & \text{if } x>5.
    \end{cases}
\end{equation}
The functions $S_u$ and $E$ are visualised in \cref{fig:simple-example}. Since $S_u$ is Lipschitz continuous (and single-valued), \cref{thm:inner-tracking} implies the inner tracking property provided the inner algorithm satisfies the basic contractivity property \cref{ineq:linear-inner-convergence}. This is the case for, e.g., forward-backward splitting by \cref{thm:subreg:convergence-result-sub-peb:general}.
Moreover, because $S_u$ is single valued and Lipschitz,  \cref{thm:calculus:single-valued-chain} gives
\begin{equation*}
    \frechetCod[J\circ S_u](x|J(u))(1) = \frechetCod S_u(x|u)(J'(u)),
    \quad
    \text{for}
    \quad
    x>0,\, u=S_u(x).
\end{equation*}
This is also visualised in \cref{fig:simple-example}.
We take next a closer look at the adjoint based on $G$, as formulated in \cref{eq:algorithm:full-oc}, and used as basis of our algorithms.

\begin{lemma}
    \label{theorem:example-adjoint}
    Let $G$ be given by \cref{eq:G-example}.
    Suppose $u, x \in \R$ satisfy $0\in G(u,x)$ and $x>0.$
    Let  $u^*=J'(u) = u-2$ and $x^*, w \in \R$.
    Then
    \begin{equation}
        \label{eq:ex-adjoint-frechet}
        (-u^*, x^*)\in\frechetCod G(u,x|0)(w)
    \end{equation}
    is equivalent to
    \begin{equation}
        \label{ex:adjoint-result}
        x^* = w \in
        \begin{cases*}
            \{-u^*\} =\{ 2-u \} & \text{if $x\in(0,5)$ and $u=5-x$}
            \\
            \{0 \}              & \text{if $x>5$ and $u=0$},
            \\
            \emptyset           & \text{otherwise}.
        \end{cases*}
    \end{equation}
    Likewise, for the limiting coderivative, $(-u^*, x^*)\in D^* G(u,x|0)(w)$ is equivalent to
    \begin{equation}
        \label{ex:adjoint-result-limiting}
        x^* = w \in
        \begin{cases*}
            \{-u^*\} = \{ 2-u\} & \text{if $x\in(0,5)$ and $u=5-x$}
            \\
            \{0\}               & \text{if $x>5$ and $u=0$},
            \\
            \{0,2\}             & \text{if $x=5$ and $u=0$}
            \\
            \emptyset           & \text{otherwise}.
        \end{cases*}
    \end{equation}
\end{lemma}

\begin{proof}
    The condition $0\in G(u,x)$ implies $u=\max(0, 5-x)$ by \cref{ex:S-lemma}. Therefore $u>0$ when $x\in(0,5)$ and $u=0$ when $x\ge5.$
    We consider the case $u>0$ first. Then $G$ is differentiable at $(u, x)$, so by \cite[Theorem 20.14]{clason2020introduction} we have $\frechetCod G(u,x|0)(w) = D^* G(u,x|0)(w) = \{G'(u,x)^*w\}=\{(w, w)\}$ . Therefore, as claimed, $x^*=w=-u^*$.


    We next find $(-u^*, x^*)\in\frechetCod G(u,x|w_*)(w)$ for $u=0$ and any $w_* \in G(u, x)$ close to zero.
    By $w_* \in G(u, x)$, this requires $-x \le w_* + 5 \le x$.
    We separate the study into two cases: $-x < w_* + 5 < x$ and $w_* + 5 = x$, the lower bound not being possible for $w_* \approx 0$.
    In either case, by the definition of the Fréchet coderivative
    \begin{equation}
        \label{eq:G-coderivative-def-2}
        \limsup_{\this x \to x,\, \this u \to u,\, G(\this u, \thisx) \ni \this w_* \to 0}
        \frac{-\dualprod{w}{\this w_* - w_*} + \dualprod{x^*}{\thisx-x}-\dualprod{u^*}{\this u - u}}
        {\norm{(\thisx-x,\this u - u,\this w_* - 0)}}
        \le 0.
    \end{equation}
    We consider the case $-x < w_* + 5 < x$ first.
    First of all, taking $u^k\to u=0$ and $x^k\to x$, we can only have $G(u^k,x^k) \ni w_*^k \to w_*$ in \cref{eq:G-example} if $u^k=0$ and $x^k>0$ for all $k$ large enough.
    Taking also $w^k_*=w_*$ and passing to the limit, it follows that, necessarily, $x^*=0$, as claimed.
    On the other hand,
    we can take $\this w_*\in G(\thisu, \thisx) = \thisx[-1,1] - 5, \this w_*\to w_*$  such that either $\this w_*>w_*$ or $\this w_*<w_*$ for all $k\in\N.$ Therefore for any $w\in\R$ we have $\sup -\dualprod{w}{\this w_*-w_*} = |w(\this w_*-w_*)|$ and thus \cref{eq:G-coderivative-def-2} can hold only if also $w=0$.
    Clearly, on the other hand, the choice $x^*=w=0$ satisfies \cref{eq:G-coderivative-def-2}.
    Thus, taking $w_*=0$, we prove the Fréchet coderivative expression for $u=0$ and $x \in (0, 5)$.

    Consider the case $u=0$ and $w_* + 5 = x$.
    Take $x^k>x$ and $u^k=0$.
    Then $0\in G(u^k, x^k)$, so we can take $w_k^*=0$ in \cref{eq:G-coderivative-def-2}.
    It follows that \cref{eq:G-coderivative-def-2} can hold only if $\limsup_{k \to \infty} \dualprod{x^*}{x^k-x}/\norm{x^k-x}\le 0.$
    Since, we have taken any $x<x^k\to x$, we must have $x^*\le 0$.
    On the other hand, taking $x^k<x$ and $\thisu = w_* + 5 -\thisx$ we have $\this w_* \in G(u^k, x^k)=\{0\}$.
    Therefore we get that the limsup in \cref{eq:G-coderivative-def-2} is at least
    \begin{align*}
        \lim_{x>\thisx\to x}
        \frac{\dualprod{x^*}{\thisx-x}-\dualprod{u^*}{\this u}}
        {\norm{(\thisx-x,\this u,0)}}
        =
        \lim_{5>\thisx\to x}
        \frac{(x^*+u^*)(\thisx - x)}
        {\norm{(\thisx-x,x-\thisx,0)}}
        = -\frac{1}{\sqrt{2}}(x^*+u^*),
    \end{align*}
    which is non-positive only if $x^*\ge -u^* = 2 - 0 = 2.$
    However, this can not hold simultaneously with $x^*\le 0.$
    Therefore, no solution exists for $u=0$ and $x=w^*+5$.

    Finally, we consider the case $x=5$ for the limiting coderivative, i.e., we find which $x^*$ and $w$ satisfy $(-u^*, x^*)\in D^* G(u,x|0)(w).$
    By the definition of limiting coderivative this is equivalent to the existence of $(u^*_k, x^*_k)\in\frechetCod G(\thisu, \thisx|\this w_*)(\this w)$ such that $(-u^*_k, x^*_k, \this w_*)\to (-u^*, x^*, 0)$ with $\thisu\to u, \thisx \to x$ and $\this w \to w.$
    By the proof above, the only possibilities are
    \begin{itemize}
        \item $(-u^*_k, x^*_k)=(2, 0)$ with $\thisu=0,\thisx > 5 + \this w_*$, $\this w_* \to 0$,  and $\this w=0.$
        \item $(-u^*_k, x^*_k)=(2-\this u, 2-\this u)\to (2,2)$ with $\thisu=5-\thisx\to 0,\thisx < 5$ and $\this w=2-\this u\to 2.$
    \end{itemize}
    Therefore, $x^*=w=0$ and $x^*=w=2$ are the solutions that satisfy  $(-u^*, x^*)\in D^* G(u,x|0)(w).$
\end{proof}

Recalling definitions of $Z$ and $\COsubdiffS_2 F$ for the limiting coderivative from \cref{eq:tracking:adjoint-solution-map,def:set-valued-target} 
we obtain the following corollary.

\begin{corollary}
    \label{ex:Z-corollary}
    Let $G$ be given by \cref{eq:G-example}. Then the  $\COsubdiffS_2 E$ and $Z$ corresponding to the limiting coderivative satisfies
    \[
        \COsubdiffS_2 E(x) = Z(x) =
        \begin{cases}
            \{x-3\}  & \text{if } 0 < x < 5,
            \\
            \{0, 2\} & \text{if } x = 5,
            \\
            \{0\}    & \text{if } x>5,
        \end{cases}
    \]
    as well as
    \[
        \Zunext =
        \begin{cases}
            \{2 - \nextu\} & \text{if } 0 < x < 5 + \Delta w_*^{k+1} \text{ and } \nextu > 0,
            \\
            \{0, 2\}       & \text{if } x = 5 + \Delta w_*^{k+1} \text{ and } \nextu = 0,
            \\
            \{0\}          & \text{if } x>5 + \Delta w_*^{k+1} \text{ and } \nextu = 0.
        \end{cases}
    \]
\end{corollary}


We next show that $E$ satisfies the set-valued descent condition \cref{ineq:set-valued-descent-inexact}.

\begin{lemma}
    \label{lemma:ex-set-valued-descent}
    Let function $R=\delta_{[\epsilon,\infty)}.$ The function $E$ given in \eqref{eq:example:F}, satisfies the inequality \cref{ineq:set-valued-descent-inexact} for $\COsubdiffS E= \COsubdiffS_2 E$ and $L=1.$ 
\end{lemma}

\begin{proof}
    We divide the verification into several cases, depending on the values of the base point $x^k$, and of $x$:
    \begin{itemize}
        \item If $\thisx\neq 5$ and $5$ is not in the interval between $x$ and $\thisx,$
              then $E$ is Lipschitz-differentiable, and we obtain \cref{ineq:set-valued-descent-inexact} for given $x^k$and $x$ by \cite[Lemma 7.1]{clason2020introduction}.
        \item If $\thisx>5$ and $x<5,$ then $x^*\in \COsubdiffS E(x^k) =\{0\}$ by \cref{ex:Z-corollary} and \cref{ineq:set-valued-descent-inexact} reduces to
              \begin{equation}
                  \label{ex:set-valued-descent-1}
                  0 \ge E(x) - E(\thisx) - \frac{1}{2}(x-\thisx)^2.
              \end{equation}
              When $x<1$ we have, as required,
              \[
                  E(x) - E(\thisx) - \frac{1}{2}(x - \thisx)^2 < \frac{9}{2} - 2  - \frac{1}{2}(4)^2 < 3-8 < 0.
              \]
              When $x\ge 1$, \cref{ex:set-valued-descent-1} also holds since
              $E(x) = \frac{1}{2}(x-3)^2 \le 2 = E(\thisx).$
        \item If $\thisx < 5$ and $x>5$, we verify the claim from
              \[
                  \begin{split}
                      0
                       &
                      \ge 2 - \frac{1}{2}(x-3)^2
                      = 2 - \frac{1}{2}(\thisx - 3 - (\thisx - x))^2
                      \\
                       &
                      = 2 - \frac{1}{2}(\thisx-3)^2 -\frac{1}{2}(x-\thisx)^2 + (x^k - 3)(\thisx - x).
                      \\
                       &
                      = E(x) - E(x^k) - \frac{1}{2}(x-x^k)^2 - \inf_{x^* \in \COsubdiffS E(x^k)} \dualprod{x^*}{x-x^k}
                  \end{split}
              \]
              where the last equality follows from the definition of $E$ and \cref{ex:Z-corollary}.
        \item If $\thisx=5,$ then it holds
              \[
                  \inf_{x^* \in \COsubdiffS E(x^k)}\dualprod{x^*}{x-x^k} =
                  \begin{cases}
                      0      & \text{if } x\ge 5,
                      \\
                      2(x-5) & \text{if } x < 5
                  \end{cases}
              \]
              by \cref{ex:Z-corollary}.
              When $x\ge 5$, \cref{ineq:set-valued-descent-inexact} again reads as \cref{ex:set-valued-descent-1} and holds since $E(x) = 2 = E(x^k).$
              When $x<5$, \cref{ineq:set-valued-descent-inexact} reads as
              \[
                  2(x-5) \ge \frac{1}{2}(x-3)^2 - 2 - \frac{1}{2}(x-5)^2.
              \]
              We verify this condition and finish the proof with the help of
              \[
                  \frac{1}{2}(x-3)^2 - \frac{1}{2}(x-5)^2
                  = \frac{1}{2}([x-3]+[x-5])([x-3]-[x-5])
                  =  \frac{1}{2}(2x-8)2
                  = 2(x-4).
                  \qedhere
              \]
    \end{itemize}
\end{proof}

The next lemma proves that the adjoint tracking \cref{ass:tracking:main}\,\cref{item:tracking:main:adjoint-tracking} holds with the forward-backward inner steps and exact adjoint solver.

\begin{lemma}
    \label{lemma:example-adjoint-tracking}
    Let $f$, $g$, and $J$ be given by \eqref{eq:simple-example} as well as $G$ be given by \cref{eq:G-example}.
    Moreover, let $u^0\ge 0, x^0\ge\epsilon>0$ and
    \begin{itemize}
        \item the inner iterate $\nextu$ be produced by the forward-backward method with step length $\tau \in (0,1),$ i.e. $\nextu = \prox_{\tau g(\freevar,\thisx)}(\thisu - \tau\grad_u f(\thisu,\thisx)),$
        \item the adjoint iterate $\nexxt w\in \Zunext$ for  $Z$ given by \cref{ex:Z-corollary} and $\Delta \nexxt w_*\in G(\nextu,\thisx),$ 
        \item the outer iterate $\thisx=\proj_{[\epsilon,\infty)}(\prev x - \sigma \widetilde x^*_k)$ be produced with the forward-backward method for $\sigma > 0$ and the differential estimate $\widetilde x^*_k$ satisfying $(-J'(\thisu), \widetilde x^*_k)\in D^*G(\thisu, \prev x| \Delta \this w_*)(\this w).$
    \end{itemize}
    Then the adjoint tracking \cref{ass:tracking:main}\,\cref{item:tracking:main:adjoint-tracking} holds for $\mu_u, \pi_w>0$ and $\kappa_w>1$ satisfying $ \mu_u \ge 1 + \kappa_w,$ \\ $(1+\kappa_w)\norm{\freevar}_U \le \mu_u \inv\tau\norm{\freevar}_*$ and $\pi_w \ge \inv\sigma\kappa_w.$
\end{lemma}

\begin{proof}
    To prove that \cref{ass:tracking:main}\,\cref{item:tracking:main:adjoint-tracking}, it suffices to use \cref{thm:adjoint-ZZ}, for which we need to verify \cref{ineq:adjoint-ZZ}, which we can equivalently write as
    \begin{multline}
        \label{ineq:ex-adjoint-proof-1}
        \kappa_w \dist( Z(x^k), \nexxt w)
        + \dist(\Zunext, \this w)
        \\
        \le
        \dist(Z(x^{k-1}), \this w)
        +
        \pi_w\norm{x^k - x^{k-1}}
        + \mu_u\left(
        \norm{\nextu - S_u(\thisx)}_U +
        \norm{\Delta\nexxt w_*}_*
        \right).
    \end{multline}
    The implicit form of the forward-backward method
    \[
        \Delta \nexxt w = \inv\tau (\nextu - \thisu) \in \subdiff_u g(\nextu,\thisx) + \grad_{u}f(\thisu,\thisx).
    \]
    and the assumption $(1+\kappa_w)\norm{\freevar}_U \le \mu_u \inv\tau\norm{\freevar}_*$
    let us show \cref{ineq:ex-adjoint-proof-1} holds by proving
    \begin{multline}
        \label{ineq:ex-adjoint-proof}
        \kappa_w \dist( Z(x^k), \nexxt w)
        + \dist(\Zunext, \this w)
        \\
        \le
        \dist(Z(x^{k-1}), \this w)
        +
        \pi_w\norm{x^k - x^{k-1}}
        + \mu_u
        \norm{\nextu - S_u(\thisx)}_U +
        (1+\kappa_w)\norm{\nextu - \thisu}_U.
    \end{multline}

    We use the following facts throughout the proof:
    \begin{enumerate}[label=(\roman*)]
        \item
              \label{ex:item-i}
              $S_u(x) = \max(0, 5-x)$ from \cref{ex:S-lemma}, which combined with \cref{ex:Z-corollary} gives the form
              \[
                  Z(x) =
                  \begin{cases}
                      \{2 - S_u(x)\}    & \text{if } 0 < x < 5,
                      \\
                      \{0, 2 - S_u(x)\} & \text{if } x = 5,
                      \\
                      \{0\}             & \text{if } x>5.
                  \end{cases}
              \]
        \item
              \label{ex:item-ii}
              \cref{theorem:example-adjoint} proves that $x^*=w$ for $x^*$ and $w$ satisfying $(-u^*, x^*)\in D^* G(u,x|0)(w).$ Since $w^k\in Z(\thisu, \prev x, \Delta \this w_*),$ we obtain similarly that
              the differential estimate $\widetilde x^*_k = w^k.$ Therefore we get for $\thisx > 5$ that the outer step
              $\thisx = \proj_{[\epsilon,\infty)}(\prev x - \sigma \widetilde x^*_k) = \prev x - \sigma w^k,$
        \item
              \label{ex:item-iv}
              \cref{thm:subreg:convergence-result-sub-peb:general} implies $\kappa_u\norm{\nextu - S_u(\thisx)}_U \le \norm{\thisu - S_u(\thisx)}_U$ for some $\kappa_u>1,$ since $f$ is Lipschitz-differentiable with $L=1,$
        \item Since the adjoint is solved exactly,  $\dist(\Zunext, \nexxt w) = 0.$
        \item If $\thisu\ge0$ and $\tau\in(0,1),$ then $\thisu - \tau(\thisu - 5) \ge 0.$ Thus
        the forward-backward step reads as
              \[
                  \begin{split}
                      \nextu
                       &
                      =
                      \prox_{\tau\thisx|\freevar|}(\thisu - \tau(\thisu - 5))
                      \\
                       &
                      =
                      \max(0, \thisu - \tau(\thisu - 5) - \tau\thisx)
                      \\
                       &
                      = \max(0, (1-\tau)\thisu + \tau(5 - \thisx)).
                  \end{split}
              \]
              Therefore $u^{(0)}>0$ implies that $\thisu\ge0$ for all $k\in\N.$
    \end{enumerate}

    We now prove \eqref{ineq:ex-adjoint-proof} through various cases:
    \begin{enumerate}[label=(\alph*)]
        \item $\thisu=0,$ in which case we consider the sub-cases:
              \begin{enumerate}[label=(\alph{enumi}.\arabic*)]
                  \item $w^k=0,$ then we have $\thisx = \prev x$ by the structure of the forward-backward updates, as observed above. We consider the following two sub-cases:
                        \begin{enumerate}[label=(\alph{enumi}.\arabic{enumii}.\arabic*)]
                            \item
                                  $\thisx \ge 5,$ then $S_u(\prev x) = S_u(\thisx) = 0$ by \cref{ex:item-i}, and also $\nextu=0$ because \cref{ex:item-iv} implies
                                  \[
                                      \norm{\nextu - 0}_U = \norm{\nextu - S_u(\thisx)}_U \le  \norm{\thisu - S_u(\prev x)}_U = \norm{0 - 0}_U = 0.
                                  \]
                                  Since $\nextu = S_u(\thisx)$ and $\Delta w_*^{k+1} = \inv\tau(\nextu - \thisu)=0,$ we have $\Zunext = Z(\thisx).$ Therefore, we have $\nexxt w = 0\in Z(\thisx)$ and $w^k\in\Zunext$ implying that the left side of \cref{ineq:ex-adjoint-proof} is zero, verifying the condition.

                            \item $\thisx < 5,$ then $S_u(\thisx)>0$  by \cref{ex:item-i}, and thus also $\nextu>0$ since it is closer to $S_u(\thisx)>0$ than $\thisu=0.$ Therefore, $\nexxt w = 2 - \nextu = \Zunext$ and $Z(\thisx) = 2 - S_u(\thisx).$ The left side of the inequality \cref{ineq:ex-adjoint-proof} is now
                                  \begin{align*}
                                      \kappa_w\norm{2 - S_u(\thisx) - (2 - & \nextu)}_U + \norm{2- \nextu - 0}_U
                                      \\
                                                                           &
                                      \le
                                      (1+\kappa_w)\norm{\nextu - S_u(\thisx)}_U
                                      +
                                      \norm{2-S_u(\thisx)}_U
                                      \\
                                                                           &
                                      =
                                      (1+\kappa_w)\norm{\nextu - S_u(\thisx)}_U
                                      +
                                      \norm{2-S_u(\prev x)}_U
                                      ,
                                  \end{align*}
                                  and the right side of \cref{ineq:ex-adjoint-proof} is
                                  \begin{align*}
                                      \norm{2-S_u(\prev x)}_U +
                                      \mu_u\left(
                                      \norm{\nextu - S_u(\thisx)}_U +  \tau\norm{\nextu - \thisu}_U
                                      \right).
                                  \end{align*}
                                  Therefore, \cref{ineq:ex-adjoint-proof} holds with $\mu_u \ge (1+\kappa_w).$
                        \end{enumerate}
                  \item $w^k= 2 - \thisu = 2,$ then $\thisx = \prev x - \sigma 2 < \prev x.$
                        We again consider two cases:
                        \begin{enumerate}[label=(\alph{enumi}.\arabic{enumii}.\arabic*)]
                            \item
                                  \label{item:a21}
                                  $\thisx <5,$ then  \cref{ex:item-i} again gives $S_u(\thisx)>0$ implying $\nextu > 0.$
                                  The left side of \cref{ineq:ex-adjoint-proof} is
                                  \[
                                      \kappa_w\norm{2 - S_u(\thisx) - (2 - \nextu)}_U + \norm{2- \nextu - (2 - \thisu)}_U
                                      =
                                      \kappa_w\norm{\nextu - S_u(\thisx)}_U + \norm{\nextu -\thisu}_U,
                                  \]
                                  implying the inequality holds for $\mu_u \ge \kappa_w.$

                            \item $\thisx \ge 5,$ then $\prev x >5$ implying $S_u(\prev x) = S_u(\thisx) = 0.$ Thus also $\nextu=0.$ Since $\nextu = S_u(\thisx)$ and $\Delta w_*^{k+1} = 0,$ we have $\Zunext = Z(\thisx)$ implying  $\nexxt w\in Z(\thisx).$ We can have $\Zunext=0$ and thus the left side of \cref{ineq:ex-adjoint-proof} equal $2$. However, also $Z(\thisx) = \{0\}$ implying $\dist(Z(\prev x).w^k) = 2$ and that the right side of \cref{ineq:ex-adjoint-proof} is larger than $2.$
                        \end{enumerate}
              \end{enumerate}
        \item $\thisu \in (0,2],$ implies $\this w=2-\thisu \le 0,$ and thus $\thisx \le \prev x.$ We check again sub-cases separately:
              \begin{enumerate}[label=(\alph{enumi}.\arabic*)]
                  \item \label{item:b1}
                        $\thisx \le 5,$ then $2-S(\thisx)\in Z(\thisx).$ We obtain
                        \[
                            \nextu \ge (1-\tau)\thisu + \tau(5 - \thisx) \ge (1-\tau)\thisu > 0
                        \]
                        from the inner step. Thus $\nexxt w = 2 - \nextu$ and this case is analogous to \cref{item:a21}.
                  \item $\thisx > 5,$ then $Z(\thisx) = 0$ and $Z(\prev x) = 0.$ Since $\nexxt w\in \Zunext \subset \{0, 2-\nextu\}$ by \cref{ex:Z-corollary}, we can have only either:
                        \begin{enumerate}[label=(\alph{enumi}.\arabic{enumii}.\arabic*)]
                            \item $\nexxt w = 0,$ then the left side of \cref{ineq:ex-adjoint-proof} is at most $\norm{2-\thisu}_U$, which is less than the right side that involves $\dist(Z(\prev x), \this w) = \norm{2-\thisu}_U.$
                            \item $\nexxt w = 2 - \nextu,$ then the left side of \cref{ineq:ex-adjoint-proof} is
                                  \[
                                      \kappa_w\norm{2-\nextu}_U + \norm{\nextu - \thisu}_U
                                      \le
                                      \kappa_w\norm{2-\thisu}_U + (1+\kappa_w)\norm{\nextu - \thisu}_U.
                                  \]
                                  The right side is
                                  \[
                                      \norm{2-\thisu}_U + \pi_w\norm{x^k - x^{k-1}}
                                      + \mu_u
                                      \norm{\nextu - S_u(\thisx)}_U +  (1 + \kappa_w)\norm{\nextu - \thisu}_U
                                      .
                                  \]
                                  Using the fact \cref{ex:item-ii} we have $x^k - x^{k-1} = \sigma w^k = \sigma(2-\thisu)$, which verifies \cref{ineq:ex-adjoint-proof} for $\pi_w \ge \inv\sigma(\kappa_w - 1).$
                        \end{enumerate}
              \end{enumerate}
        \item $\thisu > 2,$ then $\this w = 2- \thisu < 0.$ We consider the sub-cases:
              \begin{enumerate}[label=(\alph{enumi}.\arabic*)]
                  \item $\thisx \le 5,$ which is analogous to \cref{item:b1}.
                  \item $\thisx > 5,$ which we divide into two further sub-cases based on the implication of \cref{ex:Z-corollary},  $\Zunext \subset \{0, 2-\nextu\}$:
                        \begin{enumerate}[label=(\alph{enumi}.\arabic{enumii}.\arabic*)]
                            \item $\nexxt w \in\Zunext = \{0\},$ then either $\nextu = 0$ or $\nextu = 2$.
                                  In both cases
                                  \[
                                      \kappa_w \dist( Z(x^k), \nexxt w)
                                      +
                                      \dist(\Zunext, \this w) = 0 + \norm{2-\thisu}_U \le \norm{\nextu-\thisu}_U,
                                  \]
                                  which immediately verifies \cref{ineq:ex-adjoint-proof}.
                            \item $(2 - \nextu) \in \Zunext,$ then (if $0\notin\Zunext$) the left side of \cref{ineq:ex-adjoint-proof} is at most
                                  \begin{align*}
                                      \kappa_w\norm{2 - \nextu}_U + \norm{\nextu - \thisu}_U
                                       &
                                      \le
                                      \kappa_w\norm{2 - \thisu}_U + (1+\kappa_u)\norm{\nextu - \thisu}_U
                                      \\
                                       & =
                                      \kappa_w\inv\sigma\norm{\thisx - \prev x} + (1+\kappa_u)\norm{\nextu - \thisu}_U
                                  \end{align*}
                                  where the last equality follows from \cref{ex:item-ii}.
                                  Therefore, \cref{ineq:ex-adjoint-proof} holds for $\pi_w \ge \inv\sigma\kappa_w$ .
                                  \qedhere
                        \end{enumerate}
              \end{enumerate}
    \end{enumerate}
\end{proof}

\begin{remark}
    If the inner iterates $\nextu$ are good enough approximations of solutions $S_u(\thisx),$ the error is less than $2,$ the case can't happen. Then we can take less strict assumptions for $\pi_w,$ i.e. $\pi_w = \inv\sigma (\kappa_w -1)$ is enough.
\end{remark}

Finally, we can prove the assumptions of our main convergence \cref{thm:main-convergence}.

\begin{theorem}
    Let $J, R, f, g$, and $G$ be as in the example problem \cref{eq:simple-example,eq:G-example}, and $E$ be given by \cref{eq:example:F}, i.e $E=J\circ S_u$ for the inner solution mapping $S_u$ satisfying $0\in G(S_u(x),x)$.
    Moreover, let
    \begin{itemize}
        \item the inner iterate $\nextu$ be produced by the forward-backward method with step length $\tau \in (0,1)$, i.e., $\nextu = \prox_{\tau g(\freevar,\thisx)}(\thisu - \tau\grad_u f(\thisu,\thisx)),$
        \item the adjoint iterate $\nexxt w\in Z(\nextu, \thisx, \inv\tau(\nextu - \thisu))$ for  $Z$ given by \cref{ex:Z-corollary}, 
        \item the outer iterate $\thisx=\proj_{[0,\infty)}(\prev x - \sigma x^*_k)$ be produced with the forward-backward method for $\sigma > 0$ and the differential estimate $\widetilde x^*_k$ arising from $(-J'(\thisu), \widetilde x^*_k)\in D^*G(\thisu, \prev x| \inv\tau(\nextu - \thisu))(\this w).$
    \end{itemize}
    Assume the step lengths $\tau, \sigma$ and the constants $\mu_u, \pi_w>0$ and $\kappa_w>1$ to satisfy $ \mu_u \ge 1 + \kappa_w,$ \\ $(1+\kappa_w)\norm{\freevar}_U \le \mu_u \inv\tau\norm{\freevar}_*, \pi_w \ge \inv\sigma\kappa_w$,
    and $\sigma (2+\trackingressum[1]^2) < 2$ where $\kappa \defeq\min\{\kappa_u, \kappa_w\}>1$ and $\trackingressum[1]>0$ satisfies the bound \cref{eq:tracking:ressum}.
    Then $\norm{\nextu - S_u(\thisx)}_U \to 0$ and $\inf_{x^* \in [\COsubdiffS E+\subdiff R](\nextx)}\norm{x^* }\to 0.$
\end{theorem}

\begin{proof}
    The condition $\inf[E+R]>-\infty$ holds because $E, R \ge0$.
    We can take $\Omega_X = [\epsilon, \infty),$ which implies the condition \cref{eq:main-convergence:locality-lemma}.
    Moreover, the set-valued descent inequality \cref{ineq:set-valued-descent-inexact} holds for $\COsubdiffS E= \COsubdiffS_2 E$ and $L=1$ by \cref{lemma:ex-set-valued-descent}. Therefore, we can take $\COsubdiffS_2 E(\thisx)$ given by \cref{ex:Z-corollary} as the target sets.
    Our $X=\R$ is also Hilbert space.

    We next consider \cref{ass:tracking:main}:
    \begin{itemize}
        \item  Since $S_u$ is Lipschitz continuous (and single-valued) by \cref{ex:S-lemma} and the forward-backward method satisfies the basic contractivity property \cref{ineq:linear-inner-convergence} by \cref{thm:subreg:convergence-result-sub-peb:general}, \cref{thm:inner-tracking} implies the inner tracking property, \cref{ass:tracking:main}\,\cref{item:tracking:main:inner-tracking}
        \item  We verify the adjoint tracking property, \cref{ass:tracking:main}\,\cref{item:tracking:main:adjoint-tracking}, in \cref{lemma:example-adjoint-tracking}.
        \item \cref{theorem:example-adjoint,ex:Z-corollary} demonstrate that we have $\nextEstE = \nexxt w\in\Zunext$ and $\COsubdiffS_2 E(\thisx)\ni\targetE=S_w(\thisx)\in Z(\thisx).$ Therefore, taking $d_{X^*} = \norm{\freevar}$, we see that
        \[
            \norm{\nextEstE - \targetE} \le
            \norm{\nexxt w - S_w(\thisx)}_W.
        \]
        This shows that the differential transformation \cref{ass:tracking:main}\,\cref{item:tracking:main:differential-transformation} holds for $\alpha_u=0$ and $\alpha_w=1.$
    \end{itemize}
    Therefore, \cref{ass:tracking:main} holds.


    We finally verify the condition \cref{item:convergence:main:subdiff}\cref{item:main-convergence:F} of \cref{thm:main-convergence}, i.e., for all $\epsilon>0$ there exists $\delta>0$ such that, for some $k_N\in \N$ for any $k \ge k_N,$
    \begin{equation}
        \label{ineq:set-valued-descent-cont-ex}
        \norm{\nextx-\thisx} \le \delta
        \implies
        \dist(\COsubdiffS E(\nextx), \targetE) \le \epsilon.
    \end{equation}
    If either $\thisx>5$ or $\thisx<5$ for all $k\ge k_N,$ the implication \cref{ineq:set-valued-descent-cont-ex} holds since $\COsubdiffS E$ is Lipschitz-continuous by \cref{ex:Z-corollary}.
    We consider next the case $\COsubdiffS E(\nextx) = 0$ and $\targetE = x-3 = 2-S_u(\thisx),$ i.e. $\nextx\ge 5$ and $\thisx \le 5$ with $\nextx\neq\thisx.$
    By \cref{lemma:main-convergence} we have $\nextu - S_u(\thisx) \to 0$ and thus for $k$ large enough also $\nextEstE\approx\targetE.$
    However, this implies $\norm{\nextx-\thisx} = \norm{\sigma \nextEstE} \approx \sigma (5 - 3) = 2\sigma,$ which is contradiction to $\norm{\nextx-\thisx} \le \delta.$ The case $\COsubdiffS E(\nextx) = x-3 = 2-S_u(\thisx)$ and $\targetE = 0$ produces contradiction similarly.
    Therefore, \cref{ineq:set-valued-descent-cont-ex} must hold for some $k_N\in \N$ for any $k \ge k_N.$
    We have proved that the assumptions of \cref{thm:main-convergence} hold, and thus the claim follows from it.
\end{proof}

\begin{remark}
    Note, however, that since $E=J\circ S_u$ is nonconvex, the convergence given by \cref{thm:main-convergence}, $\inf_{\bar x^*\in [\COsubdiffS E+\subdiff R](\nextx)}\norm{\bar x^*}\to 0,$ could be satisfied in local optima, in $[5,\infty)$ where $E$ is constant.
\end{remark}

\section{Total variation regularised inverse problems}
\label{sec:tv}

In this section, we base the inner problem on the ill-posed inverse problem
\[
    A_x u = m + \nu,
\]
where $A_x$ is a forward operator corresponding to the (hyper)parameters $x$.
It maps the unknown $u$ into corresponding measurable data.
The measurement is $m$, however, it may deviate from the true data by noise $\nu$.
Since the dimension of the measurements $m$ is likely smaller than the dimension of $x$, the system will not have a unique solution. The solution may also be highly sensitive to the noise.
Incorporating a data fidelity $d$ to model the noise, and adding (discretised) total variation regularisation to introduce our preconceptions of a good solution, we obtain the inner optimisation problem
\begin{equation}
    \label{eq:tv:inverse-problem}
    \min_u d(A_x u, m) + C(x)\norm{Ku}_{2,1}
\end{equation}
Here $K$ is a discretised differential operator.
The bilevel problem with \cref{eq:tv:inverse-problem} as the inner problem arises from the need to choose an optimal regularisation parameter $C(x)$ and a measurement operator $A_x$ with respect to the outer objective. The corresponding outer problem reads as
\begin{equation}
    \label{eq:bilevel-inverse-problem}
    \min_x J(u) + R(x)  \quad \text{such that $u$ solves \cref{eq:tv:inverse-problem}.}
\end{equation}
We recall that we write $F \defeq J \circ S$, where $S: x \mapsto u$ is the solution mapping of \eqref{eq:tv:inverse-problem}.

The regularisation term $\norm{Ku}_{2,1}$ is nonsmooth in $u$, and although convex, its proximal map is hard to evaluate.
We therefore use the primal-dual algorithm of \cref{ex:inner-pdps} for the solution of \eqref{eq:tv:inverse-problem}.
\Cref{ex:inner-pdps,ex:outer-fb} provide generic inner and outer algorithms for the bilevel problem \cref{eq:bilevel-inverse-problem}.
However, the adjoint algorithm needs to be specifically developed for this case.
That is the goal of this section.

First, in \cref{sec:tv:general} we provide general results on the adjoint inclusion, where both the data and regularisation terms in the inner problem can depend on the parameter $x$.
Then, in \cref{sec:tv:data}, we refine the results for the case where only the data term depends on $x$. This makes the differential transformation \cref{ass:tracking:main}\,\cref{item:tracking:main:differential-transformation} easier to satisfy as discussed in
\cref{sec:differential-transformation}.
All the spaces throughout this section are Hilbert spaces, so we are justified in identifying the spaces with their dual spaces. In particular, the inner optimality condition $G:U\times X\setto W_*$ where we can take $W_*=W=U$, because $G(\freevar, x)$ is a subdifferential in a Hilbert space. Thus also the coderivatives $\anyCod G(u,x|w): U \setto U\times X.$

\subsection{General results for the adjoint inclusion}
\label{sec:tv:general}

We consider first more general inner problem (primal) formulation that encompasses \cref{eq:tv:inverse-problem},
\begin{equation}
    \label{eq:sum-of-convex-objective}
    \min_{u_p} f(u_p, x) + g(Ku_p, x),
\end{equation}
with $f$ and $g$ convex functions.
The corresponding bilevel problem is
\begin{equation}
    \label{eq:primal-dual-bilevel-problem}
    \min_{u\in U_p\times U_d, x\in X} J(u) + R(x)\quad\text{subject to}\quad 0 \in G(u, x) \defeq
    \begin{pmatrix}
        \grad_{u_p} f(u_p, x) + K^*u_d
        \\
        -Ku_p + \subdiff_{u_d} g^*(u_d, x)
    \end{pmatrix},
\end{equation}
where the primal-dual optimality condition $0 \in G(u,x)$ arises from \cref{eq:sum-of-convex-objective} as discussed in \cref{ex:inner-pdps}.

The next lemma analyses the adjoint inclusion
\begin{equation}
    \label{eq:tv:adjoint-incl}
    (-u^*, x^*)\in \anyCod G(u,x|\Delta w_*)(w)
\end{equation}
for $\anyCod = \frechetCod$ or $\anyCod=D^*$.
For the basic adjoint inclusion in \eqref{eq:algorithm:full-oc}, we take
\begin{equation}
    \label{eq:tv:ustar-choices}
    u^*=J'(u).
\end{equation}
For the general step \eqref{eq:algorithm:abstract:adjoint} with $u=\nextu$, we would have $u^*=J'(\nextu)-\Delta u_{k+1}^*$ for a perturbation $\Delta u_{k+1}^* \approx 0$.
In practise, we will have $\Delta u_{k+1}^*$ arise implicitly from the inexact solution of \eqref{eq:tv:adjoint-incl} for $u^*=J'(\nextu)$.

\begin{lemma}
    \label{lemma:spesific-adjoint-general}
    Let $U = U_p\times U_d$ and $X$ be Hilbert spaces, $K\in \linear(U_p;U_d)$ and $G:U\times X \setto U$.
    Let $G$ be as in \eqref{eq:primal-dual-bilevel-problem} for $f:U_p\times X \to \R$ twice continuously differentiable and $g^*:U_d\times X \setto \extR$ convex, proper, and lower semicontinuous.
    Pick $x, x^*\in X$,
    \[
        u = (u_p, u_d) \in U,
        \quad
        \Delta w_* = (\Delta w_p, \Delta w_d)\in G(u,x),
        \quad\text{and}\quad
        w=(w_p, w_d)\in U.
    \]
    Then, for any $u^*\in U$ and $\anyCod = \frechetCod$ or $\anyCod=D^*$, the adjoint inclusion \eqref{eq:tv:adjoint-incl} holds, i.e., $(-u^*, x^*)\in \anyCod G(u,x|\Delta w_*)(w),$ if and only if
    \begin{subequations}
        \label{eq:PD-adjoint}
        \begin{align}
            \label{eq:PD-adjoint-inner}
            0
             &
            =
            u^*
            +
            \begin{pmatrix}
                \grad_{u_p}^2f(u_p, x)w_p - K^*w_d
                \\
                Kw_p
            \end{pmatrix}
            +
            \begin{pmatrix}
                0
                \\
                u_ g^*
            \end{pmatrix}
            \quad\text{and}
            \\
            \label{eq:PD-adjoint-outer}
            x^*
             &
            =
            \grad_{xu_p}f(u_p, x)w_p
            +
            x_g^*
            \\
            \shortintertext{for some}
            \label{eq:PD-adjoint-g}
            (u_g^*, x_g^*)
             &
            \in \anyCod \subdiff_{u_d} g^*(u_d,x|\Delta w_{d} +  Ku_p)(w_d) \subset U_d\times X.
        \end{align}
    \end{subequations}
\end{lemma}

\begin{proof}
    Split
    \[
        G=G_1+G_2
        \quad\text{for}\quad
        G_1(u, x)
        \defeq
        \begin{pmatrix}
            \grad_{u_p} f(u_p, x) + K^*u_d
            \\
            -Ku_p
        \end{pmatrix}
        \quad
        \text{ and }
        \quad
        G_2(u, x)
        \defeq
        \begin{pmatrix}
            0
            \\
            \subdiff_{u_d} g^*(u_d, x)
        \end{pmatrix}.
    \]
    The coderivative sum rule
    \cite[Theorem 1.62]{mordukhovich2006variational}
    gives
    \begin{equation}
        \label{eq:frechedCod-sum}
        \anyCod G(u,x|\Delta w_*)(w)
        =
        \anyCod G_2(u,x|\Delta w_* - G_1(u,x))(w)
        +
        [G_1'(u,x)]^*w.
    \end{equation}
    We have
    $$[G_1'(u,x)]^*w = ([G^{(u)}_{1}(u,x)]^*w, [G^{(x)}_{1}(u,x)]^*w)$$
    with
    \begin{align}
        \label{eq:PD-G1-derivative-inner}
        [G^{(u)}_{1}(u,x)]^*w
         &
        =
        \begin{pmatrix}
            \grad_{u_p}^2f(u_p, x) & -K^*
            \\
            K                              & 0
        \end{pmatrix}
        \begin{pmatrix}
            w_p
            \\
            w_d
        \end{pmatrix}
        =
        \begin{pmatrix}
            \grad_{u_p}^2f(u_p, x)w_p - K^*w_d
            \\
            Kw_p
        \end{pmatrix}
        \quad
        \text{ and }
        \\
        \label{eq:PD-G1-derivative-outer}
        [G^{(x)}_{1}(u,x)]^*w
         &
        =
        \begin{pmatrix}
            \grad_{xu_p}f(u_p, x) & 0
        \end{pmatrix}
        \begin{pmatrix}
            w_p
            \\
            w_d
        \end{pmatrix}
        =
        \grad_{x u_p}f(u_p, x)w_p.
    \end{align}
    We finish the proof by showing
    \begin{equation*}
        \anyCod G_2(u,x|\Delta w_* - G_1(u,x))(w)
        = \{
        ((0, u_g^*), x_g^*)\in U\times X \mid
        (u_g^*, x_g^*)\in \anyCod \subdiff_{u_d} g^*(u_d,x|\Delta w_{d} +  Ku_p)(w_d)
        \},
    \end{equation*}
    which combined with \cref{eq:frechedCod-sum,eq:PD-G1-derivative-inner,eq:PD-G1-derivative-outer} yields the claim \cref{eq:PD-adjoint}.

    By the definition of Fréchet coderivative $((u_p^*, u_g^*), x_g^*)\in  \anyCod G_2(u,x|\Delta w_* - G_1(u,x))(w)$  is equivalent to
    \begin{align}
        \label{eq:coderivative-def-G2}
        \limsup_{\this x \to x,\, \this u \to u,\, G_2(\this u, \thisx) \ni \this w_* \to w_*}
        \frac{-\iprod{w}{\this w_* - w_*}_U + \iprod{x^*_g}{\thisx-x}_X + \iprod{(u_p^*,u_g^*)}{\this u - u}_U}
        {\norm{(\thisx-x,\this u - u,\this w_* - w_*)}}
        \le 0
    \end{align}
    with $w_*=\Delta w_* - G_1(u,x).$ The operator $G_2$ doesn't depend on $u^*_p$ and thus $w_*\in G_2(u, x)$ implies $w_*\in G_2(u + \tau_k (u^*_p), x).$ Therefore we can take $u^k=u + \tau_k (u^*_p)$ with $\tau_k\to 0, \thisx=x$ and $w^k_* = w_*.$ We get that the limsup in \cref{eq:coderivative-def-G2} is at least $\iprod{u^*_p}{\tau_k u^*_p}/\norm{\tau_k u^*_p} = \norm{u^*_p},$ and thus $u^*_p=0$ is necessary for \cref{eq:coderivative-def-G2} to hold.

    Because $\Delta w_*\in G(u,x)$ it holds $w_*\in G(u,x) - G_1(u,x) = G_2(u,x)$ and $w_* = (0, \Delta w_d + Ku_p).$ Similarly the primal component of every $w^k_*$ has to be zero. Thus we obtain \cref{eq:coderivative-def-G2} is equivalent to

    \begin{align*}
        \limsup_{\substack{\this x \to x,\, \this u_d \to u_d,
        \\
                \subdiff_{u_d}g^*(u_d^k, \thisx) \ni \this w_{*,d} \to \Delta w_d + Ku_p}}
                \frac{-\iprod{w_d}{\this w_{*,d}  - (\Delta w_d + Ku_p)}_{U_d} + \iprod{x^*_g}{\thisx-x}_X + \iprod{u_g^*}{\this u_d - u_d}_{U_d}}
                {\norm{(\thisx-x,\this u_d - u_d,\this w_{*,d}  - (\Delta w_d + Ku_p))}}
        \le 0,
    \end{align*}
    which by the definition equals $(u_g^*, x_g^*)\in \frechetCod \subdiff_{u_d} g^*(u_d,x|\Delta w_{d} +  Ku_p)(w_d).$
    Since $u^*_p=0$ for all Fréchet coderivatives (of $G_2$) it also has to hold for limiting coderivatives.
\end{proof}

\begin{remark}
    To solve \eqref{eq:PD-adjoint}, we only need to find $w=(w_p, w_d)$, $u^*_g$, and $x^*_g$ satisfying \cref{eq:PD-adjoint-inner,eq:PD-adjoint-g}.
    Then \cref{eq:PD-adjoint-outer} directly produces $x^*$.
    The former two have a solution if and only if there exists a solution -- with the optimal value equal to zero -- to the optimisation problem
    \begin{align}
        \label{min:adjoint-PDPS}
        & \min_{w_p\in U_p,w_d, u^*_g\in U_d, x^*_g\in X} \frac{1}{2}\norm{u^*_p +\grad_{u_p}^2f(u_p, x)w_p - K^*w_d}^2_{U_p}
        +
        \frac{1}{2}\norm{u^*_d + Kw_p + u^*_g}^2_{U_d}
        \\
        & \text{subject to}
        \quad
        (u_g^*, x_g^*)\in \anyCod \subdiff_{u_d} g^*(u_d,x|\Delta w_{d} +  Ku_p)(w_d). \nonumber
    \end{align}
    The structure of this problem suggests to solve it with forward-backward splitting.
    We return to this in \cref{sec:tv:summary}, after forming $\anyCod \subdiff_{u_d} g^*$.
\end{remark}

\subsection{Adjoint inclusion for non-parametric regularisation term}
\label{sec:tv:data}

In the numerical experiments of \cref{sec:numerical}, we concentrate on \eqref{eq:tv:inverse-problem} with only the forward operator $A_x$ dependent on the parameter $x$.
That is, $C(x) \equiv C$ is a constant.
To apply \cref{sec:tv:general} to this setting, we take
\[
    g(u_d; x) \defeq C\norm{u_d}_{2,1}
    \defeq
    C \sum_{j=1}^n\norm{u_{d,j}}_2.
\]
Then the Fenchel conjugate
\begin{equation}
    \label{def:g-conjugate-2}
    g^*(u_d, x)
    = \sum_{j=1}^{n} \delta_{B_{\R^2}(0,C)}(u_{d,j}).
\end{equation}
Its subdifferential
\begin{equation*}
    \subdiff_{u_{d,j}} g^*(u_d, x)
    =
    \begin{cases}
        \{0\}              & \text{if } \norm{u_{d,j}} < C,
        \\
        u_{d,j}[0, \infty) & \text{if } \norm{u_{d,j}} = C,
        \\
        \emptyset          & \text{otherwise.}
    \end{cases}
\end{equation*}
Recalling \cref{lemma:spesific-adjoint-general}, we need the coderivatives of $\subdiff_u g^*$.
By separability, it suffices to study the component coderivatives.

\begin{lemmaE}
    \label{lemma:subdiff-coderivative}
    Let $z, y_*, y\in \R^2, C>0$. Then
    \begin{equation}
        \label{eq:disc-coderivative}
        \frechetCod \subdiff \delta_{B_{\R^2}(0,C)}(z| y_*)(y)
        =
        \begin{cases}
            \{0\}     & \text{if } \norm{z} < C, y_* = 0, y\in\R^2,
            \\
            z\R + \frac{\norm{y_*}}{\norm{z}}y
                      & \text{if } \norm{z} = C, y_*\in  z(0, \infty), \iprod{y}{z}=0,
            \\
            z[0,\infty)
                      & \text{if } \norm{z} = C,  y_*=0, \iprod{y}{z}\ge 0,
            \\
            \emptyset & \text{otherwise }
        \end{cases}
    \end{equation}
    and
    \begin{equation}
        \label{eq:disc-coderivative-limiting}
        D^* \subdiff \delta_{B_{\R^2}(0,C)}(z| y_*)(y)
        =
        \begin{cases}
            \{0\}     & \text{if } \norm{z} < C, y_* = 0, y\in\R^2,
            \\
            z\R + \frac{\norm{y_*}}{\norm{z}}y
                      & \text{if } \norm{z} = C, y_*\in  z[0, \infty), \iprod{y}{z}=0,
            \\
            z[0,\infty)
                      & \text{if } \norm{z} = C,  y_*=0, \iprod{y}{z}> 0,
            \\
            \{0\}
                      & \text{if } \norm{z} = C,  y_*=0, \iprod{y}{z}< 0,
            \\
            \emptyset & \text{otherwise.}
        \end{cases}
    \end{equation}
\end{lemmaE}

\begin{proofE}
    We first prove \eqref{eq:disc-coderivative}, and then use it to prove \eqref{eq:disc-coderivative-limiting}.
    By definition, we have  $z^*\in \frechetCod \subdiff  \delta_{B_{\R^2}(0,C)}(z| y_* )(y)$ if and only if
    \begin{equation}
        \label{ineq:Frechet-codiff-proof}
        \limsup_{\graph \subdiff  \delta_{B_{\R^2}(0,C)} \ni (\tilde z, \tilde y_*) \to (z, y_*)}
        \frac{
            \iprod{z^*}{\tilde z-z}
            - \iprod{y}{\tilde y_* - y_*}
        }{
            \norm{\tilde z-z}
            + \norm{\tilde y_* - y_*}
        }
        \le 0.
    \end{equation}
    We analyse this limit in various cases:

    \begin{figure}[t]
        \centering
        \begin{tikzpicture}
            \draw [black,line width=0.6mm,domain=0:360] plot[smooth] ({1.5*cos(\x)}, {1.5*sin(\x)});
            \fill (0,0) circle[radius=3pt];
            \node[inner sep=0pt] (point2) at (0.85, 1.55)
            {$\tilde z$};
            \node[inner sep=0pt] (point3) at (1.65,0.6)
            {$z$};
            \draw [red,line width=1mm,domain=20:60] plot ({1.5*cos(\x)}, {1.5*sin(\x)});
            \draw [black,line width=0.6mm,domain=0:360] plot[smooth] ({5+2*cos(\x)}, {2*sin(\x)});
            \draw [red,line width=1mm,domain=20:60] plot ({5+2*cos(\x)}, {2*sin(\x)});
            \draw [black,line width=0.5mm,dashed] (5,0) -- (6,1.7321);
            \draw [red,line width=1mm] (5.5,0.8660) -- (6.025,1.7754);
            \fill (5,0) circle[radius=3pt];
            \node[inner sep=0pt] (point4) at (5.2,0.9)
            {$\tilde y_*$};
            \node[inner sep=0pt] (point5) at (6.9,2)
            {$y'_*= \frac{\norm{y_*}}{\norm{z}}\tilde z$};
            \node[inner sep=0pt] (point6) at (7.85,0.8)
            {$y_* = \frac{\norm{y_*}}{\norm{z}}z$};
        \end{tikzpicture}
        \caption{Demonstrating the paths of the limiting progress $\graph \subdiff  \delta_{B_{\R^2}(0,C)} \ni (\tilde z, \tilde y_*) \to (z, y_*).$}
        \label{fig:cone-base}
    \end{figure}

    \begin{enumerate}
        \item
              $\norm{z} < C$ and $y_* = 0$, in which case we need to prove that $z^*=0$ and $y\in\R^2$:
              We must have $\tilde y_*=0$ to stay in $\graph \subdiff  \delta_{B_{\R^2}(0,C)}$ and thus $- \iprod{w}{\tilde y_* - y_*}=0$ for any $y\in\R^2.$
              Since we have $B(z, \varepsilon)\times\{0\}\subset \graph \subdiff  \delta_{B_{\R^2}(0,C)}$ for small enough $\varepsilon,$ and $z+\frac{1}{n}z^*\in B(z, \varepsilon)$ for $n>\inv{\epsilon}\norm{z^*},$ taking $\tilde z = z+\frac{1}{n}z^*$ and $\tilde y_* = 0$ we get
              \begin{equation*}
                  \limsup_{\graph \subdiff  \delta_{B_{\R^2}(0,C)} \ni (\tilde z, \tilde y_*) \to (z, y_*)}
                  \frac{
                      \iprod{z^*}{\tilde z-z}
                      - \iprod{y}{\tilde y_* - y_*}
                  }{
                      \norm{\tilde z-z}
                      + \norm{\tilde y_* - y_*}
                  }
                  \ge
                  \limsup_{n\to 0}\frac{\iprod{z^*}{\frac{1}{n}z^*}}{\norm{\frac{1}{n}z^*}} = \norm{z^*}.
              \end{equation*}
              Therefore, we must have $z^*=0$ for \cref{ineq:Frechet-codiff-proof} to hold.

        \item
              $\norm{z} = C$ and $y_*\in z(0, \infty)$, in which case we need to prove that $\iprod{y}{z}=0$ and $z^*\in z\R + \frac{\norm{y_*}}{\norm{z}}y$:
              We must in \eqref{ineq:Frechet-codiff-proof} have $\iprod{y}{z}=0$ for \cref{ineq:Frechet-codiff-proof} to hold because otherwise either $\tilde y_* = (1 - \varepsilon)y_*$ or $\tilde y_* = (1 + \varepsilon)y_*$ for some $\varepsilon>0$ gives $- \iprod{y}{\tilde y_* - y_*}>0$ regardless of $\tilde z$.
              Since $\iprod{y}{z}=0$, the vectors $y$ and $z$ form basis for $\R^2$ and we can write $z^* = sz + ty$ for some $s,t\in\R.$

              \cref{fig:cone-base} demonstrates the path from $\tilde z$ to $z,$ along the circle $\{v\in\R^2 \mid \norm{v} = C\},$ and how the path from $\tilde y_*$ to $y_*$ can be divided into two components, $\tilde y_* - y'_*$ and $y'_* - y_*$. For $(z,y_*), (\tilde z, \tilde y_*)\in \graph \subdiff  \delta_{B_{\R^2}(0,C)}$ it holds that
              $y'_*= \frac{\norm{y_*}}{\norm{z}}z'$ and
              $y_* = \frac{\norm{y_*}}{\norm{z}}z.$
              The limiting behaviour $(\tilde z, \tilde y_*) \to (z, y_*)$ is such that
              \begin{align*}
                   & \limsup_{\graph \subdiff  \delta_{B_{\R^2}(0,C)} \ni (\tilde z, \tilde y_*)\to (z, y_*)} \frac{\iprod{z}{\tilde z - z}}{\norm{\tilde z - z}}
                  =
                  \limsup_{\norm{\tilde z}=C, \tilde z \to z} \frac{\iprod{z}{\tilde z - z}}{\norm{\tilde z - z}}
                  =0
                  \quad
                  \text{and }
                  \\
                   & \limsup_{\graph \subdiff  \delta_{B_{\R^2}(0,C)} \ni (\tilde z, \tilde y_*)\to (z, y_*)} \frac{\iprod{y}{\tilde y_* - y'_*}}{\norm{\tilde y_* - y_*}} =
                  \limsup_{\graph \subdiff  \delta_{B_{\R^2}(0,C)} \ni (\tilde z, \tilde y_*)\to (z, y_*)} \frac{\iprod{y}{z}}{\norm{\tilde y_* - y_*}}  = 0.
              \end{align*}
              Therefore we have
              \begin{align*}
                   & \limsup_{\graph \subdiff  \delta_{B_{\R^2}(0,C)} \ni (\tilde z, \tilde y_*) \to (z, y_*)}
                  \frac{
                      \iprod{z^*}{\tilde z-z}
                      - \iprod{y}{\tilde y_* - y_*}
                  }{
                      \norm{\tilde z-z}
                      + \norm{\tilde y_* - y_*}
                  }
                  \\
                  =\,
                   &
                  \limsup_{\graph \subdiff  \delta_{B_{\R^2}(0,C)} \ni (\tilde z, \tilde y_*) \to (z, y_*)}
                  \frac{
                      \iprod{sz + ty}{\tilde z - z}
                      - \iprod{y}{\tilde y_* - y'_* + y'_* - y_*}
                  }{
                      \norm{\tilde z-z}
                      + \norm{\tilde y_* - y_*}
                  }
                  \\
                  =\,
                   &
                  \limsup_{\graph \subdiff  \delta_{B_{\R^2}(0,C)} \ni (\tilde z, \tilde y_*) \to (z, y_*)}
                  \frac{
                      \iprod{ty}{\tilde z - z}
                      - \iprod{y}{ y'_* - y_*}
                  }{
                      \norm{\tilde z-z}
                      + \norm{\tilde y_* - y_*}
                  }
                  \\
                  =\,
                   &
                  \limsup_{\graph \subdiff  \delta_{B_{\R^2}(0,C)} \ni (\tilde z, \tilde y_*) \to (z, y_*)}
                  \left(
                  t -\frac{\norm{y_*}}{\norm{z}}
                  \right)
                  \frac{
                      \iprod{y}{\tilde z - z}
                  }{
                      \norm{\tilde z-z}
                      + \norm{\tilde y_* - y_*}
                  }.
              \end{align*}
              Since
              $\limsup_{\graph \subdiff  \delta_{B_{\R^2}(0,C)} \ni (\tilde z, \tilde y_*) \to (z, y_*)}\frac{\iprod{y}{\tilde z - z}}{\norm{\tilde z - z}} = \norm{y}$
              we must have $t=\norm{y_*}/\norm{z}$ for \cref{ineq:Frechet-codiff-proof} to hold.


        \item
              $\norm{z} = C$, $x > 0$, and $y_*=0$, in which case we need to prove that $\iprod{y}{z}\ge 0$ and $z^*\in z[0,\infty)$:
              Write $z^* = sz + tv$ for some $s,t\in\R$ and $v$ such that $\iprod{z}{v}=0.$ Taking $\tilde{y}_* = y_*$ we obtain that the limsup in \cref{ineq:Frechet-codiff-proof} is bounded below by $\norm{t v}\norm{\tilde z - z}/\norm{\tilde z - z} = \norm{t v}$ and thus we must have $t=0.$ On the other hand, taking $\tilde z=z$ and $\iprod{y}{\tilde y_* - y_*} < 0$ we find again that \cref{ineq:Frechet-codiff-proof} is bounded below and thus we must have  $\iprod{y}{\tilde y_* - y_*} \ge 0.$ Since $\tilde y_* - y_* = \tilde y_* \in z[0,\infty)$ we get the equivalent condition $\iprod{y}{z} \ge 0.$
              Similarly to $\tilde y_* - y_*$ in the previous case and \cref{fig:cone-base}, we can split  $\tilde z - z$ into the  components $\tilde z - z'$ and $z' - z.$ 
              Using this split, $\tilde y_* - y_*\in z[0,\infty)$  and $\iprod{y}{z}\ge 0$, we get
              \begin{align*}
                  \limsup_{\graph \subdiff  \delta_{B_{\R^2}(0,C)} \ni (\tilde z, \tilde y_*) \to (z, y_*)}
                  \frac{
                      \iprod{z^*}{\tilde z-z}
                      - \iprod{y}{\tilde y_* - y_*}
                  }{
                      \norm{\tilde z-z}
                      + \norm{\tilde y_* - y_*}
                  }
                  =
                  \limsup_{\graph \subdiff  \delta_{B_{\R^2}(0,C)} \ni (\tilde z, \tilde y_*) \to (z, y_*)}
                  \frac{
                      -s\norm{z}\norm{z'-z}
                  }{
                      \norm{\tilde z-z}
                  }.
              \end{align*}
              This is less than zero, i.e., \cref{ineq:Frechet-codiff-proof} holds, if and only if $s\ge 0$.
        \item Otherwise: the point $(z, y_*)$ does not belong to the closed $\graph \subdiff  \delta_{B_{\R^2}(0,C)}$ and thus \cref{ineq:Frechet-codiff-proof} cannot be satisfied.
    \end{enumerate}

    The limiting coderivative equals the Fréchet coderivative with the following exceptions:
    \begin{itemize}
        \item when $\norm{z}=C$ and $\iprod{y}{z}=0$, we have $\frechetCod \subdiff\delta_{B_{\R^2}(0,C)}(z|0)(y)=z[0,\infty)$ but taking $y^k_*\in z(0,\infty)$ such that $y^k_*\to 0$ we obtain $\coderivative \subdiff\delta_{B_{\R^2}(0,C)}(z|0)(y)=z\R$ from the definition of (finite dimensional) limiting coderivative and $\frechetCod \subdiff\delta_{B_{\R^2}(0,C)}(z|y^k_*)(y)=z\R + \frac{\norm{y_*^k}}{\norm{z}}y,$

        \item when $\norm{z}=C, y_*=0$ and $\iprod{y}{z}<0$, the limiting coderivative is the singleton $\{0\}$ since $0\in \frechetCod \subdiff\delta_{B_{\R^2}(0,C)}(z^k|0)(y)$ for $z^k$ such that $\norm{z^k}<C$ and $z^k\to z.$
              \qedhere
    \end{itemize}
\end{proofE}

\begin{lemma}
    \label{lemma:multivar-coderivative}
    Let $U_d = \prod_{j=1}^n\R^2, X$ be a Hilbert space, $u_d, w_*, w_d\in U_d, x\in X.$ Moreover, let function $g^*:U_d\times X\to \extR$ be defined as in \cref{def:g-conjugate-2} 
    for some $C>0$. Then
    \begin{equation*}
        \frechetCod \subdiff_{u_d} g^*(u_d, x| w_* )(w_d)
        =
        \left(
        \prod_{j=1}^n  \frechetCod \subdiff \delta_{B_{\R^2}(0,C)}(u_{d,j}| w_{*,j})(w_{d,j})
        \right)
        \times \{0\}.
    \end{equation*}
\end{lemma}

\begin{proof}
    Follows from the definition of the Fréchet coderivative and the separable structure
\end{proof}

With \cref{lemma:subdiff-coderivative,lemma:multivar-coderivative} we can expand \cref{eq:PD-adjoint-g}, and thus write the optimisation problem \cref{min:adjoint-PDPS} in a more explicit form.

\begin{lemma}
    \label{lemma:PD-adjoint-g-alt}
    Let $U_p$ and  $X$ be Hilbert spaces, $U_d = \prod_{j=1}^n\R^2, U = U_p\times U_d$ and $K\in \linear(U_p;U_d).$
    Let $G$ be as in \eqref{eq:primal-dual-bilevel-problem} for $f:U_p\times X \to \R$ twice continuously differentiable and $g^*:U_d\times X\to \extR$ as in \cref{def:g-conjugate-2} for some $C>0$.
    Suppose $u_p\in U_p, u_d, w_d\in U_d$ and $x\in X$ as well as
    \begin{equation*}
        \Delta w_* = (\Delta w_p, \Delta w_d)\in G(u,x),
        \quad\text{i.e.,}\quad
        \,
        \Delta w_d\in-Ku_p +\subdiff_{u_d}g^*(u_d,x)
        \quad
        \text{and}
        \quad
        w_*=\Delta w_{d} +  Ku_p.
    \end{equation*}
    Moreover, let $\lambda \odot:U_d \to U_d$ and $V:U_d \to U_d,$ for fixed $u_d$ be defined by
    \begin{equation}
        \label{eq:useful-operators}
        [\lambda \odot u_d]_j = \lambda_j u_{d,j}
        \quad
        \text{and}
        \quad
        [V w_d]_j
        \defeq
        \begin{cases}
            \frac{\norm{w_{*,j}}}{\norm{u_{d,j}}}w_{d,j}
             &
            \text{if } \norm{u_{d,j}} = C, w_{*,j}\in  u_{d,j}(0, \infty),
            \\
            0
             &
            \text{otherwise}.
        \end{cases}
    \end{equation}
    Then the adjoint inclusion \eqref{eq:tv:adjoint-incl} holds for $\anyCod = \frechetCod$, i.e., $(-u^*, x^*)\in \frechetCod G(u,x|\Delta w_*)(w)$, if and only if \cref{eq:PD-adjoint-inner,eq:PD-adjoint-outer} hold with
    \begin{equation}
        \label{eq:PD-adjoint-g-alt}
        x^*_g = 0, \, u^*_g = \lambda \odot u_d + Vw_d
        \quad
        \text{with }
        \begin{cases}
            \lambda_j = 0
             & \text{if } \,\norm{u_{d,j}} < C,  w_{*,j}=0
            \\
            \lambda_j \in \R,\, \iprod{w_{d,j}}{u_{d,j}}=0
             & \text{if } \,\norm{u_{d,j}} = C, w_{*,j}\in  u_{d,j}(0, \infty),
            \\
            \iprod{w_{d,j}}{u_{d,j}}\ge 0, \lambda_j \ge 0
             & \text{if } \,\norm{u_{d,j}} = C,  w_{*,j}=0
        \end{cases}
    \end{equation}
    For $\anyCod = \frechetCod$, the adjoint inclusion \eqref{eq:tv:adjoint-incl} holds if and only if \cref{eq:PD-adjoint-inner,eq:PD-adjoint-outer} hold and
    \begin{equation}
        \label{eq:PD-adjoint-g-alt-limiting}
        \begin{cases}
            \lambda_j = 0
             & \text{if } \,\norm{u_{d,j}} < C,  w_{*,j}=0
            \\
            \iprod{w_{d,j}}{u_{d,j}}=0
             & \text{if } \,\norm{u_{d,j}} = C, w_{*,j}\in  u_{d,j}(0, \infty),
            \\
            \left\{
            \begin{array}{c}
                \lambda_j\iprod{w_{d,j}}{u_{d,j}}\ge 0
                \text{ and }
                \\
                \max(\iprod{w_{d,j}}{u_{d,j}}, \lambda_j)\ge0
            \end{array}
            \right\}
             & \text{if } \,\norm{u_{d,j}} = C,  w_{*,j}=0
        \end{cases}
    \end{equation}
\end{lemma}

\begin{proof}
    By \cref{lemma:spesific-adjoint-general}, we only need to prove the equivalence of \cref{eq:PD-adjoint-g}, i.e.,
    \[
        (u_g^*, x_g^*)
        \in \anyCod \subdiff_{u_d} g^*(u_d,x|\Delta w_{d} +  Ku_p)(w_d)
    \]
    to \eqref{eq:PD-adjoint-g-alt} or \eqref{eq:PD-adjoint-g-alt-limiting}, depending on the choice of the coderivative.
    By assumption, $\Delta w_{d} +  Ku_p \in \subdiff_{u_d}g^*(u_d,x)$.
    Therefore, \cref{lemma:multivar-coderivative} establishes that $x_g^*=0$
    and $[u_g^*]_j \in \frechetCod \subdiff \delta_{B_{\R^2}(0,C)}(u_{d,j}| w_{*,j})(w_{d,j})$ for all $j=1,\ldots,n$, when \cref{eq:PD-adjoint-g} holds.
    This reduces our task to proving that \eqref{eq:PD-adjoint-g-alt} or \eqref{eq:PD-adjoint-g-alt-limiting} gives for all $j=1,\ldots,n$ the same expression for $[u_g^*]_j$ as \cref{lemma:subdiff-coderivative} does.

    We start with the Fréchet coderivative.
    \begin{enumerate}[label=(\alph*)]
        \item $\norm{u_{d,j}} < C$ and $w_{*,j} = 0$.
              Then, by the definition in \eqref{eq:useful-operators}, $[\lambda \odot u_d + Vw_d]_j = \lambda_j u_{d,j}$.
              Now \eqref{eq:PD-adjoint-g-alt} yields $x_g^*=0$ and $u_g^*=0$.
              The equivalence now follows from \cref{lemma:subdiff-coderivative,lemma:multivar-coderivative}, which also give the only option $[u_g^*]_j=0$ for \cref{eq:PD-adjoint-g} to hold.
        \item $\norm{u_{d,j}} = C$ and $w_{*,j} \in u_{d,j}(0,\infty)$.
              Then, by the definition in \eqref{eq:useful-operators},
              $
                  [\lambda \odot u_d + Vw_d]_j = \lambda_j u_{d,j} + \frac{\norm{w_{*,j}}}{\norm{u_{d,j}}}w_{d,j}.
              $
              Now \eqref{eq:PD-adjoint-g-alt} yields in agreement with \cref{lemma:subdiff-coderivative} that
              $
                  [u_g^*]_j \in \R u_{d,j} + \tfrac{\norm{w_{*,j}}}{\norm{u_{d,j}}}w_{d,j}
              $
              provided that $\iprod{w^*_{d,j}}{u_{d,j}}=0$, with no solution otherwise.
        \item $\norm{u_{d,j}} = C$ and $w_{*,j} = 0$.
              Then, by the definition in \eqref{eq:useful-operators}, $[\lambda \odot u_d + Vw_d]_j = \lambda_j u_{d,j}$.
              Now \eqref{eq:PD-adjoint-g-alt} yields in agreement with \cref{lemma:subdiff-coderivative} that
              $
                  [u_g^*]_j
                  \in
                  [0, \infty) u_{d,j}
              $
              provided that $\iprod{w_{d,j}}{u_{d,j}}\ge0$, with no solution otherwise.
    \end{enumerate}
    For the limiting coderivative cases (a) and (b) work out exactly in the same way, while case (c) is somewhat more tedious to prove.
    Write
    \begin{equation}
        \label{eq:A-constraint-set-proof}
        \begin{split}
            A
             &
            \defeq
            \{
            [\lambda \odot u_d + Vw_d]_j \mid
            \lambda_j\iprod{w_{d,j}}{u_{d,j}}\ge 0
            \text{ and }
            \max(\iprod{w_{d,j}}{u_{d,j}}, \lambda_j)\ge0
            \}
            \\
             &
            =
            \{
            \lambda_ju_{d,j} \mid
            \lambda_j>0
            \text{ and }
            \iprod{w_{d,j}}{u_{d,j}}>0,
            \text{ or }
            \lambda_j=0,
            \text{ or }
            \iprod{w_{d,j}}{u_{d,j}}=0
            \}
        \end{split}
    \end{equation}
    and
    \begin{equation*}
        B
        \defeq
        \{
        u_{g,j}^*
        \mid
        (u_g^*, x_g^*)
        \in \anyCod \subdiff_{u_d} g^*(u_d,x|\Delta w_{d} +  Ku_p)(w_d)
        \}.
    \end{equation*}
    We have to show that $A\subset B$ and $B\subset A,$ which both require going through three cases.
    To show that $A \subset B$, let $t \in A$, and consider the cases:
    \begin{enumerate}[label=(c.\arabic*),noitemsep]
        \item $\lambda_j>0$ and  $\iprod{w_{d,j}}{u_{d,j}}>0$: then $t=\lambda_j u_{d,j}\in u_{d,j}[0,\infty)$. Third row of \cref{eq:disc-coderivative-limiting} now yields $t\in B$.
        \item $\lambda_j=0$: then $t=0$ and, since $t\in\{0\}, t\in u_{d,j}[0,\infty)$ and $t\in u_{d,j}\R + 0$, \cref{eq:disc-coderivative-limiting} yields $t\in B$.
        \item $\iprod{w_{d,j}}{u_{d,j}}=0$: then $t=\lambda_j u_{d,j} \in u_{d,j}\R$,  o the second row of \cref{eq:disc-coderivative-limiting} yields $t\in B$.
    \end{enumerate}
    To show that $B \subset A$, let $t \in B$, and consider the cases:
    \begin{enumerate}[resume*]
        \item $\iprod{w_{d,j}}{u_{d,j}}=0$: because $w_{*,j} = 0,$ \cref{eq:disc-coderivative-limiting} gives $t = \lambda_j u_{d,j}$ for $\lambda_j\in\R$. But now $t\in A$ by \cref{eq:A-constraint-set-proof}.
        \item $\iprod{w_{d,j}}{u_{d,j}}>0$: again \cref{eq:disc-coderivative-limiting} gives $t = \lambda_j u_{d,j}$ for $\lambda_j\ge 0,$ and therefore $t\in A$ by \cref{eq:A-constraint-set-proof}.
        \item $\iprod{w_{d,j}}{u_{d,j}}<0$: then \cref{eq:disc-coderivative-limiting} gives $t=0$ and so \cref{eq:A-constraint-set-proof} yields $t\in A$.
    \end{enumerate}
    We have thus proved that $u^*_g=\lambda \odot u_d + Vw_d$ with $\lambda_j$ and $w_{d,j}$ constrained as in \cref{eq:PD-adjoint-g-alt} for $\anyCod=\frechetCod$, and as in \cref{eq:PD-adjoint-g-alt-limiting} for $\anyCod=D^*$.
\end{proof}

Reformulating the third case of \eqref{eq:PD-adjoint-g-alt-limiting} into two branches, as well as replacing  \cref{eq:PD-adjoint-inner,eq:PD-adjoint-g} with the optimisation problem
\cref{min:adjoint-PDPS}, we obtain:

\begin{corollary}
    \label{cor:tv:adjoint-optim}
    Suppose the assumptions of \cref{lemma:PD-adjoint-g-alt} hold.
    Define the adjoint objective function $\mathscr{L}$ as
    \begin{equation}
        \label{eq:adjoint-objective}
        \mathscr{L}(w_p, w_d, \lambda)
        \defeq
        \frac{1}{2}\norm{u^*_p +\grad_{u_p}^2f(u_p, x)w_p - K^*w_d}^2_{U_p}
        +
        \frac{1}{2}\norm{u^*_d + Kw_p + \lambda \odot u_d + Vw_d}^2_{U_d},
    \end{equation}
    the constraint set corresponding to the Fréchet coderivative as
    \begin{equation}
        \label{eq:adjoint-constraint:frechet}
        \widehat{\mathscr{C}}_d
        \defeq
        \left\{
        \begin{array}{l}(w_d,\lambda)
            \\
            \ \in U_d \times \R^n
        \end{array}
        \,\middle|\,
        \begin{array}{ll}
            \lambda_j = 0
            & \text{if } \norm{u_{d,j}} < C,  [\Delta w_{d} +  Ku_p]_j=0
            \\
            \iprod{w_{d,j}}{u_{d,j}}=0,
            & \text{if } \norm{u_{d,j}} = C, [\Delta w_{d} +  Ku_p]_j\in  u_{d,j}(0, \infty)
            \\
            \iprod{w_{d,j}}{u_{d,j}}\ge 0, \lambda_j \ge 0
            & \text{if } \norm{u_{d,j}} = C,  [\Delta w_{d} +  Ku_p]_j=0
        \end{array}
        \right\}
    \end{equation}
    and the constraint set corresponding to limiting coderivative as
    {\small
    \begin{equation}
        \label{eq:adjoint-constraint:limiting}
        \mathscr{C}_d
        \defeq
        \left\{
        \begin{array}{l}
            (w_d,\lambda)
            \\
            \ \in U_d \times \R^n
        \end{array}
        \,\middle|\,
        \begin{array}{ll}
            \lambda_j = 0
            & \text{if } \norm{u_{d,j}} < C,  [\Delta w_{d} +  Ku_p]_j=0
            \\
            \iprod{w_{d,j}}{u_{d,j}}=0,
            & \text{if } \norm{u_{d,j}} = C, [\Delta w_{d} +  Ku_p]_j\in  u_{d,j}(0, \infty)
            \\
            \left\{
            \begin{array}{l}
                \lambda_j \ge 0 \text{ and } \iprod{w_{d,j}}{u_{d,j}}\ge 0;
                \text{ or}
                \\
                \lambda_j \le 0 \text{ and } \iprod{w_{d,j}}{u_{d,j}}=0
            \end{array}
            \right\}
            & \text{if } \norm{u_{d,j}} = C,  [\Delta w_{d} +  Ku_p]_j=0
        \end{array}
        \right\}.
    \end{equation}}%
    In the definitions of both constraint sets, $j$ runs over \, $1,\ldots,n$ in the qualifier.
    Then \eqref{eq:tv:adjoint-incl} holds, i.e., $(-u^*, x^*)\in \anyCod G(u,x|\Delta w_*)(w)$, if and only if, $x^*=\grad_{xu_p}f(u_p,x)w_p$ for some $w=(w_p,w_d)$ and $\lambda\in\R^n$ that achieve
    \begin{align}
        \label{min:adjoint-PDPS-g}
        0 = \min_{w_p\in U_p,w_d\in U_d, \lambda\in \R^n}
        \mathscr{L}(w_p, w_d, \lambda)
        \quad
        \text{subject to}
        \quad
        (w_d, \lambda) \in \breve{\mathscr{C}}_d,
    \end{align}
    where $\breve{\mathscr{C}}_d = \mathscr{C}_d$ or $\breve{\mathscr{C}}_d = \widehat{\mathscr{C}}_d$ according to the choice of the coderivative $\anyCod=D^*$ or $\anyCod=\frechetCod$.
\end{corollary}

\begin{remark}
    \label{rem:tv:adjoint-optim}
    For the Fréchet coderivative, \eqref{min:adjoint-PDPS-g} is a finite-dimensional and convex constrained quadratic problem, so has a minimiser.
    For the limiting coderivative, the problem may be nonconvex in the third case of \eqref{eq:adjoint-constraint:limiting}, but again, by considering the two convex branches individually, we obtain the existence of a minimiser.
    The main question, therefore, is if the minimum value is zero.
    This is tied to the satisfaction of \cref{eq:adjoint-tk}, which we use to guarantee the differential transformation \cref{ass:tracking:main}\,\cref{item:tracking:main:differential-transformation}.
    The latter guarantees the existence of a solution to \eqref{eq:tv:adjoint-incl}, hence to \eqref{min:adjoint-PDPS-g}.
\end{remark}

\begin{example}[Denoising as the inner problem]
    \label{ex:tv:denoising}
    Take $f(u_p,x) = \frac{x}{2}\norm{u_p - m}^2$ for some noisy measurement $m\in U_ p$ and therefore $\grad_{u_p}^2f(u_p, x)w_p = xw_p.$ Taking $J(u)=\frac{1}{2}\norm{u_p - b}^2$ for a target “true solution” $b\in U_p$, recalling \eqref{eq:tv:ustar-choices}, we have $u^*_p = u_p - b$ and $u^*_d = 0.$
    Since the optimal value of the adjoint problem \eqref{min:adjoint-PDPS-g} is zero when the original adjoint inclusion \eqref{eq:tv:adjoint-incl} has a solution, we can solve the problem for $w_p$ by setting the first term in $\mathscr{L}(w_p, w_d, \lambda)$ equal to zero.
    This yields
    \[
        w_p = \inv x(K^*w_d + b - u_p).
    \]
    The rest of the adjoint problem \cref{min:adjoint-PDPS-g} now reduces to
    \begin{align}
        \min_{w_d\in U_d, \lambda\in \R^n}
        \frac{1}{2}\norm{\inv xK(K^*w_d + b - u_p) + \lambda \odot u_d + Vw_d}^2_{U_d}
        \quad
        \text{subject to}
        \quad
        (w_d, \lambda) \in \mathscr{C}_d,
    \end{align}
    where $V$ is defined in \eqref{eq:useful-operators}.
\end{example}

\subsection{Summary}
\label{sec:tv:summary}

\begin{algorithm}
    \caption{Nonsmooth Bilevel Primal-Dual method}
    \label{alg:fb-adjoint-solver}
    \begin{algorithmic}[1]
        \Require
        On Hilbert spaces $U =U_p\times U_d$ and $X$, functions $J:U\to\R, R:X\to\extR, f =f_0 + e$ for $f_0, e: U_p \times X \to \extR$ and $g^*:U_d\times X \setto \extR$, as well as linear operator $K\in \linear(U_p;U_d),$ such that $J, R, f_0, e, g$ are convex, proper, and lower semicontinuous in their first parameter. Additionally, $J$ and $e$ have Lipschitz gradients, the latter with respect to  first parameter, with constants $L$ and $L(x)$ respectively.
        For each iteration, a way to construct the adjoint objective functions $\mathscr{L}_k:U \times \R^n\to\R,$ and constraint sets $\breve{\mathscr{C}}_{d,k}\subset U_d \times \R^n$.
        Moreover, an outer step length parameter $\tau>0,$ an adjoint step length parameter $\theta,$ inner step length parameters $\tau_p, \tau_d>0$ satisfying $\tau_p L(x^k)/2 + \tau_p\tau_d\norm{K}^2\le 1$ and a parameter $\omega>0.$
        We are also given iteration numbers  $\Ninner,\Nadjoint\in\N$ for the inner and adjoint loops respectively.
        \State Pick initial iterates $u^0\defeq(u_p^0,u_d^0)\in U,$ $v^0\defeq(w_p^0, w_d^0, \lambda^0)\in U\times\R^n,$ and $x^{0}\in X.$
        \For{$k \in \N$}
            \State
            $
            u^{k,0}
                \defeq
                u^k
            $
            \For{$i = 0,\ldots,\Ninner-1$}
                \Comment{inner PDPS}
                \State
                $u_p^{k,i+1}
                    \defeq
                    \prox_{\tau_p f_0(\freevar; \thisx)}(u_p^{k,i} - \tau_p(K^*u_d^{k,i} + \grad_{u_p} e(u_p^{k,i};\thisx))$
                \State
                $u_d^{k,i+1}
                    \defeq
                    \prox_{\tau_d g^*(\freevar; \thisx)} (u_d^{k,i} + \tau_d K((1+\omega)u_p^{k,i+1} - \omega u_p^{k,i}))$
            \EndFor
            \State
            $
                \nextu
                \defeq
                u^{k,\Ninner}
                ,
                \,
                v^{k,0}
                \defeq
                v^k
            $
            \For{$j = 0,\ldots, \Nadjoint-1 $}
                \Comment{Forward-backward adjoint solver}
                \State\label{step:alg:adjoint-loop}\vspace{-1.5ex}
                $
                    v^{k,j+1}
                    \defeq
                    \proj_{\breve{\mathscr{C}}_{d,k}}(
                    v^{k,j}
                    - \theta\grad\mathscr{L}_k(v^{k,j};u^{k+1}))$
            \EndFor
            \State
            $
                v^{k+1}
                \defeq
                v^{k,\Nadjoint}
            $
            \State
            $\nextEstF
                \defeq
                [f^{(u_p,2)}(u_p^{k+1}, \thisx)]^*w_p^{k+1}
            $
            \Comment{Form the differential estimate}
            \State
            $\nextx \defeq \prox_{\tau R}(x^k - \tau \nextEstF)$
            \Comment{Forward-backward outer step}
        \EndFor
    \end{algorithmic}
\end{algorithm}

We present in \cref{alg:fb-adjoint-solver} our overall proposed algorithm for bilevel problems \eqref{eq:intro:bilevel-problem} where $G$ arises from the primal-dual optimality conditions of \eqref{eq:sum-of-convex-objective}.
In it, on each outer iteration
\begin{enumerate}[noitemsep]
    \item we work towards a solution of the inner problem, by taking (single or multiple) primal-dual proximal splitting (PDPS) steps,
    \item work towards a root of the adjoint inclusion \eqref{eq:PD-adjoint} (with $\Delta w^{k+1}_*=-M(\nextu-\thisu)$ for $M$ defined in \cref{PDPS-M}) by taking (single or multiple) forward-backward steps on \eqref{min:adjoint-PDPS} (when $g$ is a non-parametric total variation regularisation term, \eqref{min:adjoint-PDPS-g}), and follow this by
    \item taking forward-backward steps to update the outer variable.
\end{enumerate}
Corresponding to the construction of $\nextEstF=x^*$ as an approximate solution of \eqref{eq:tv:adjoint-incl} (equivalently \eqref{eq:PD-adjoint}) for $\Delta w_*=-M(\nextu-\thisu)$ for $M$ defined in \cref{PDPS-M}) for a choice of coderivative $\anyCod=D^*,\frechetCod$, we work with the target sets $\COsubdiffS F(\thisx) = \anyCOsubdiffS_2 F(\thisx)$ (see \cref{rem:composition-subdiff}) and the target $\targetF$, solved as $x^*$ from \eqref{eq:PD-adjoint} with $\Delta w_*=0$.
The convergence of the method, in the sense that $\inf_{\bar x^*\in [\anyCOsubdiffS_2 F+\subdiff R](\nextx)}\norm{\bar x^*}\to 0$, thus follows from \cref{thm:main-convergence} subject to the satisfaction of:
\begin{enumerate}[noitemsep,label=(\roman*)]
    \item The tracking \cref{ass:tracking:main} for the inner PDPS and the forward-backward adjoint solver.
    \item The pseudo-smoothness property \eqref{ineq:set-valued-descent-inexact} on the optimistic selection $E=\infmap{F}$ or the pessimistic selection $E=\supmap{F}$, where $F= J \circ S$ for the inner solution mapping $S$.
    \item The continuity condition \cref{item:main-convergence:F} or \cref{item:main-convergence:R} in the statement of the theorem.
    \item The local initialisation and level set boundedness condition \cref{eq:main-convergence:locality-lemma}.
\end{enumerate}
These properties need to be analysed for each specific problem, and their complete verification remains outside the scope of the present already long work, where we were able to provide a suggestion of the feasibility of their satisfaction through the example of \cref{sec:simple-example}.
The continuity condition \cref{item:main-convergence:R} can always be made to hold for Lipschitz-differentiable $R$, or, e.g. barrier functions (see \cref{rem:convergence:barrier}).

As a modification of \cref{alg:fb-adjoint-solver}, recalling \cref{rem:tracking:exact}, \textbf{if} both the inner problem and adjoint are solved exactly, both \cref{ass:tracking:main}\,\cref{item:tracking:main:inner-tracking,item:tracking:main:adjoint-tracking} hold.
We can also adaptively take $\Ninner$ and $\Nadjoint$ sufficiently large that these conditions hold for prescribed parameters.

Moreover,
\begin{enumerate}
    \item
          \cref{ass:tracking:main}\,\cref{item:tracking:main:inner-tracking} follows from \cref{thm:subreg:convergence-result-sub-peb:pdps,thm:inner-tracking} subject to
          \begin{enumerate}[label=(\Roman*)]
            \item The inner solution mapping satisfying the Lipschitz-type property \eqref{eq:inner-tracking:lip-cond}, by, e.g., having the Aubin property or being Pompeui--Hausdorff--Lipschitz; see \cref{lemma:inner-tracking:aubin}, and the discussion preceding it.

            It is not uncommon for the Aubin property to hold for primal-only formulations of the inner problem, see \cite[§28]{clason2020introduction} and \cref{sec:simple-example}.
            However, the primal-dual formulation can add additional difficulties and locality restrictions.

            \item The mapping $G(\freevar; \thisx)$, as defined in \eqref{eq:primal-dual-bilevel-problem}, being metrically subregular for all $k \in \N$ with compatible neighbourhoods and moduli, as discussed in \cref{rem:subreg:chaining}.

            It is enough that each $G(\freevar, \thisx)$ is strongly monotone (outside a kernel subspace) with uniform factors.
        \end{enumerate}
    \item
          \cref{ass:tracking:main}\,\cref{item:tracking:main:adjoint-tracking} follows from \cref{thm:adjoint-ZZ} subject to:
          \begin{enumerate}[resume*]
            \item\label{item:summary:adjoint-lipschitz}
                The adjoint solution mapping satisfying the bi-Lipschitz-type property \eqref{ineq:adjoint-ZZ}.
                This requires that exact adjoint solutions exist: that \eqref{eq:adjoint-t-exact} holds.

                \Cref{lemma:adjoint-tracking:aubin} provides one simplified tool if the inner solution mapping $S$ is single-valued and Lipschitz.
                The general condition \eqref{ineq:adjoint-ZZ} makes no such requirement, as we saw in the example of \cref{sec:simple-example}.
                However, a detailed study of the condition for the total variation regularised problems remains outside the scope of the present, already long, work.

                \begin{remark}
                    \label{rem:tv:special-adjoint-lip}
                    Nevertheless, we note that, presently, by \cref{lemma:spesific-adjoint-general},
                    \[
                        \begin{split}
                        \breve Z(u, x, w_*)
                        &
                        \defeq
                        \{
                            w\in W
                            \mid
                            0 \in \anyCod G(u,x|w_*)(w) +  (J'(u), -x^*),\,
                            x^*\in X^*
                        \}
                        \\
                        &
                        =
                        \{
                            w\in W
                            \mid
                            \text{\cref{eq:PD-adjoint-inner,eq:PD-adjoint-g} hold with } u^*=J'(u)
                        \}.
                        \end{split}
                    \]
                    In the strictly complementary case, (i.e., only the first two cases of \eqref{eq:adjoint-constraint:frechet} or \eqref{eq:adjoint-constraint:frechet} occur), this can be reduced to a simple matrix system.
                    If we replace $g$ by its Moreau–Yosida regularisation (which is less than the second-order regularisation commonly required), i.e., add a quadratic to $g^*$, then it is not difficult to see that the system is non-singular provided $\grad_{u_p}^2 f(u_p, x)w_p$ and $Kw_p=0$ imply $w_p=0$. If $K$ is a discretised differential operator, this essentially means that $f$ must be sensitive to each component of $w_p$ changing by the same constant.
                    Then $\breve Z$ is single-valued and Lipschitz.
                \end{remark}

            \item
                The contractivity \eqref{ineq:linear-adjoint-convergence-lem} of the adjoint update on \cref{step:alg:adjoint-loop} of  \cref{alg:fb-adjoint-solver}.
                This requires that algorithmic adjoint solutions exist: that \eqref{eq:adjoint-tk} holds and hence that zero is reached in \eqref{min:adjoint-PDPS-g}.

                For the Fréchet coderivative, $\widehat{\mathscr{C}}_d$ of \eqref{eq:adjoint-constraint:frechet} is convex.
                Since $\mathscr{L}=\frac{1}{2}\norm{A\freevar -b}^2$ for some $A$ is linear-quadratic, $P \defeq \grad\mathscr{L} + \subdiff\delta_{\mathscr{C}}$ is $A^*A$-strongly monotone in the sense of \cref{lemma:subreg:monotonicity-subregularity}.
                Thus, that lemma with $M=\inv\theta\Id$ verifies the metric subregularity of $P$ at any of its minimisers with the global radius $\delta=\infty$.
                Now \cref{thm:subreg:convergence-result-sub-peb:fb,rem:subreg:chaining} prove \eqref{ineq:linear-adjoint-convergence-lem}.

                For the limiting coderivative, $\mathscr{C}_d$ of \eqref{eq:adjoint-constraint:limiting} is nonconvex in the third case.
                However, it is only the union of two convex sets. We call them here branches.
                We can, as in the Fréchet case, obtain contractivity on the active branch, where $\dist(\anyZunext, w^k)$ is achieved. The left hand side of \eqref{ineq:linear-adjoint-convergence-lem} can then only be improved by taking the minimum over the solutions on both branches, individually.
                Thus \eqref{ineq:linear-adjoint-convergence-lem} still holds.
        \end{enumerate}
    \item
          \cref{ass:tracking:main}\,\cref{item:tracking:main:differential-transformation} follows from \cref{thm:tracking:diff-transformation} subject to:
          \begin{enumerate}[resume*]
            \item The Lipschitz-type assumption \eqref{eq:tracking:diff-transformation:lipschitz}.

                Presently, from \cref{lemma:spesific-adjoint-general},
                \begin{equation}
                    \label{eq:tv:breve-t}
                    \begin{split}
                    \breve T(u,w,\Delta u^*,\Delta w_*; x)
                    &
                    \defeq
                    \{ x^* \in X^* \mid 0 \in \anyCod G(u,x|\Delta w_*)(w) +  (J'(u)- \Delta u^*, -x^*) \}
                    \\
                    &
                    = \{\grad_{xu_p}f(u_p, x)w_p \mid  \text{\cref{eq:PD-adjoint-inner,eq:PD-adjoint-g} hold with } u^*=J'(u)\}.
                    \end{split}
                \end{equation}
                Therefore, \eqref{eq:tracking:diff-transformation:lipschitz} can be verified from the properties of $f$.
                This may require the variables $u_p$, $w_p$, and $x$ to be bounded.
                In \cref{ass:tracking:main}, we already allow the restriction $x \in \Omega_X$ and $u \in \Omega_U$, which are then ensured by the convergence \cref{thm:main-convergence} (via \cref{lemma:main-convergence}); see \cref{remark:main-convergence}.
                The variable $w_p$ can be bounded through the adjoint optimisation problem  \eqref{min:adjoint-PDPS-g} by comparing against $(w_p, w_d, \lambda)=0$, recalling that $u^*=J'(\nextu)$.

                \item The existence of exact and algorithmic adjoint solutions, i.e., again, \cref{eq:adjoint-tk,eq:adjoint-t-exact}.

                This, again, depends on a careful analysis of the specific problem.
                In particular, under the special conditions of \cref{rem:tv:special-adjoint-lip} of \cref{item:summary:adjoint-lipschitz}, we can guarantee the existence of $w=S_w(\nextu)$ or $w=\nexxt w$.
                Then \eqref{eq:tv:breve-t} immediately establishes  \cref{eq:adjoint-tk,eq:adjoint-t-exact}.

          \end{enumerate}

          Otherwise, the structural assumptions of \cref{thm:tracking:diff-transformation} correspond to those of this section, as stated in \cref{lemma:spesific-adjoint-general}.
\end{enumerate}

The main challenge is, therefore, \eqref{ineq:set-valued-descent-inexact}, adjoint existence  \cref{eq:adjoint-tk,eq:adjoint-t-exact}, and the Lipschitz-type properties \cref{eq:inner-tracking:lip-cond,ineq:adjoint-ZZ} of the inner and adjoint solution mappings.
These need to be analysed on a case-by-case basis.

\section{Numerical results}
\label{sec:numerical}

We present two total variation regularised imaging problems and assess the performance of our \cref{alg:fb-adjoint-solver} on them.
The example problems are:
\begin{description}
    \item[denoise] Find an optimal regularisation parameter to denoise a photographic image with simulated (Gaussian) noise.
    \item[deblur] Find an optimal parametrisation of convolution kernel and regularisation parameter to
          deblur/deconvolve a photographic image with simulated blur and noise.
\end{description}
We call the evaluated algorithms:
\begin{description}
    \item[implicit] \cref{alg:fb-adjoint-solver}  with $\Ninner=500$ and $\Nadjoint  = 2000.$ With near exact initialisation of the inner and adjoint variables this algorithm solves the inner and adjoint problems near exactly within each outer iterate.

    \item[single-loop] \cref{alg:fb-adjoint-solver}  with $\Ninner=1$ and $\Nadjoint  = 3,$ i.e. the version that allows inexact solution of the inner and adjoint problems, but it also uses the same near exact initialisation of the inner and adjoint variables.
\end{description}
The goal of these numerical results is to show that an implementation our algorithm works in practice, not to demonstrate the use in actual applications, where the size and the amount of the target images would be higher.

In \cref{sec:objective functions}, we present the inner, adjoint and outer objective functions of our numerical example bilevel problems. We present numerical details in \cref{sec:numerical-setup}, and discuss the results in \cref{sec:results}.

\begin{remark}[Algorithms in the literature]
    The “implicit” method roughly corresponds to the best applicable variant of first-order algorithms found in the literature.
    Aside from stochastic single-loop algorithms for smooth and strongly convex problems developed for machine learning applications (see \cite{suonpera2022bilevel} for an overview), these methods generally depend on exact or near-exact inner and adjoint solutions, by any available means.
    Apart from \cite{villacis2025variational}, the convergence theories, including that of SparseHO \cite{bertrand2022implicit}, generally require additional smoothing.

    Also, the specific formulations treated in the literature are not always the same for the same application problem.
    Therefore, the algorithms are not necessarily directly comparable.
    In particular, our primal-dual formulation of the inner problem is rarely found in the literature. The closest practical match to our approach is given by \cite{bogensperger2025adaptively} that has significant smoothness and strong convexity requirements.

    The primal-dual transformation of the inner problem is, however, not critical in case of algorithms that require exact or near-exact inner and adjoint solutions.
    The allowed placement of the hyperparameters $x$ is much more critical.
    In particular, \cite{villacis2025variational} that has the least smoothness assumptions, only applies to inner problems of the form $\min_u f(u) + \sum_i \psi(x_i) g_i(u)$, where $f$ is smooth, and $g_i$ possibly nonsmooth.
    This excludes even our formulation \eqref{eg:denoise-inner-objective} of the denoising problem, that places $x$ in front of $f$.
    Our deblurring problem is not covered by the model of \cite{villacis2025variational}.

    As for second-order methods, \cite{villacis2025variational} also considers a nonsmooth trust region method based on the general locally Lipschitz approach of \cite{qi1994trust}.
    The specific approach of \cite{villacis2025variational} is unapplicable for the same reasons as above.
    The authors also do not compare the forward-backward and trust region methods in terms of CPU time or total computational complexity, only in terms of iterations.
    The overall method of \cite{qi1994trust} might be applicable, if we were able to show the required Lipschitz assumptions.
    Judging from our earlier experience \cite{tuomov-tgvlearn,dizonvalkonen2024tracking}, semismooth Newton's methods have significant performance challenges on the type of problems considered here, even with appropriate smoothing.
\end{remark}

\subsection{Objective functions}
\label{sec:objective functions}
For both problems we take the outer fidelity term as $J(u) = \frac{1}{2}\norm{u_p - b}_2^2$ for $b\in\R^{N^2}$ the “ground truth” image of dimensions $N\times N$.
The latter is a cropped portion of image 07 or 02 -- for denoising and deblurring respectively -- of the free Kodak dataset \cite{franzenkodak}  converted to grey values in $[0,1]$; see \cref{original-denoise,original-blur}.

For the denoising experiment our inner problem reads as
\begin{equation}
    \label{eg:denoise-inner-objective}
    \min_{u_p\in\R^{256^2}} \frac{x}{2}\norm{u_p-m}^2 + C\norm{Ku_p}_{2,1}
    \quad (x\in\R_+)
\end{equation}
with simulated measurement $m,$ see \cref{noisy-denoise}, obtained from $b$ by adding Gaussian noise of standard deviation $0.08$
and the matrix $K$ is a (discrete) backward difference operator with Dirichlet boundary conditions. We used the choice $C=0.1,$ and the outer regulariser $R=0.$
In \cref{alg:fb-adjoint-solver} we take $f_0 = 0, e(u_p,x) = \frac{x}{2}\norm{u_p-m}^2$, and $g(u,x) =  C\norm{Ku_p}_{2,1}.$
The adjoint objective is given by \cref{eq:adjoint-objective} with $u^*=\grad J(u),$ for the constraint set of \cref{eq:adjoint-constraint:limiting} corresponding to the limiting coderivative (of $G$).


For the deblurring experiment, we take as the forward operator $A_x$ the convolution with $5\times5$-kernel parametrised by $x$ as visualised in \cref{fig:numerical:deblurring-param}.
With this and the $K$ from above, the inner problem reads as
\begin{equation}
    \label{eg:deblur-inner-objective}
    \min_{u_p\in\R^{128^2}} \frac{x_1}{2}\norm{A_xu_p-m}^2 + C\norm{Ku_p}_{2,1}
    \quad (x\in\R^4_+).
\end{equation}
Write $r_{\phi}$ be the operator for rotating the image $\phi$ degrees, clockwise for $\phi>0$ and counterclockwise for $\phi<0.$
We then use in \cref{eg:deblur-inner-objective} the simulated measurement $m=r_{-1}(A_x r_1(b)) + \epsilon$, see \cref{blurred-noisy} for the true kernel weights $x_2=0.15,x_3=0.1$ and $x_4=0.75,$ as well as Gaussian noise $\epsilon$ from distribution with standard deviation $0.01.$
The aim of rotations in simulating the measurement is to provide additional modelling error to avoid inverse crime.
We take the constant $C=0.1$ and the outer regulariser $R(x) = \beta(x_2 + x_3 + x_4 - 1)^2$ with $\beta=10^4.$
The functions for the adjoint objective \cref{eq:adjoint-objective} are analogous to the denoising experiment with the difference of $f(u_p,x) = f_0(u_p,x) + e(u_p,x) = e(u_p,x) = \frac{x_1}{2}\norm{A_xu_p-m}^2.$

\begin{figure}[t]
    \centering
    \begin{subfigure}[t]{0.32\textwidth}
        \centering
        \includegraphics[width=\textwidth]{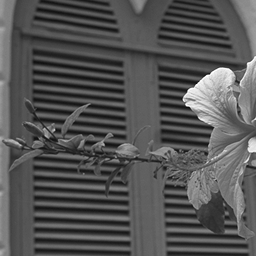}
        \caption{Original}
        \label{original-denoise}
    \end{subfigure}
    \hfill%
    \begin{subfigure}[t]{0.32\textwidth}
        \centering
        \includegraphics[width=\textwidth]{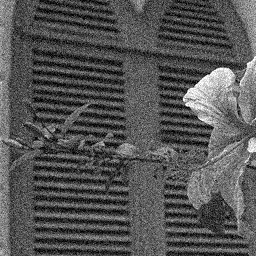}
        \caption{Noisy; error 22.1\%}
        \label{noisy-denoise}
    \end{subfigure}
    \hfill%
    \begin{subfigure}[t]{0.32\textwidth}
        \centering
        \includegraphics[width=\textwidth]{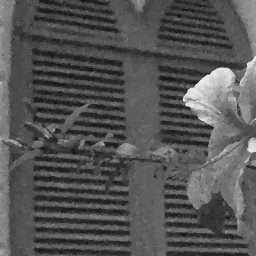}
        \caption{Result; error 10.7\%}
        \label{reg-denoise}
    \end{subfigure}
    \caption{Denoising data and result for the "single-loop" method ($N=256$). The errors in (\subref{noisy-denoise}) and (\subref{reg-denoise}) are $\norm{\text{image} - \text{target} }_2/\norm{\text{target}}_2$.}
    \label{fig:denoising}
\end{figure}

\begin{figure}[t]
    \begin{subfigure}[t]{0.24\textwidth}%
        \resizebox{\linewidth}{!}{%
            \begin{tikzpicture}[inner sep=0pt,outer sep=0pt]%
                \fill[gray!40!white] (0.6,0) rectangle (2.4,3.0);
                \fill[gray!40!white] (0,0.6) rectangle (3.0,2.4);
                \fill[blue!40!white] (0.6,1.2) rectangle (2.4,1.8);
                \fill[blue!40!white] (1.2,0.6) rectangle (1.8,2.4);
                \fill[red!40!white] (1.2,1.2) rectangle (1.8,1.8);
                \node at (1.5, 1.5) {$x_2$};
                \node at (1.5, 2.1) {$x_3$};
                \node at (2.1, 2.1) {$x_4$};
                \node at (0.3, 0.3) {$0$};
                \node at (0.3, 2.7) {$0$};
                \node at (2.7, 0.3) {$0$};
                \node at (2.7, 2.7) {$0$};
                \draw[step=0.6cm,black,very thin] (0,0) grid (3.0,3.0);
            \end{tikzpicture}
        }%
        \caption{Kernel structure}
        \label{fig:numerical:deblurring-param}
    \end{subfigure}
    \hfill%
    \begin{subfigure}[t]{0.24\textwidth}%
        \centering
        \includegraphics[width=\textwidth]{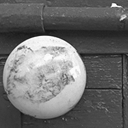}
        \caption{Original}
        \label{original-blur}
    \end{subfigure}
    \hfill%
    \begin{subfigure}[t]{0.24\textwidth}
        \centering
        \includegraphics[width=\textwidth]{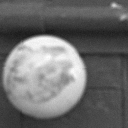}
        \caption{Blurry; error $9.8\%$}
        \label{blurred-noisy}
    \end{subfigure}
    \hfill%
    \begin{subfigure}[t]{0.24\textwidth}
        \centering
        \includegraphics[width=\textwidth]{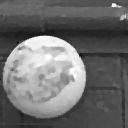}
        \caption{Result; error $6.9\%$}
        \label{reg-blur}
    \end{subfigure}
    \caption{Deconvolution kernel parametrisation, data, and result for the "single-loop" method ($N=128$). The different colours in (\subref{fig:numerical:deblurring-param}) represent different components of $x$ such that elements of kernel with same colour have same value. The errors in (\subref{blurred-noisy}) and (\subref{reg-blur}) are $\norm{\text{image} - \text{target} }_2/\norm{\text{target}}_2$.
    }
    \label{fig:deblurring}
\end{figure}
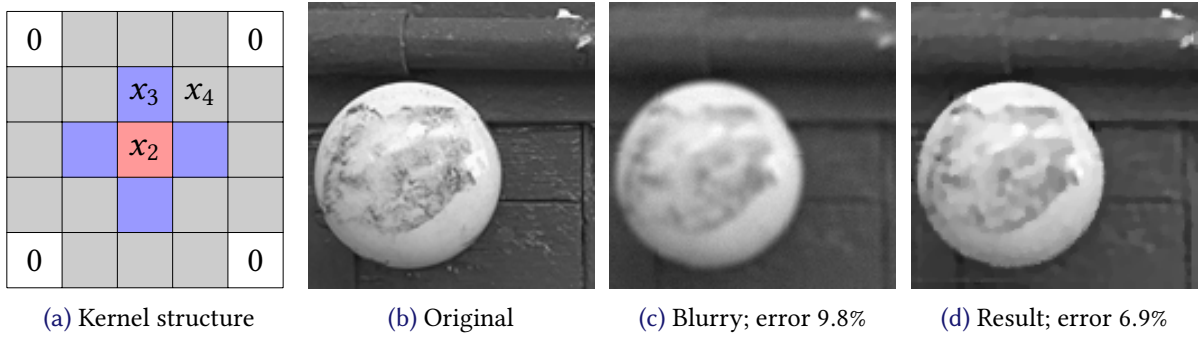

\subsection{Numerical setup}
\label{sec:numerical-setup}

Our algorithm implementation (code files) will be available on Zenodo for the publication of this article. The choices for adjoint and outer step lengths for both experiments are listed in \cref{tab:setup}. We picked inner step lengths $\tau_p, \tau_d > 0$ satisfying $\tau_p L(x)/2 + \tau_p\tau_d \norm{K}^2\le 1$ for PDPS with the upper bound of $\norm{K}^2\le 8$ from \cite{chambolle2004algorithm}. The Lipschitz constant $L(x)$ depends on the outer variable (iterates) since the first terms of the inner objectives from \cref{eg:denoise-inner-objective,eg:deblur-inner-objective} also has dependency on it.
All the other step lengths, numbers of steps
for solving inner problem and adjoint system to a high precision in implicit method, and parameters,
such as $\omega = 1,$ are chosen by trial and error to obtain an apparently stable, but as efficient as possible, algorithm. We do not attempt to verify their theoretical conditions.

We choose the adjoint system that corresponds to limiting coderivative (of $G$) and essentially the constraint set $\mathscr{C}_{d,k}$ given by \cref{eq:adjoint-constraint:limiting}. We also tried the system corresponding to Fréchet coderivative, but observed no significant difference. We took the former since we considered the potentially empty outer objective differential more serious issue than the potential issues with the convergence of the adjoint algorithm from nonconvexity of the constraint set.

\begin{remark}
    To compute the constraint set  $\mathscr{C}_{d,k}$ given by \cref{eq:adjoint-constraint:limiting}, we require $\Delta w_d^{k+1} + Ku_p^{k+1}$.
    We can avoid calculating this explicitly:
    We have $[\Delta w^{k+1}_d + K\nextu_p]_j \in \subdiff \delta_{B_{\R^2}(0,C)}(\nextu_{d,j}),$  which implies that $[\Delta w^{k+1}_d + K\nextu_p]_j= v_j \nextu_{d,j}$ for some $v_j\ge 0$.
    From the treatment of the PDPS as inner method \cref{ex:inner-pdps}, we have $(\Delta \nexxt w_p, \Delta \nexxt w_d) = \Delta w_*^{k+1}=-M(\nextu - \thisu)$. Thus
    \[
        (u_d^k + \tau K(2u_p^{k+1}- u_p^k))_j
        =
        (1 + \tau v_j)u_{d,j}^{k+1}.
    \]
    On the other hand, the dual update of the inner PDPS yields
    \[
        u_{d,j}^{k+1}
        = \proj_{B(0,1)}([u_d^k + \tau K(2u_p^{k+1}- u_p^k)]_j)
    \]
    These two equations allow us to solve for $v_j$ and then $[\Delta w^{k+1}_d + K\nextu_p]_j= v_j \nextu_{d,j}$.
    This is practical because we have to compute $\norm{[u_d^k + \tau K(2u_p^{k+1}- u_p^k)]_j}$ within the inner step to perform the projection onto $B(0,1).$
\end{remark}

The initial iterates for the outer variable were taken as $x^{(0)}=4$ for the denoise experiment and $x^{(0)}=(10, 0.25, 0.25, 0.5)$ for the deblurring experiment. The initial iterates for the inner and adjoint variables were obtained by running one iteration of outer loop of \cref{alg:fb-adjoint-solver} with zero initialisation, $\Ninner = 10 000$ and $\Nadjoint =50 000.$

\begin{table}[tb]
    \caption{
        Algorithm parametrisation, time multiplier, and total outer steps to reach threshold computational resources (20000 CPU time) value.
        The time multipliers allow conversion from computational resources to seconds.
        It varies between algorithms and problems due to different levels of parallelisability.
        The ‘inner steps’ and ‘adjoint steps’ refer to $\Ninner$ and $\Nadjoint $ in \cref{alg:fb-adjoint-solver}, which indicate the number of iterations taken towards a solution of the inner problem or the adjoint inclusion on every outer iteration.
        Step lengths for PDPS (inner steps) varied for iterations, because the relevant Lipschitz variable depends on outer iterate.
    }
    \label{tab:setup}
    \centering
    \begin{NiceTabular}{ccrrrrcr}
         &             & outer           & inner & adjoint      & time   &                  &
        \\
         & method      & steps           & steps & steps        & mult.  & $\theta$         & $\tau$
        \\
        \toprule
        \Block{2-1}{denoise}
         & implicit    & $7.3\cdot10^2$  & $500$ & $2\cdot10^3$ & $0.31$ & $5\cdot 10^{-3}$ & $3\cdot 10^{-4}$
        \\
         & single-loop & $4.4\cdot 10^5$ & $1$   & $3$          & $0.29$ & $5\cdot 10^{-3}$ & $1\cdot 10^{-6}$
        \\
        \midrule
        \Block{2-1}{deblur}
         & implicit    & $1.7\cdot 10^3$ & $500$ & $2\cdot10^3$ & $0.23$ & $5\cdot 10^{-4}$ & $1\cdot 10^{-2}$
        \\
         & single-loop & $7.8\cdot 10^5$ & $1$   & $3$          & $0.23$ & $5\cdot 10^{-4}$ & $5\cdot 10^{-4}$
        \\
        \bottomrule
    \end{NiceTabular}
\end{table}

To compare algorithm performance, we plot norms of differential estimates and values of outer objective as function
of the CPU time value of Matlab on an AMD Ryzen 5 5600H CPU. We call this value “computational
resources”, since it measures the use of
several CPU cores by Matlab’s internal linear algebra.
Thus CPU time is more accurate indicator of the actual resources than the
elapsed real time.

\begin{figure}[t]
    \centering
    \input{Tikz_image3.tex}
    \caption{Graph of composed objective for the denoising problem along with near-optimal $\tilde x$ found by recursive subdivision.
        Subfigure (\subref{fig:denoising:diffest}) presents the estimate of the differential target set $\COsubdiffS_2 F$ defined in \cref{def:set-valued-target}. The estimate is obtained as
        $\widetilde x^*_1$ given by \cref{alg:fb-adjoint-solver} with $\Ninner = 10000, \Nadjoint =50000$, and zero initialisation of iterates. The values for composed objective function in subfigure (\subref{fig:denoising:values}) are obtained using the same inner solution estimate than for the differential target estimate.
        In (\subref{fig:denoising:diffest}), the dashed blue line indicates negative values, and the red continuous line, positive values. This numerical evidence also confirms that the minimum of $J\circ S_u$ is characterised by a zero of $\COsubdiffS_2 F.$
    }
    \label{fig:numerical:J_grid128}
\end{figure}

\subsection{Results}
\label{sec:results}

The numerical evidence presented in \cref{fig:numerical:performace} suggest that both methods converge in terms of both function values, and the infimal norm of the coderivative. This numerical convergence is observed even if we can't confirm all of the assumptions for the convergence \cref{thm:main-convergence}. Especially for the single-loop method the validation of the assumptions would require additional work. The convergence seems to be even linear, at least in the infimal coderivative norm. This is reasonable for denoising since by \cref{fig:numerical:J_grid128} the outer objective resembles a locally strongly convex function around the minimum.

With limiting coderivatives, the possibly nonconvex constraint set $\mathscr{C}_{d,k}$ could in principle be problematic in the adjoint algorithm. However, in practise, we did not observe any problems: the convergence behaviour of adjoint algorithms with limiting coderivatives was not significantly different from   the Fréchet coderivative, for which the set $\mathscr{C}_{d,k}$ is convex.

\begin{figure}[t]
    \centering
    \input{Tikz_image_plots.tex}
    \caption{Performance of the compared methods. Top row: (norm of) differential estimate convergence, for denoising $R=0,$ and bottom row: outer function value convergence.
        The “computational resources” is the spent CPU time over multiple cores, parallelisation level depending on algorithm.
    }
    \label{fig:numerical:performace}
\end{figure}
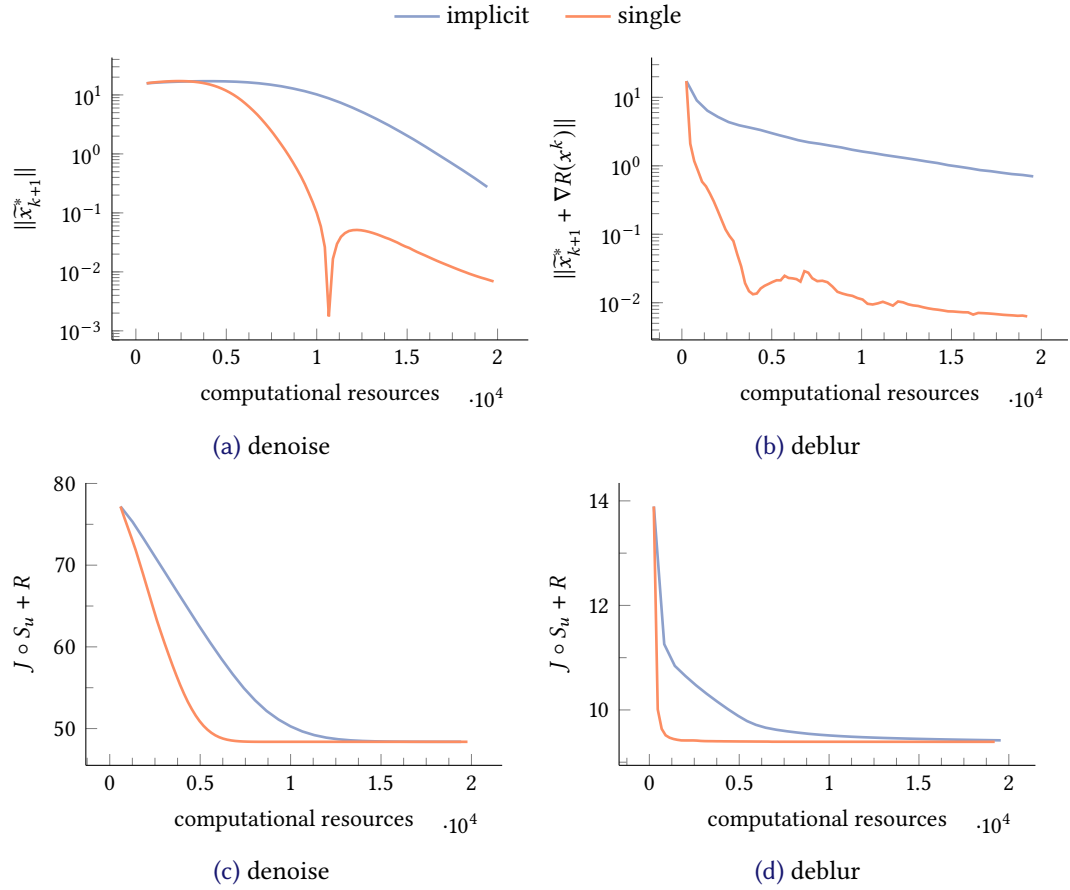

To conclude, our numerical experiments confirm that our proposed single-loop method converges, and that the convergence is faster than that of the (near-exact) implicit method.
The theoretical gaps outlined in \cref{sec:tv:summary} are a rich source of research questions for the future.


\input{nonsmooth-bilevel.bbl}

\appendix

\AtEndtrue

\section{Contractivity under metric subregularity}

\input{inner-lemmas}

\section{Proofs of auxiliary results for total variation}

Here we prove further auxiliary results that we did not prove in the main text.

\printProofs

\section{Scalar tracking results}
\label{sec:scalar-tracking}

We recall from \cite{dizonvalkonen2024tracking} the following scalar tracking results.
In our work, we take $\distU{k} = \norm{\thisu - S_u(\prev x)}_U + \norm{\Delta w_*^{k}}_* , \distW{k} = \norm{w^{k} - S_w(\prev x)}_W, \thisDistXprev = \norm{\thisx - \prev x}$ and $  \scalarTrackingErrorThis =  \norm{\nextEstE - \targetE}$.

\input{appendix-tracking}

\end{document}

%% file: symbols.tex
\newcommand{\term}{\emph}

\newcommand{\field}[1]{\mathbb{#1}}
\newcommand{\N}{\mathbb{N}}

\newcommand{\R}{\field{R}}
\newcommand{\extR}{\overline \R}
\newcommand{\B}{B}

\newcommand{\norm}[1]{\|#1\|}

\newcommand{\abs}[1]{|#1|}
\newcommand{\adaptabs}[1]{\left|#1\right|}

\newcommand{\inv}[1]{#1^{-1}}
\newcommand{\grad}{\nabla}
\newcommand{\freevar}{\,\boldsymbol\cdot\,}
\newcommand{\Union}\bigcup
\newcommand{\Isect}\bigcap
\newcommand{\union}\cup
\newcommand{\isect}\cap
\newcommand{\bigunion}\bigcup
\newcommand{\bigisect}\bigcap

\newcommand{\defeq}{:=}

\newcommand{\downto}{\searrow}
\newcommand{\upto}{\nearrow}
\newcommand{\subdiff}{\partial}

\DeclareMathOperator*{\argmin}{arg\,min}

\DeclareMathOperator{\interior}{int}

\DeclareMathOperator{\closure}{cl}
\DeclareMathOperator{\bd}{bd}

\DeclareMathOperator{\Dom}{dom}

\DeclareMathOperator{\graph}{Graph}

\DeclareMathOperator{\dist}{dist}
\DeclareMathOperator{\prox}{prox}
\DeclareMathOperator{\proj}{proj}
\DeclareMathOperator{\range}{ran}

\newcommand{\iprod}[2]{\langle #1,#2\rangle}

\def \weaktostarSym{\setbox0=\hbox{$\rightharpoonup$}\rlap{\hbox
        to\wd0{\hss\raise1ex\hbox{$\scriptscriptstyle{*\,}$}\hss}}\box0}
\def \weaktostar    {\mathrel{\weaktostarSym}}

\def\linear{\mathbb{L}}
\newcommand{\setto}{\rightrightarrows}

\def\extR{\overline \R}

\def\dualprod#1#2{\langle #1|#2\rangle}

\let\phi=\varphi
\let\epsilon=\varepsilon

\DeclareMathOperator{\Id}{Id}

\def\nexxt#1{#1^{k+1}}
\def\this#1{#1^k}
\def\prev#1{#1^{k-1}}
\let\opt\widehat

\def\nextu{\nexxt{u}}
\def\nextx{\nexxt{x}}

\def\thisu{\this{u}}
\def\thisx{\this{x}}

\def\dir#1{\Delta#1}

\newcommand{\frechetNormal}{\widehat N}
\newcommand{\frechetCod}{\widehat D^*}

\newcommand{\coderivative}{D^*}

\def\diffwrt#1#2{#1^{(#2)}}

\def\MAX#1{\overline{#1}}

\def\primaldifffact{\mu_u}

\def\trackingres{\psi}
\newcommand{\trackingressum}[1][p]{\varsigma_{#1}}

\def\localset{\Omega_X}

\newcommand{\dis}[4][]{%
    d_{#2}%
    \ifthenelse{\equal{#1}{2}}{^2}{}%
    (#3, #4)%
}

\newcommand{\adaptdis}[4][]{%
    d_{#2}%
    \ifthenelse{\equal{#1}{2}}{^2}{}%
    \left(#3, #4\right)%
}

\def \approxinSym{\setbox0=\hbox{$\in$}\rlap{\hbox
        to\wd0{\hss\raise1ex\hbox{$\sim$}\hss}}\box0}

\def\COsubdiffS{\mathscr{d}}

\def\nextEstF{\widetilde x^*_{k+1}}
\def\nextEstE{\widetilde x^*_{k+1}}
\def\targetF{x^*_{k+1}}
\def\targetE{x^*_{k+1}}

\def\anyCod{{\breve D^*}}
\def\anyCOsubdiffS{{\breve{\mathscr{d}}}}

\def\weakstarlimsup{\mathop{\operatorname{w-\!\ast\!-\kern.07em lim\,sup}\,}}

\def\optx{\opt x}
\def\optu{\opt u}

\def\Ninner{N_{\mathrm{inner}}}
\def\Nadjoint{N_{\mathrm{adjoint}}}

\def\varitys#1{}

\def\distU#1{{\varitys{SteelBlue!70}d^u_{#1}}}
\def\distUsq#1{{\varitys{SteelBlue!70}(d^u_{#1})^2}}
\newcommand{\nextDistU}[1][k]{\distU{#1+1}}
\newcommand{\thisDistU}[1][k]{\distU{#1}}
\def\initDistU{\distU{1}}
\def\initDistUsq{\distUsq{1}}

\def\distW#1{{\varitys{ForestGreen!70}d^w_{#1}}}
\def\distWsq#1{{\varitys{ForestGreen!70}(d^w_{#1})^2}}
\newcommand{\nextDistW}[1][k]{\distW{#1+1}}
\newcommand{\thisDistW}[1][k]{\distW{#1}}
\def\initDistW{\distW{1}}
\def\initDistWsq{\distWsq{1}}

\newcommand{\distX}[1]{{\varitys{orange!70}\varrho_{#1}}}
\newcommand{\thisDistXprev}[1][k]{\distX{#1}}
\newcommand{\nextDistXthis}[1][k]{\distX{#1+1}}
\newcommand{\nextDistXthisSq}[1][k]{{\varitys{orange!70}\varrho_{#1+1}^2}}
\newcommand{\InstthisDistXprev}[1][k]{{\varitys{orange!70}b_X(\thisx, \prev x)}}
\newcommand{\InstnextDistXthis}[1][k]{{\varitys{orange!70}b_X(\nextx, \thisx)}}
\newcommand{\InstnextDistXthisSq}[1][k]{{\varitys{orange!70}b_X^2(\nextx, \thisx)}}
\newcommand{\InstnextDistZthis}[1][k]{{\varitys{orange!70}b_Z(\nexxt z, \this z)}}

\def\scalarTrackingError#1{{\varitys{magenta!70}\widetilde d_{#1}}}
\def\scalarTrackingErrorSq#1{{\varitys{magenta!70}\widetilde d_{#1}^2}}
\def\scalarTrackingErrorThis{\scalarTrackingError{k}}
\def\scalarTrackingErrorThisSq{\scalarTrackingErrorSq{k}}

\def\infmap#1{\mathop{\mathrm{Inf}}\nolimits_{#1}}
\def\supmap#1{\mathop{\mathrm{Sup}}\nolimits_{#1}}
\def\Sopt{S^\sharp}
\def\Spess{S^\flat}

\def\frechetSubdiff{\partial_F}
\def\limitingSubdiff{\partial_M}
\def\anySubdiff{\breve\partial}
\def\lsubdiff#1{\subdiff_{|#1}}

%% file: symbols-texonly.tex
\renewrobustcmd{\downto}{{{\mathchoice%
            {\rotatebox[origin=c]{-20}{$\to$}}
            {\rotatebox[origin=c]{-20}{$\to$}}
            {\rotatebox[origin=c]{-20}{\scalebox{0.75}{$\to$}}}
            {\rotatebox[origin=c]{-20}{\scalebox{0.6}{$\to$}}}
}}}

\renewrobustcmd{\upto}{{{\mathchoice%
            {\rotatebox[origin=c]{20}{$\to$}}
            {\rotatebox[origin=c]{20}{$\to$}}
            {\rotatebox[origin=c]{20}{\scalebox{0.75}{$\to$}}}
            {\rotatebox[origin=c]{20}{\scalebox{0.6}{$\to$}}}
}}}

%% file: Tikz_image1.tex
	\begin{subfigure}[b]{0.49\textwidth}
	\centering
	\begin{tikzpicture}[
		declare function={
			func(\x)=  and(\x >= -0, \x < 5) * (5-abs(\x))     +
			(\x >= 5) * (0);
		}
		]
		\begin{axis}[
			axis x line=middle, axis y line=middle,
			ymin=-2, ymax=7, ytick={-1,...,6}, ylabel=$S_u(x)$,
			xmin=-2, xmax=8, xtick={-1,...,7}, xlabel=$x$,
			domain=-0:7,samples=101, 
			]
			\addplot [blue,line width=2pt] {func(x)};
            \addplot[->,line width=1pt, samples=3, smooth,domain=-5:7,black] coordinates {(6,0)(6,2)};
            \fill[color=black] (axis cs:6.0, 2.0) circle(0.1pt) node [above] {$(x^*_2,-J'(S_u(x_2)))$};
            \addplot[<-,line width=1pt, samples=3, smooth,domain=-1:1,black] coordinates {(-1,2)(1,4)};
            \fill[color=black] (axis cs:-1.94, 2.0) circle(0.01pt) node [below right] {$(x^*_1,-J'(S_u(x_1)))$};
		\end{axis}
	\end{tikzpicture}
    \caption{solution mapping $S_u$}
\end{subfigure}
\begin{subfigure}[b]{0.49\textwidth}
	\centering
	\begin{tikzpicture}[
        declare function={
            func(\x)= and(\x >= -0, \x <= 5) * (3 -\x) +
            (\x > 5) * (-2);
        }
        ]
	\begin{axis}[
        axis x line=middle, axis y line=middle,
        ymin=-3, ymax=4, ytick={-2,...,3}, ylabel=$J'(S_u(x))$,
        xmin=-1, xmax=8, xtick={-1,...,7}, xlabel=$x$,
        domain=-0:5,samples=101, 
        ]
        \addplot [magenta,line width=2pt] {func(x)};
        \addplot[line width=2pt, samples=10, smooth,domain=0:7,magenta] coordinates {(4.97,-2)(7,-2)};
    \end{axis}
	\end{tikzpicture}
    \caption{"direction" for differential $D^* S_u(x|S_u(x))$}
\end{subfigure}
\begin{subfigure}[b]{0.49\textwidth}
	\centering
	\begin{tikzpicture}[
		declare function={
			func(\x)=  
			and(\x >= -0, \x < 5) * 1/2*(\x-3)*(\x-3) +
			(\x >= 5) * (2);
		}
		]
		\begin{axis}[
			axis x line=middle, axis y line=middle,
			ymin=-2, ymax=7, ytick={-1,...,6},
            ylabel=$E(x)$,
			xmin=-2, xmax=8, xtick={-1,...,7}, xlabel=$x$,
			domain=-0:7,samples=101, 
			]
			\addplot [blue,line width=2pt] {func(x)};
            \addplot[<-,line width=1pt, samples=3, smooth,domain=-1:1,black] coordinates {(-1,1)(1,2)};
            \addplot[->,line width=1pt, samples=3, smooth,domain=-5:7,black] coordinates {(6,2)(6,1)};
            \fill[color=black] (axis cs:6.0, 1.0) circle(0.1pt) node [right] {$(x^*_2,-1)$};
            \fill[color=black] (axis cs:-1, 1.0) circle(0.1pt) node [below] {$(x^*_1,-1)$};
		\end{axis}
	\end{tikzpicture}
    \caption{outer objective $E=J\circ S_u$}
\end{subfigure}
\hspace{0.15cm}
\begin{subfigure}[b]{0.49\textwidth}
	\centering
	\begin{tikzpicture}[
		declare function={
			func(\x)= and(\x >= -0, \x <= 5) * (\x-3) +
			(\x > 5) * (0);
		}
		]
		\begin{axis}[
			axis x line=middle, axis y line=middle,
			ymin=-4, ymax=3, ytick={-3,...,2}, ylabel=$D^* E(x|E(x))(1)$,
			xmin=-1, xmax=8, xtick={-1,...,7}, xlabel=$x$,
			domain=-0:5,samples=101, 
			]
			\addplot [magenta,line width=2pt] {func(x)};
			\addplot[line width=2pt, samples=10, smooth,domain=0:7,magenta] coordinates {(5,0)(7,0)};
            \fill[color=magenta] (axis cs:5.0, 0.0) circle(2.5pt);
            \fill[color=magenta] (axis cs:5.0, 2.0) circle(2.5pt);
            \fill[color=black] (axis cs:1.0, -2.0) circle(2.0pt) node [below right] {$(x_1,x^*_1)$};
            \fill[color=black] (axis cs:6.0, 0) circle(2.0pt) node [above] {$(x_2,x^*_2)$};
		\end{axis}
	\end{tikzpicture}
    \caption{a differential for the outer objective}
\end{subfigure}

%% file: Tikz_image3.tex
\begin{subfigure}[b]{0.49\textwidth}
    \centering
    \begin{tikzpicture}
        \begin{axis}[%
                width=\linewidth,
                height=0.75\linewidth,
                xlabel near ticks,
                ylabel near ticks,
                scaled y ticks=false,
                yminorticks=true,
                minor y tick num=1,
                xminorticks=true,
                minor x tick num=3,
                axis x line*=bottom,
                axis y line*=left,
                xlabel={$x$},
                ylabel={$J\circ S_u(x)$},
                y tick label style={/pgf/number format/fixed},
                outer sep=0pt,
                font=\footnotesize,
            ]
            \addplot [color=Set2-C, line width=1pt] table[x=x,y=Fvalue]{results/denoise/bilevelFvalues.txt};
            \fill[color=red] (axis cs:2.0, 48.0) circle(1.5pt) node [above] {$\tilde x$};
        \end{axis}
    \end{tikzpicture}
    \caption{function value}
    \label{fig:denoising:values}
\end{subfigure}
\hfill%
\begin{subfigure}[b]{0.49\textwidth}
    \centering
    \begin{tikzpicture}
        \begin{axis}[%
                width=\linewidth,
                height=0.75\linewidth,
                ymode=log,
                xlabel near ticks,
                ylabel near ticks,
                scaled y ticks=false,
                yminorticks=true,
                minor y tick num=1,
                xminorticks=true,
                minor x tick num=3,
                axis x line*=bottom,
                axis y line*=left,
                xlabel={$x$},
                ylabel={$\COsubdiffS_2 F(x)$},
                outer sep=0pt,
                font=\footnotesize,
            ]

            \addplot[
                color=Set2-B,
                line width=1pt,
            ]
            table[
                    x=x,
                    y expr=\thisrow{Fdifferential} > 0 ? \thisrow{Fdifferential} : nan
                ]{results/denoise/bilevelFdifferentials.txt};

            \addplot[
                color=Set2-C,dashed,
                line width=1pt,
            ]
            table[
                    x=x,
                    y expr=\thisrow{Fdifferential} < 0 ? -\thisrow{Fdifferential} : nan
                ]{results/denoise/bilevelFdifferentials.txt};

        \end{axis}
    \end{tikzpicture}
    \caption{estimate of differential}
    \label{fig:denoising:diffest}
\end{subfigure}

%% file: Tikz_image_plots.tex
\pgfplotslegendfromname{leg:performance}\\
\medskip
\begin{subfigure}[b]{0.45\textwidth}
    \begin{tikzpicture}
        \begin{axis}[%
            width=\linewidth,
            height=0.75\linewidth,
            ymode=log,
            xlabel near ticks,
            ylabel near ticks,
            scaled y ticks=false,
            yminorticks=true,
            minor y tick num=1,
            xminorticks=true,
            minor x tick num=3,
            axis x line*=bottom,
            axis y line*=left,
            legend columns=4,
            legend style={draw=none,font=\small},
            xlabel={computational resources},
            ylabel={$\norm{\widetilde x^*_{k+1}}$},
            y tick label style={/pgf/number format/fixed},
            outer sep=0pt,
            font=\footnotesize,
            ]

            \addplot [color=Set2-C, line width=1pt] table[x=cputime,y=xstarNorm]{results/denoise/bilevelDenoiseNonsmoothImplicit.txt};

            \addplot [color=Set2-B, line width=1pt] table[x=cputime,y=xstarNorm]{results/denoise/bilevelDenoiseNonsmoothSingle.txt};
        \end{axis}
    \end{tikzpicture}
    \caption{denoise}
    \label{fig:numerical:denoise-xstar-norm}
\end{subfigure}%
\begin{subfigure}[b]{0.45\textwidth}
    \begin{tikzpicture}
        \begin{axis}[%
            width=\linewidth,
            height=0.75\linewidth,
            ymode=log,
            xlabel near ticks,
            ylabel near ticks,
            scaled y ticks=false,
            yminorticks=true,
            minor y tick num=1,
            xminorticks=true,
            minor x tick num=3,
            axis x line*=bottom,
            axis y line*=left,
            legend columns=4,
            xlabel={computational resources},
            ylabel={$\norm{\widetilde x^*_{k+1} + \grad R(\thisx)}$},
            y tick label style={/pgf/number format/fixed},
            outer sep=0pt,
            font=\footnotesize,
            legend columns=4,
            legend style={
                draw=none,
                font=\small,
                /tikz/column 2/.style={column sep=1em,},
                /tikz/column 4/.style={column sep=1em,},
                /tikz/column 6/.style={column sep=1em,},
            },
            legend to name = {leg:performance},
            ]

            \addplot [color=Set2-C, line width=1pt] table[x=cputime,y=xstarNorm]{results/deblur/bilevelDeblurNonsmoothImplicit.txt};
            \addlegendentry{implicit}

            \addplot [color=Set2-B, line width=1pt] table[x=cputime,y=xstarNorm]{results/deblur/bilevelDeblurNonsmoothSingle.txt};
            \addlegendentry{single}

            \end{axis}
    \end{tikzpicture}
    \caption{deblur}
    \label{fig:numerical:deblur-xstar-norm}
\end{subfigure}%
\\%
\begin{subfigure}[b]{0.45\textwidth}
    \begin{tikzpicture}
        \begin{axis}[%
            width=\linewidth,
            height=0.75\linewidth,
            xlabel near ticks,
            ylabel near ticks,
            scaled y ticks=false,
            yminorticks=true,
            minor y tick num=1,
            xminorticks=true,
            minor x tick num=3,
            axis x line*=bottom,
            axis y line*=left,
            xlabel={computational resources},
            ylabel={$J\circ S_u + R$},
            y tick label style={/pgf/number format/fixed},
            outer sep=0pt,
            font=\footnotesize,
            ]

            \addplot [color=Set2-C, line width=1pt] table[x=cputime,y=JplusR]{results/denoise/bilevelDenoiseNonsmoothImplicit.txt};

            \addplot [color=Set2-B, line width=1pt] table[x=cputime,y=JplusR]{results/denoise/bilevelDenoiseNonsmoothSingle.txt};

        \end{axis}
    \end{tikzpicture}
    \caption{denoise}
    \label{fig:numerical:outer-value-denoise}
\end{subfigure}%
\begin{subfigure}[b]{0.45\textwidth}
    \begin{tikzpicture}
        \begin{axis}[%
            width=\linewidth,
            height=0.75\linewidth,
            xlabel near ticks,
            ylabel near ticks,
            scaled y ticks=false,
            yminorticks=true,
            minor y tick num=1,
            xminorticks=true,
            minor x tick num=3,
            axis x line*=bottom,
            axis y line*=left,
            xlabel={computational resources},
            ylabel={$J\circ S_u + R$},
            y tick label style={/pgf/number format/fixed},
            outer sep=0pt,
            font=\footnotesize,
            ]

            \addplot [color=Set2-C, line width=1pt] table[x=cputime,y=JplusR]{results/deblur/bilevelDeblurNonsmoothImplicit.txt};

            \addplot [color=Set2-B, line width=1pt] table[x=cputime,y=JplusR]{results/deblur/bilevelDeblurNonsmoothSingle.txt};

        \end{axis}
    \end{tikzpicture}
    \caption{deblur}
    \label{fig:numerical:outer-value-deblur}
\end{subfigure}%

%% file: inner-lemmas.tex

With the goal of verifying the inner tracking property for standard inner algorithms, subject to metric subregularity, we prove here the linear convergence of forward-backward type methods to the entire set of minimisers of a convex problem.
Our approach includes primal-dual methods, and relaxes the standard assumption of strong convexity.
Similar results first appeared in \cite{tuomov-partial-subreg}, however, focussing on basic primal-dual methods, they did not include forward steps.
Our proof here adapts those in \cite[Theorem 29.8 and Lemmas 29.6 and 29.7]{clason2020introduction} to (a) treat primal-dual methods, (b) include the residual $\norm{\nexxt u-\this u}^2$ in the claim, and (c) to remove the squares.


We consider here, in a Hilbert space $U$, general forward-backward type methods that solve optimality conditions of the form
\[
    0 \in T(\nextu) \defeq \subdiff g(\nextu) + \grad f(\nextu) + \Xi\nextu,
\]
and algorithms of the form
\begin{equation}
    \label{eq:subreg:fb}
    0 \in \subdiff g(\nextu) + \grad f(\thisu) + \Xi\nextu + M(\nextu-\thisu),
\end{equation}
where $\Xi \in \linear(U; U)$ is skew-adjoint; $M \in \linear(U; U)$ is self-adjoint and positive definite; $g: U \to \extR$ is convex, proper, and lower semicontinous; and $f: U \to \extR$ is convex and Fréchet differentiable with an $L_M$-Lipschitz gradient with respect to the $M$ and $\inv M$-norms: $\norm{\grad f(u)-\grad f(\tilde u)}_{\inv M} \le L_M\norm{u-\tilde u}_M$.
For the basic forward-backward method, $M=\inv\tau\Id$. For the PDPS, see \cref{ex:inner-pdps}.

\begin{lemma}
    \label{lemma:subreg:error-bound-first-estimate-fb}
    For $\nextu$ generated by  \eqref{eq:subreg:fb} from $\thisu$,
    we have
    $
        2(1+L_M^2)\norm{\nextu-\thisu}_M^2
        \ge
        \dist_{\inv M}^2(0, T(\nextu)).
    $
\end{lemma}

\begin{proof}
    Using \eqref{eq:subreg:fb}, we have
    \begin{equation}
        \label{eq:subreg:error-bound-first-estimate}
        \frac{1}{2}\norm{\nextu-\thisu}_M^2
        =
        \frac{1}{2}\dist_{\inv M}^2(0, \{-M(\nextu-\thisu)\})
        \ge
        \frac{1}{2}\dist_{\inv M}^2(0, \subdiff g(\nextu)+\grad f(\thisu)+\Xi\nextu).
    \end{equation}
    Young's inequality for any $\alpha \in (0, 1)$ then yields
    \begin{multline*}
        \frac{1}{2}\dist_{\inv M}^2(0, \subdiff g(\nextu)+\grad f(\thisu)+\Xi\nextu)
        \\
        \begin{aligned}
             &
            =
            \frac{1}{2}\dist_{\inv M}^2(\grad f(\nextu)-\grad f(\thisu), \subdiff g(\nextu)+\grad f(\nextu)+\Xi\nextu)
            \\
             &
            =
            \inf_{q \in \subdiff g(\nextu)}
            \frac{1}{2}\norm{(\grad f(\nextu)-\grad f(\thisu))-(q+\grad f(\nextu)+\Xi\nextu)}_{\inv M}^2
            \\
             &
            \ge
            \frac{1-\inv\alpha}{2}\norm{\grad f(\nextu)-\grad f(\thisu)}_{\inv M}^2+\inf_{q \in \subdiff g(\nextu)} \frac{1-\alpha}{2}\norm{q+\grad f(\nextu)+\Xi\nextu}_{\inv M}^2
            \\
             &
            \ge
            \frac{(1-\inv\alpha)L_M^2}{2}\norm{\nextu-\thisu}_{M}^2 + \frac{1-\alpha}{2}\dist_{\inv M}^2(0, T(\nextu)),
        \end{aligned}
    \end{multline*}
    where we have used in the last step that $\alpha \in (0, 1)$ and that $\grad f$ is Lipschitz continuous between the $\inv M$ and $M$-norms.
    Combining this estimate with \eqref{eq:subreg:error-bound-first-estimate} and taking $\alpha=1/2$, we obtain that
    \[
        \frac{1+L_M^2}{2}\norm{\nextu-\thisu}_{M}^2 \ge \frac{1}{4}\dist_{\inv M}^2(0, T(\nextu)).
        \qedhere
    \]
\end{proof}

\begin{theorem}
    \label{thm:subreg:convergence-result-sub-peb:general}
    With the above setup, suppose $0 < \epsilon \le 1-L_M/2$, and that $T$ is metrically subregular with respect to the $M$ and $\inv M$-norms at $\opt u$ for $\opt w=0$ with the modulus $\kappa_M>0$ and radius $\delta>0$, i.e.,
    \begin{equation}
        \label{eq:subreg:subregularity}
        \dist_M(u, \inv T(0)) \le \kappa_M \dist_{\inv M}(0, T(u))
        \quad (u \in \B_M(\optu, \delta)).
    \end{equation}
    Take $\thisu \in \B_M(\opt u, \delta)$ for a $\optu \in \inv T(0)$, and generate $\nextu$ through the satisfaction of \eqref{eq:subreg:fb}.
    Then $\nextu \in \B_M(\opt u, \delta)$, and, for $\rho=(1-L_M/2 - \epsilon)/(2(1+L_M^2)\kappa_M^2)>0$ and any $\alpha>0$, we have
    \begin{equation}
        \label{eq:subreg:contractivity}
        2\sqrt{\frac{1+\rho}{2+\alpha}}\dist_M(\nextu, \inv T(0))
        +
        2\sqrt{\frac{\epsilon}{2+\inv\alpha}}\norm{\nextu-\thisu}_M
        \le
        \dist_M(\thisu, \inv T(0)).
    \end{equation}
\end{theorem}

\begin{proof}
    \textbf{Step 1:}
    We first show that  $\nextu \in \B_M(\opt u, \delta)$ using standard Féjer monotonicity.
    In fact, take any element $\bar u \in \inv T(0)$.
    Then for some $\bar w \in \subdiff g(\bar u)$ we have $0=\bar w + \grad f(\bar u) + \Xi\bar u$.
    Also let $\nexxt{w} \in \subdiff g(\nextu)$ be the element that satisfies \eqref{eq:subreg:fb} holds.
    By \cite[Theorem 7.2]{clason2020introduction}, equipping $U$ with the norm $\norm{\freevar}_M$ and $U^*$ with norm $\norm{R\freevar}_{\inv M}$, where $R: U^* \to U$ is the Riesz map (for the original inner product of $U$, unrelated to $M$), the Lipschitz continuity of $\grad f$ implies the three-point monotonicity
    \[
        \iprod{\grad f(\thisu)-\grad f(\bar u)}{\nextu-\bar u}_U
        =
        \dualprod{f'(\thisu)-f'(\bar u)}{\nextu-\bar u} \ge -\frac{L_M}{4}\norm{\nextu-\thisu}_M^2.
    \]
    The first equality here holds because the special choice of norms does not affect the duality pairing, hence not the Riesz representation using the original inner product of $U$.
    Combining this with the monotonicity of $\subdiff g$, and the fact that by skew-symmetricity $\iprod{\Xi u}{u}=0$ for all $u$, we obtain
    \begin{equation}
        \label{eq:subreg:mono}
        \iprod{\nexxt{w}+\grad f(\thisu)+\Xi\nextu}{\nextu-\bar u}
        =
        \iprod{\nexxt{w}-\bar w+\grad f(\thisu)- \grad f(\bar u)}{\nextu-\bar u}
        \ge -\frac{L_M}{4}\norm{\nextu-\thisu}_M^2.
    \end{equation}
    Applying \eqref{eq:subreg:fb}, this gives
    $
        -\iprod{M(\nextu-\thisu)}{\nextu-\bar u}
        \ge -\frac{L_M}{4}\norm{\nextu-\thisu}_M^2.
    $
    Pythagoras' identity then establishes
    \begin{equation}
        \label{eq:subreg:fejer-extra}
        \frac{1}{2}\norm{\thisu-\bar u}_M^2
        \ge
        \frac{1-L_M/2}{2}\norm{\nextu-\thisu}_M^2 + \frac{1}{2}\norm{\nextu-\bar u}_M^2.
    \end{equation}
    Since $L_M \le 2$ by assumption, we obtain the Féjer monotonicity $\norm{\nextu-\bar u}_M \le \norm{\thisu-\bar u}_M$.
    Taking $\bar u = \opt u$, our assumption $\thisu \in \B_M(\opt u, \delta)$, thus, implies $\nextu \in \B_M(\opt u, \delta)$.

    \textbf{Step 2:}
    We then prove that
    \begin{equation}
        \label{eq:subreg:convergence-result-sub-peb:squared}
        \frac{1}{2}\dist_{M}^2(\thisu, \inv T(0))
        \ge
        \frac{1+\rho}{2}\dist_{M}^2(\nextu, \inv T(0))
        +
        \frac{\epsilon}{2}\norm{\nextu-\thisu}_M^2.
    \end{equation}
    Indeed, combining \cref{lemma:subreg:error-bound-first-estimate-fb} and the definition of metric subregularity yields the \emph{error bound}
    \begin{equation*}
        2(1+L_M^2)\norm{\nextu-\thisu}_M^2
        \ge
        \dist_{\inv M}^2(0, T(\nextu))
        \ge
        \kappa_M^{-2}
        \dist_M^2(\nextu, \inv T(0)).
    \end{equation*}
    It follows for all $\bar u \in \inv T(0)$ that
    \begin{equation}
        \label{eq:subreg:growth}
        \begin{split}
            \frac{1-L_M/2-\epsilon}{2}\norm{\nextu-\thisu}_M^2
            + \frac{1}{2}\norm{\nextu-\bar u}_M^2
             &
            \ge
            \frac{\rho}{2}\dist_M^2(\nextu; \inv T(0))
            + \frac{1}{2}\norm{\nextu-\bar u}_M^2
            \\
             &
            \ge
            \frac{1+\rho}{2}\dist_M^2(\nextu; \inv T(0)).
        \end{split}
    \end{equation}
    Summing \cref{eq:subreg:fejer-extra,eq:subreg:growth} yields
    \[
        \frac{1}{2}\norm{\thisu-\bar u}_M^2
        \ge
        \frac{\epsilon}{2}\norm{\nextu-\thisu}_M^2
        + \frac{1+\rho}{2}\dist_M^2(\nextu; \inv T(0)).
    \]
    Taking the infimum over $\bar u \in \inv T(0)$, we obtain  \eqref{eq:subreg:convergence-result-sub-peb:squared}

    \textbf{Step 3:}
    We now remove the squares from \eqref{eq:subreg:convergence-result-sub-peb:squared}.
    By Young's inequality, for any $\alpha>0$ and $a,b \in \R$, we have
    \[
        \frac{1+\rho}{2}a^2
        +
        \frac{\epsilon}{2}b^2
        \ge
        \left(
        \sqrt{\frac{1+\rho}{2+\alpha}}a
        +
        \sqrt{\frac{\epsilon}{2+\inv\alpha}}b
        \right)^2.
    \]
    Taking $a=\dist_M^2(\nextu, \inv T(0))$ and $b=\norm{\nextu-\thisu}_M$, and combining with  \eqref{eq:subreg:convergence-result-sub-peb:squared} finishes the proof.
\end{proof}

The following lemmas verify the required metric subregularity subject to strong monotonicity on subspaces; see also \cite{tuomov-partial-subreg}.
We recall that we assume $M$ to be self-adjoint and positive definite.

\begin{lemma}
    \label{lemma:subreg:monotonicity-subregularity}
    Suppose $T$ is $\Gamma$-strongly monotone for a self-adjoint $\Gamma \in \linear(U; U)$ at a $\opt u \in \inv T(0)$ in the sense that, for some $\gamma>0$,
    \begin{equation}
        \label{eq:subreg:monotonicity}
        \norm{u-\opt u}_\Gamma^2 \le \iprod{u-\opt u}{q - \opt q}
        \quad
        \text{for all}
        \quad
        u \in U,\, q \in T(u),\, \opt q \in T(\opt u).
    \end{equation}
    Suppose, moreover, that $V \defeq \inv T(0) - \opt u$ is a closed subspace of $U$, and, that we have $\gamma M \le \Gamma$ on $V_M^\perp$ (i.e., $\iprod{v^\perp}{[\Gamma - \gamma M]v^\perp} \ge 0$ for all $v^\perp \in V_M^\perp$) for some $\gamma > 0$ and
    $
        V_M^\perp \defeq \{ v^\perp \in U \mid \iprod{Mv}{v^\perp} = 0 \text{ for all } v \in V \}.
    $
    Then the metric subregularity \eqref{eq:subreg:subregularity} holds with $\delta=\infty$ and $\kappa_M=\inv\gamma$.
\end{lemma}

\begin{proof}
    We can decompose\footnote{Solving the problem $\min_{v \in V} \norm{u-v}_M^2$ gives $u= v + \inv M z$ for some $z \in N_V(v)$. Since $V$ is a closed subspace, $\inv M z \in V_M^\perp$.}
    any $u \in U$ as $u=u_V + u^\perp \in V + V_M^\perp$.
    Hence, for any $q \in T(u)$,
    \[
        \begin{split}
        \gamma \dist_M(u, \inv T(0))^2
        &
        =
        \inf_{\bar u \in \inv T(0)} \gamma\norm{u-\bar u}_M^2
        =
        \inf_{v \in V} \gamma\norm{v+\opt u-\bar u}_M^2
        =
        \gamma \norm{u^\perp -\opt u^\perp}_M^2
        \le
        \gamma \norm{u^\perp -\opt u^\perp}_\Gamma^2
        \\
        &
        =
        \gamma \norm{u -\opt u}_\Gamma^2
        \le
        \iprod{u-\opt u}{q - 0}
        \le
        \norm{u-\opt u}_M \norm{q}_{\inv M}
        =
        \dist_M(u, \inv T(0))\norm{q}_{\inv M}.
        \end{split}
    \]
    Taking the infimum over $q \in T(u)$, and dividing by $\dist_M(u, \inv T(0))$, we obtain \eqref{eq:subreg:subregularity} with $\kappa_u=\inv\gamma$.
\end{proof}

\begin{lemma}
    \label{lemma:subreg:monotonicity-subregularity:basic}
    Suppose $T$ is $\Gamma=\gamma M$-strongly monotone in the sense \eqref{eq:subreg:monotonicity}.
    Then the metric subregularity \eqref{eq:subreg:subregularity} holds with $\delta=\infty$ and $\kappa_M=\inv\gamma$.
\end{lemma}

\begin{proof}
    We apply \cref{lemma:subreg:monotonicity-subregularity}, where now $V=\{0\}$, because \eqref{eq:subreg:monotonicity} implies $\inv T(0)=\optu$ through the positive definiteness of $M$.
\end{proof}

\printProofs[inner-fb]

\begin{corollary}[Nonconvex $g$]
    \label{cor:subreg:nonconvex-g}
    Suppose that $g$ is globally nonconvex, however, convex within $B(\optu, \delta')$ for some $\delta'$.
    For a convex $f: U \to \R$ with an $L$-Lipschitz gradient, let
    \[
        T(u)
        =
        \begin{cases}
            \partial[g + f +\delta_{B(\optu, \delta')}](u), & u \in \interior B(\optu, \delta'), \\
            \emptyset, & \text{otherwise}.
        \end{cases}
    \]
    Assume that $\optu \in \inv T(0)$, and that $T$ is metrically subregular at $\optu$ with factor $\kappa>0$ and radius $\delta \in (0, \delta')$.
    Also assume the lower bound $\inf(g+f) \ge -C$ for some $C \in \R$, and take $\tau>0$ small enough that $\tau L <1$ as well as $2\tau([g+f](u^0) + C) < (1-\tau L)(\delta'-\delta)^2$.
    Given $\thisu \in B(\optu, \delta)$, update
    $
        \nextu = \argmin_u g(u) + \iprod{\grad f(u)}{u-\thisu} + \frac{1}{2\tau}\norm{u-\thisu}^2.
    $
    Then $\nextu \in B(\optu, \delta)$ and \eqref{eq:ubreg:convergence-result-sub-peb:fb} holds with $\epsilon=1/2$.
\end{corollary}

The minimisation-form update here is, in the convex case, equivalent to the subdifferential form \eqref{eq:subreg:fb} with $\Xi=0$ and $M=\inv\tau\Id$.

\begin{proof}
    The update gives
    $
        g(\nextu) + \iprod{\grad f(\thisu)}{\nextu-\thisu} + \frac{1}{2\tau}\norm{\nextu-\thisu}^2 \le g(\thisu).
    $
    Using here the descent inequality for $f$, equivalent to the $L$-Lipschitz gradient (see, e.g., \cite[Theorem 7.1]{clason2020introduction}), we obtain
    $
        [g+f](\nextu) +  \frac{1-\tau L}{2\tau}\norm{\nextu-\thisu}^2 \le [g+f](\thisu).
    $
    Thus $\{[g+f](\thisu)\}_{k \in \N}$ is non-increasing, and, also using our step length bound, $\norm{\nextu-\thisu}^2 \le \frac{2\tau}{1-\tau L}([g+f](u^0) + C) < (\delta'-\delta)^2$.
    Since $\thisu \in B(\optu, \delta)$, it follows that $\nextu \in \interior B(\optu, \delta')$.
    We also have $L\tau \le 2(1-\epsilon)$ for $\epsilon=1/2$.
    Now, since all our points of interest, $\optu, \thisu, \nextu$ are in the interior of the region of convexity $B(\optu, \delta')$ “a priori”, the proof of  \cref{thm:subreg:convergence-result-sub-peb:fb} goes through “a posteriori”.
\end{proof}

\printProofs[inner-pdps]

%% file: appendix-tracking.tex

\begin{assumption}[{\cite[Assumption A.1]{dizonvalkonen2024tracking}}]
    \label{ass:scalar-tracking:main}
    For a given $k \ge 0$ and scalars $\distU{0}, \ldots, \distU{k+1}, \thisDistW[0], \ldots,  \thisDistW[k+1], \distX{1}, \ldots, \distX{k} \ge 0$, and $\scalarTrackingError{0}, \ldots, \scalarTrackingError{k} \in \R$, there exist $\pi_u, \pi_w, \primaldifffact \alpha_w,\alpha_u>0$, such that
    \begin{align}
        \label{item:scalar-tracking:main:inner-tracking}
        \tag{i}
        \kappa_u \distU{j+1}
        &
        \le
        \distU{j}
        + \pi_u\distX{j},
        &&
        \quad\text{for all}\quad j=1,\ldots,k,
        \\
        \label{item:scalar-tracking:main:adjoint-tracking}
        \tag{ii}
        \kappa_w \distW{j+1}
        &
        \le
        \thisDistW[j]
        + {\primaldifffact} \nextDistU
        + \pi_w\distX{j},
        &&
        \quad\text{for all}\quad j=1,\ldots,k,
        \quad\text{and}
        \\
        \label{item:scalar-tracking:main:differential-transformation}
        \tag{iii}
        \scalarTrackingError{j}
        &
        \le
        \alpha_u \distU{j+1}
        + \alpha_w \distW{j+1}
        &&
        \quad\text{for all}\quad j=0,\ldots,k.
    \end{align}
\end{assumption}

To simplify the following results, we introduce $\kappa\defeq \min\{\kappa_u, \kappa_w\}$ and  $\overline\kappa \defeq \max\{\kappa_u, \kappa_w\}$.

\begin{lemma}[{\cite[Lemma A.3]{dizonvalkonen2024tracking}}]
    \label{lemma:scalar-tracking:inner-product-error-estimate}
    Suppose \cref{ass:scalar-tracking:main} holds for a $k \ge 0$.
    Then for any $p \in (0, \kappa)$, we have
    \begin{equation}
        \label{eq:scalar-tracking:inner-product-error-estimate}
        \scalarTrackingErrorThisSq
        \le
        (\alpha_u d^u_{k+1} + \alpha_w d^w_{k+1})^2
        \le
        \breve e_{p,k},
    \end{equation}
    where, for
    $
        \trackingres_j
        \defeq \alpha_u\kappa_u^{-j}\pi_u + \alpha_w[\iota_j{\primaldifffact}\pi_u + \kappa_w^{-j}\pi_w]
    $
    and, we set
        {\allowdisplaybreaks
            \begin{align}
                \label{eq:scalar-tracking:ressum}
                \trackingressum
                 &
                \defeq
                \frac{\overline\kappa}{p}
                \sum_{j=0}^\infty p^j \trackingres_j
                \le
                \frac{(\alpha_u\pi_u+\alpha_w\pi_w)\kappa\overline\kappa}{p(\kappa-p)}
                + \frac{\alpha_w {\primaldifffact}\pi_u\overline\kappa}{p^2(\kappa-p)^2}
                \quad\text{and}
                \\
                \label{eq:scalar-tracking:ek:0}
                \breve e_{p,k}
                 &
                \defeq
                \frac{\trackingressum(\alpha_u\kappa_u^{-k} + \alpha_w \iota_k {\primaldifffact})}{\pi_u p^k}\initDistUsq
                +
                \frac{\trackingressum\alpha_w\kappa_w^{-k}}{\pi_w p^k}\initDistWsq
                +
                \sum_{j=0}^{k-1} \frac{\trackingressum\trackingres_{k-j}}{ p^{k-j}} \nextDistXthis[j].
            \end{align}
        }
        The second inequality of \cref{eq:scalar-tracking:inner-product-error-estimate} holds even if \cref{ass:scalar-tracking:main}\,\cref{item:scalar-tracking:main:differential-transformation} does not.
\end{lemma}

\begin{lemma}[{\cite[Lemma A.4]{dizonvalkonen2024tracking}}]
    \label{lemma:scalar-tracking:error-sum}
    Suppose \cref{ass:scalar-tracking:main} holds for a $k \ge 0$.
    Then for any $p \in (0, \kappa)$, we have
    \begin{equation}
        \label{eq:scalar-tracking:inner-product-error-estimate:x}
        \scalarTrackingErrorThisSq
        \le
        \trackingressum^2 \nextDistXthisSq
        +
        e_{p,k},
    \end{equation}
    where, for $\breve e_{p,k}$ defined \eqref{eq:scalar-tracking:ek:0},
    \begin{equation}
        \label{eq:scalar-tracking:ek}
        e_{p,k} \defeq \breve e_{p,k} - \trackingressum^2 \nextDistXthisSq
    \end{equation}
    satisfies
    \begin{equation}
        \label{eq:scalar-tracking:ressum:sm}
        \begin{split}
            \sup_{N \in \N}
            \sum_{k=0}^{N-1} p^k e_{p,k}
             &
            \le
            \Psi_p
            \defeq
            \frac{\initDistUsq}{\pi_u} \bigg(\frac{\trackingressum\alpha_u\kappa}{\kappa-1} + \frac{\trackingressum\alpha_w{\primaldifffact}}{(\kappa-1)^2}\bigg)
            +
            \frac{\initDistWsq}{\pi_w} \bigg(\frac{\trackingressum\alpha_w\kappa}{\kappa-1}\bigg).
        \end{split}
    \end{equation}
\end{lemma}

\begin{lemma}[{\cite[Lemma A.5]{dizonvalkonen2024tracking}}]
    \label{lemma:scalar-tracking:error-one}
    Suppose \cref{ass:scalar-tracking:main} holds for a $k \in \N$.
    Then, for $\breve e_{1,n}$ given in \eqref{eq:scalar-tracking:ek:0} we have $\sum_{n=0}^{k-1} \breve e_{1,n} \le \Psi_1 + \varsigma_1^2\sum_{n=0}^{k-1}\nextDistXthisSq[n] .$
\end{lemma}